\documentclass{article}
\usepackage[utf8]{inputenc}

\usepackage{authblk}
\usepackage{hyperref}
    
\title{Solving high-dimensional Hamilton-Jacobi-Bellman PDEs using \\neural networks: perspectives from the theory of controlled diffusions and measures on path space}
\author[1]{Nikolas N\"usken}
\author[2,3]{Lorenz Richter}
\date{June 21, 2021}

\affil[1]{Institute of Mathematics, Universit\"at Potsdam, 14476 Potsdam, Germany, \href{mailto:nuesken@uni-potsdam.de}{nuesken@uni-potsdam.de}}
\affil[2]{Institute of Mathematics, Freie Universit\"at Berlin, 14195 Berlin, Germany, \href{mailto:lorenz.richter@fu-berlin.de}{lorenz.richter@fu-berlin.de}}
\affil[3]{Institute of Mathematics, Brandenburgische Technische Universität Cottbus-Senftenberg, 03046 Cottbus, Germany}

\usepackage{mathrsfs, amssymb}
\usepackage{amsthm}
\usepackage[intlimits]{empheq}
\usepackage{bbm}     \usepackage{amsmath}
\usepackage{xcolor}
\usepackage[scientific-notation=true]{siunitx}
\usepackage{float}
\usepackage{dsfont}	
\usepackage[numbers]{natbib}
\usepackage[nottoc,numbib]{tocbibind}
\usepackage{bbm}
\usepackage{subcaption}
\usepackage[inner=2cm,outer=2cm,top=2cm,bottom=2cm]{geometry}
\usepackage[ruled,vlined]{algorithm2e}
\usepackage{enumitem}

\let\P\relax\newcommand{\P}{\mathbb{P}}
\newcommand*{\Q}{\mathbb{Q}}
\newcommand*{\U}{\mathbb{U}}
\DeclareMathOperator*{\E}{\mathbb{E}}
\DeclareMathOperator*{\R}{\mathbb{R}}
\DeclareMathOperator*{\KL}{\operatorname{KL}}
\DeclareMathOperator*{\Var}{\operatorname{Var}}

\DeclareMathOperator*{\RE}{\mathrm{RE}}
\DeclareMathOperator*{\CE}{\mathrm{CE}}
\DeclareMathOperator*{\Tr}{\mathrm{Tr}}

\theoremstyle{plain}
\newtheorem{theorem}{Theorem}[section]
\newtheorem{lemma}[theorem]{Lemma}

\newtheorem{proposition}[theorem]{Proposition}
\newtheorem{assumption}{Assumption}
\newtheorem{problem}{Problem}[section]

\theoremstyle{definition}
\newtheorem{definition}[theorem]{Definition}
\newtheorem{example}[theorem]{Example}

\theoremstyle{remark}
\newtheorem{remark}[theorem]{Remark}

\usepackage[T1]{fontenc}
\usepackage{lipsum}
\usepackage{tocloft}
\begin{document}
    
\maketitle
    
\begin{abstract}
Optimal control of diffusion processes is intimately connected to the problem of solving certain Hamilton-Jacobi-Bellman equations. Building on recent machine learning inspired approaches towards high-dimensional PDEs, we investigate the potential of \emph{iterative diffusion optimisation} techniques, in particular considering applications in importance sampling and rare event simulation, and focusing on problems without diffusion control, with linearly controlled drift and running costs that depend quadratically on the control. More generally, our methods apply to nonlinear parabolic PDEs with a certain shift invariance. The choice of an appropriate loss function being a central element in the algorithmic design, we develop a principled framework based on divergences between path measures, encompassing various existing methods. Motivated by connections to forward-backward SDEs, we propose and study the novel \emph{log-variance} divergence, showing favourable properties of corresponding Monte Carlo estimators. The promise of the developed approach is exemplified by a range of high-dimensional and metastable numerical examples. 
\end{abstract}
\section{Introduction}
    
Hamilton-Jacobi-Bellman partial differential equations (HJB-PDEs) are of central importance in applied mathematics. Rooted in reformulations of classical mechanics \cite{goldstein2002classical} in the nineteenth century, they nowadays form the backbone of (stochastic) optimal control theory \cite{nisio2014stochastic,yong1999stochastic}, having a profound impact on neighbouring fields such as optimal transportation \cite{villani2003topics,villani2008optimal}, mean field games \cite{carmona2018probabilistic}, backward stochastic differential equations (BSDEs) \cite{carmona2016lectures} and large deviations \cite{feng2006large}. Applications in science and engineering abound; examples include  stochastic filtering and data assimilation \cite{mitter1996filtering,reich2019data}, the simulation of rare events in molecular dynamics \cite{hartmann2014characterization,hartmann2012efficient,zhang2014applications}, and nonconvex optimisation \cite{chaudhari2018deep}. Many of these applications involve HJB-PDEs in high-dimensional or even infinite-dimensional state spaces, posing a formidable challenge for their numerical treatment and in particular rendering grid-based schemes infeasible.
    
In recent years, approaches to approximating the solutions of high-dimensional elliptic and parabolic PDEs have been developed combining well-known Feynman-Kac formulae with machine learning methodologies, seeking scalability and robustness in high-dimensional and complex scenarios \cite{weinan2017deep, han2018solving}. Crucially, the use of artificial neural networks offers the promise of accurate and efficient function approximation which in conjunction with Monte Carlo methods might beat the \emph{curse of dimensionality}, as investigated in \cite{beck2020overcoming,cheridito2019efficient,grohs2019deep,hutzenthaler2019proof}.
    
In this paper, we focus on HJB-PDEs that can be linked to \emph{controlled diffusions} (see Section \ref{sec: connections}),
\begin{equation}
\label{eq:controlled diffusion}
\mathrm d X_s^u = \left(b(X_s^u,s) + \sigma(X_s^u,s) u(X_s^u,s)\right) \mathrm ds + \sigma(X^u_s,s) \, \mathrm dW_s, \qquad X_0^u = x_{\mathrm{init}},
\end{equation}
where $b$ and $\sigma$ are coefficients derived from the model at hand, and $u$ is to be thought of as an adaptable steering force to be chosen so as to minimise a given objective functional. In terms of the problems and applications alluded to in the first paragraph, we are particularly interested in situations where applying a suitable control $u$ improves certain properties of \eqref{eq:controlled diffusion}; often these are related to sampling efficiency, exploration of state space, or fit to empirical data. We have been particularly motivated by the prospect of directing recent advances in the methodology for solving high-dimensional HJB-PDEs towards the challenges of rare event simulation \cite{bucklew2013introduction}.
    
Our attention in this paper is constrained to a class of algorithms that may be termed \emph{iterative diffusion optimisation} (IDO) techniques, related in spirit to reinforcement learning \cite{powell2019reinforcement}. Speaking in broad terms, those are characterised by the following outline of steps meant to be executed iteratively until convergence or until a satisfactory control $u$ is found:
    
\begin{enumerate}
\item
\label{it:step 1}
Simulate $N$ realisations $\{(X_s^{u,(i)})_{0 \le s \le T}, \,\, i=1,\ldots,N\}$ of the solution to \eqref{eq:controlled diffusion}.    
\item
Compute a performance measure and a corresponding gradient associated to the control $u$, based on\\ ${\{(X_s^{u,(i)})_{0 \le s \le T}, \,\, i=1,\ldots,N\}}$.
\item
Modify $u$ according to the gradient obtained in the previous step. Repeat starting from \ref{it:step 1}.
\end{enumerate}
Many algorithmic approaches from the literature can be placed in the IDO framework, in particular some that connect forward-backward SDEs and machine learning \cite{weinan2017deep,han2018solving} as well as some that are rooted in molecular dynamics and optimal control \cite{hartmann2012efficient,kappen2012optimal,zhang2014applications}. Those instances of IDO mainly differ in terms of the performance measure employed in step 2, or, in other words, in terms of an underlying loss function $\mathcal{L}(u)$ constructed on the set of control vector fields. Typically, $\mathcal{L}(u)$ is given in terms of expectations 
involving the solution to \eqref{eq:controlled diffusion}. Consequently, step 1 can be thought of as providing an empirical estimate of this quantity (and its gradient) based on a sample of size $N$.
    
For a principled design and understanding of IDO-like algorithms, it is central to analyse the properties of loss functions and corresponding Monte Carlo estimators, and identify guidelines that promise good performance.       
Permissible loss functions include those that admit a global minimum representing the solution to the problem at hand. Moreover, suitable loss functions yield themselves to efficient optimisation procedures (step 3) such as stochastic gradient descent. In this respect, important desiderata are the absence of local minima as well as the availability of low-variance gradient estimators.

In this article, we show that a variety of loss functions can be constructed and analysed in terms of divergences between probability measures on the path space associated to solutions of \eqref{eq:controlled diffusion}, providing a unifying framework for IDO and extending on previous works in that direction \cite{hartmann2012efficient,kappen2012optimal,zhang2014applications}. As this perspective entails the approximation of a target probability measure as a core element, our approach exposes connections to the theory of variational inference \cite{blei2017variational,zhang2018advances}. Classical divergences include the relative entropy (or $\mathrm{KL}$-divergence) and its counterpart, the cross-entropy. Motivated by connections to forward-backward SDEs and importance sampling, we propose the novel family of \emph{log-variance} divergences,
\begin{equation}
\label{eq:log var intro}
D^{\mathrm{Var(log)}}_{\widetilde{\mathbb{P}}}(\mathbb{P}_1 \vert \mathbb{P}_2) = {\Var}_{\widetilde{\mathbb{P}}} \left( \log \frac{\mathrm{d}\mathbb{P}_2}{\mathrm{d}\mathbb{P}_1}\right),
\end{equation}
parametrised by a probability measure $\widetilde{\mathbb{P}}$. Loss functions based on these divergences can be viewed as modifications of those proposed in \cite{weinan2017deep,han2018solving} for solving forward-backward SDEs, essentially replacing second moments by variances, see Section \ref{sec:BSDEs}. Moreover, it turns out that the log-variance divergences are closely related to the $\KL$-divergence (see Proposition \ref{prop: equivalence moment log-variance}), allowing us to draw (perhaps surprising) connections to methods that directly attempt to optimise the dynamics with respect to a control objective.
    
As the loss functions considered in this article are defined in terms of expected values, practical implementations require appropriate Monte Carlo estimators whose variance directly impacts algorithmic performance. We study the associated relative errors, in particular in high-dimensional settings and for $\P_1 \approx \P_2$, i.e. close to the optimal control. The proposed log-variance divergence and its corresponding standard Monte Carlo estimator turn out to be robust in both settings, in a precise sense that will be developed in later sections. After the completion of this manuscript, the potential of the log-variance divergences for inferences in computational Bayesian statistics has been explored in \cite{richter2020vargrad}, along with a more careful analysis of their relations to control variates (see also Remark \ref{remark: control variate} below). 
    
\subsection{Our contributions and overview}

The primary contributions of this article can be summarised as follows:

\begin{enumerate}
    \item
    Building on earlier work connecting optimal control functionals and the $\mathrm{KL}$-divergence \cite{hartmann2012efficient,kappen2012optimal,zhang2014applications}, we develop the perspective of constructing loss functions via divergences on path space, offering a systematic approach to algorithmic design and analysis.
    \item
    We show that modifications of recently proposed approaches based on forward-backward SDEs \cite{weinan2017deep,han2018solving} can be placed within this framework. Indeed, the log-variance divergences \eqref{eq:log var intro} encapsulate a family of forward-backward SDE systems (see Section \ref{sec:BSDEs}). The aforementioned adjustments needed to establish the path space perspective often lead to faster convergence and more accurate approximation of the optimal control, as we show by means of numerical experiments.
    \item
    We show that certain instances of algorithms based on the control objective (or $\mathrm{KL}$-divergence) and forward-backward SDEs (or the log-variance divergences) are equivalent when the sample size $N$ in step 1 is large.
    \item
    We investigate the properties of sample based gradient estimators associated to the losses and divergences under consideration. In particular, we define two notions of stability: robustness of a divergence under tensorisation (related to stability in high-dimensional settings) and robustness at the optimal control solution (related to stability of the final approximation). From the losses and divergences considered in this article, we show that only the log-variance divergences satisfy both desiderata and illustrate our findings by means of extensive numerical experiments. 
\end{enumerate}
    
The paper is structured as follows. In Section \ref{sec: connections} we provide a literature overview, stating connections between different perspectives on the control problem under consideration and summarising corresponding numerical treatments. As a unifying viewpoint, in Section \ref{sec: divergences section} we define viable loss functions through divergences on path space and discuss their connections to the algorithmic approaches encountered in Section \ref{sec: connections}. In particular, we elucidate the relationships of the log-variance divergences with forward-backward SDEs. In the two upcoming sections we analyse properties of the suggested losses, where in Section \ref{sec: infinite batch size} we obtain equivalence relations that hold in an infinite batch size limit and in Section \ref{sec:finite sample properties} we investigate the variances associated to the losses' estimator versions. In the latter case, we consider stability close to the optimal control solution as well as in high dimensional settings. In Section \ref{sec:numerics} we provide numerical examples that illustrate our findings. Finally, we conclude the paper with Section \ref{sec:outlook}, giving an outlook to future research. Most of the proofs are deferred to the appendix.
    
\section{Optimal control problems, change of path measures and Hamilton-Jacobi-Bellman PDEs: connections and equivalences}
\label{sec: connections}
    
In this section we will introduce three different perspectives on essentially the same problem.
Throughout, we will assume a fixed filtered probability space $(\Omega, \mathcal{F},(\mathcal{F}_t)_{t \ge 0}, \Theta)$ satisfying the `usual conditions' \cite[Section 21.4]{klenke2013probability} and  consider stochastic differential equations (SDEs) of the form
\begin{equation}
\label{eq:uncontrolled SDE}
\mathrm d X_s = b(X_s,s) \, \mathrm ds + \sigma(X_s,s) \, \mathrm dW_s, \qquad X_t = x_{\mathrm{init}},
\end{equation}
on the time interval $s \in [t,T]$, $0 \le t < T < \infty$. Here, $b: \mathbb{R}^d \times [t, T] \to \R^d$ denotes the drift coefficient, $\sigma: \R^d \times [t,T]\to \R^{d\times d}$ denotes the diffusion coefficient, $(W_s)_{t \le s \le T}$ denotes standard $d$-dimensional Brownian motion, and $x_{\mathrm{init}} \in \mathbb{R}^d$ is the (deterministic) initial condition.  We will work under the following conditions specifying the regularity of $b$ and $\sigma$. 
\begin{assumption}[Coefficients of the SDE \eqref{eq:uncontrolled SDE}]
\label{ass:SDE coefficients}
The coefficients $b$ and $\sigma$ are continuously differentiable, $\sigma$ has bounded first-order spatial derivatives, and $(\sigma \sigma^\top)(x,s)$ is positive definite for all $(x,s) \in \mathbb{R}^d \times [t,T]$. Furthermore,
there exist constants $C, c_1, c_2>0$ such that 
\begin{subequations}
\begin{align}
\label{eq:linear growth}
\vert b(x,s) \vert \le C \left( 1 + \vert x \vert \right), \qquad \qquad \qquad \qquad &\text{(linear growth)} \\
c_1 \vert \xi \vert^2 \le \xi \cdot (\sigma\sigma^\top)(x,s) \xi \le c_2 \vert \xi \vert^2, \qquad \qquad \qquad \qquad & \text{(ellipticity)}
\end{align}
\end{subequations}
for all $(x,s) \in \mathbb{R}^d \times [t,T]$ and $\xi \in \mathbb{R}^d$.
\end{assumption}
    
Let us furthermore introduce a modified version of \eqref{eq:uncontrolled SDE},
\begin{equation}
\label{eq:controlled SDE}
\mathrm d X_s^u = \left(b(X_s^u,s) + \sigma(X_s^u,s) u(X_s^u,s)\right) \mathrm ds + \sigma(X^u_s,s) \, \mathrm dW_s, \qquad X_t^u = x_{\mathrm{init}},
\end{equation}
where we think of $u: \mathbb{R}^d \times [t,T] \to \mathbb{R}^d$ as a control term steering the dynamics. We will throughout assume that $u \in \mathcal{U}$, the set of \emph{admissible controls}.
For definiteness, we will set
\begin{equation}
\label{eq:control set}
\mathcal{U} = \left\{  u \in C^1(\mathbb{R}^d \times [0,T];\mathbb{R}^d): \quad u \,\, \text{grows at most linearly in $x$, in the sense of \eqref{eq:linear growth}}\right\},
\end{equation}
but note that the smoothness and boundedness assumptions can be relaxed in various scenarios. Under Assumption \ref{ass:SDE coefficients} and with $\mathcal{U}$ as defined in \eqref{eq:control set}, the SDEs \eqref{eq:uncontrolled SDE} and \eqref{eq:controlled SDE} admit unique strong solutions according to \cite[Theorem 5.2.1]{oksendal2013stochastic}.

\subsection{Optimal control}
\label{sec:optimal control}
    
Consider the cost functional
\begin{equation}
\label{eq:cost functional}
J(u; x_{\mathrm{init}},t) = {\E}\left[ \int_t^T \left(  f(X^u_s, s) + \frac{1}{2}|u(X^u_s, s)|^2 \right)\mathrm ds + g(X^u_T) \Bigg| X_t^u = x_{\mathrm{init}}  \right],
\end{equation}
where $f \in C^1( \mathbb{R}^d \times [t,T]; [0 ,\infty))$ specifies a part of the running and $g \in C^1( \mathbb{R}^d; \mathbb{R})$ the terminal costs, and $(X^u_s)_{t \le s \le T}$ denotes the unique strong solution to the controlled SDE \eqref{eq:controlled SDE} with initial condition $X_t^u = x_{\mathrm{init}}$. Throughout we assume that $f$ and $g$ are such that the expectation in \eqref{eq:cost functional} is finite, for all $(x_{\mathrm{init}},t) \in \mathbb{R}^d \times [0,T]$. Our objective is to find a control $u \in \mathcal{U}$ that minimises \eqref{eq:cost functional}: 
    
\begin{problem}[Optimal control]
\label{prob:optimal control}
For $(x_{\mathrm{init}},t) \in \mathbb{R}^d \times [0,T]$, find $u^* \in \mathcal{U}$ such that 
\begin{equation}
\label{eq:optimal control min}
J(u^*; x_{\mathrm{init}}, t) = \inf_{u \in \mathcal{U}} J(u;x_{\mathrm{init}},t).
\end{equation}
\end{problem}
    
Defining the \emph{value function}  \cite[Section I.4]{fleming2006controlled}, or `optimal cost-to-go',
\begin{equation}
\label{eq:value_function_definition}
V(x, t) = \inf_{u \in \mathcal{U}} J(u; x, t),
\end{equation}
it is well-known that under suitable conditions, $V$ satisfies a Hamilton-Jacobi-Bellman PDE involving the infinitesimal generator \cite[Section 2.3]{pavliotis2014stochastic} associated to the uncontrolled SDE \eqref{eq:uncontrolled SDE},
\begin{equation}
L = \frac{1}{2} \sum_{i,j=1}^d (\sigma \sigma^\top)_{ij}(x,t) \partial_{x_i} \partial_{x_j} + \sum_{i=1}^d b_i(x,t) \partial_{x_i}.
\end{equation}
The optimal control solving \eqref{eq:optimal control min} can then be  recovered from $u^* = -\sigma^\top \nabla V$ (see Theorem \ref{thm:connections} for details). Let us state this reformulation of Problem \ref{prob:optimal control} as follows:
\begin{problem}[Hamilton-Jacobi-Bellman PDE]
\label{prob:HJB}
Find a solution $V$ to the PDE
\begin{subequations}
\label{eq:HJB}
\begin{align}
\label{eq:HJB 1st}
(L + \partial_t) V(x,t) - \frac{1}{2} \vert \sigma^\top \nabla V(x,t) \vert^2 + f(x,t) & = 0, \qquad  & (x,t) \in \mathbb{R}^d \times [0,T), \\
\label{eq:HJB final condition}
V(x,T) & = g(x), \qquad  & x \in \mathbb{R}^d,
\end{align}
\end{subequations}
where $f$ and $g$ are as in \eqref{eq:cost functional}.
\end{problem}
Throughout, we will focus on solutions to \eqref{eq:HJB} that admit bounded and continuous derivatives of up to first order in time and second order in space (see, however, Remark \ref{rem:viscosity}). This set will be denoted by $C_b^{2,1}(\mathbb{R}^d \times [0,T];\mathbb{R})$. Solutions to elliptic and parabolic PDEs admit probabilistic representations by means of the celebrated Feynman-Kac formulae  \cite[Sections 1.3.3 and 6.3]{pham2009continuous}. To wit, consider the following coupled system of forward-backward SDEs (in the following FBSDEs for short):
\begin{problem}[Forward-backward SDEs]
\label{prob:FBSDE}
For $(x_{\mathrm{init}},t) \in \mathbb{R}^d \times [0,T]$,
find progressively measurable stochastic processes $Y : \Omega \times [t,T] \to \mathbb{R}$ and $Z : \Omega \times [t,T] \to \mathbb{R}^d$ such that 
\begin{subequations}
\label{eq:FBSDE}
\begin{align}
\label{eq:forward SDE}
\mathrm{d} X_s & = b(X_s, s) \, \mathrm{d} s + \sigma(X_s, s) \, \mathrm{d} W_s,  \quad  &X_t   = x_{\mathrm{init}},  \\
\label{eq:backward SDE}
\mathrm{d} Y_s & = -f(X_s,s) \, \mathrm{d} s + \frac{1}{2} \vert Z_s \vert^2 \, \mathrm{d}s + Z_s \cdot \mathrm{d} W_s, \quad &  Y_T   = g(X_T),
\end{align}
\end{subequations}
almost surely.
\end{problem}
Under suitable conditions, It\^{o}'s formula implies that $Y$ is connected to the value function $V$ as defined in \eqref{eq:value_function_definition} via $Y_s = V(X_s,s)$. Similarly, $Z$ is connected to the optimal control $u^*$ through $Z_s = -u^*(X_s,s) = \sigma^\top \nabla V(X_s,s)$. See \cite{pardoux1998backward,pardoux1990adapted} and Theorem \ref{thm:connections} for details.
    
\subsection{Conditioning and rare events}
\label{sec:change-of-path-measure}
    
One major motivation for our work is the problem of sampling rare transition events in diffusion models. In this section we will explain how this challenge can be formalised in terms of weighted measures on path space, leading to a close connection to the optimal control problems encountered in the previous section.
    
We will fix the initial time to be $t=0$, i.e. consider the SDEs \eqref{eq:uncontrolled SDE} and \eqref{eq:controlled SDE} on the interval $[0,T]$. For fixed initial condition $x_{\mathrm{init}} \in \mathbb{R}^d$, let us introduce the path space
\begin{equation}
\mathcal{C} = C_{x_{\mathrm{init}}}([0,T];\mathbb{R}^d) = \left\{ X: [0,T]   \rightarrow \mathbb{R}^d  \,\, \vert \,\, X \; \text{continuous}, \; X_0 = x_{\mathrm{init}} \right\},
\end{equation}
equipped with the supremum norm and the corresponding Borel-$\sigma$-algebra, and denote the set of probability measures on $\mathcal{C}$ by $\mathcal{P}(\mathcal{C})$. The SDEs \eqref{eq:uncontrolled SDE} and \eqref{eq:controlled SDE} induce probability measures on $\mathcal{C}$ defined to be the laws associated to the corresponding strong solutions; those measures will be denoted by $\mathbb{P}$ and $\mathbb{P}^u$, respectively\footnote{Of course, we have that $\mathbb{P}^0$ coincides with the path measure associated to the uncontrolled dynamics, i.e. $\mathbb{P}^0 = \mathbb{P}$.}. Furthermore, we define the \emph{work functional} $\mathcal{W}:\mathcal{C} \to \mathbb{R}$ via
\begin{equation}
\label{eq:work} \mathcal{W}(X)= \int_0^T f(X_s, s) \, \mathrm ds + g(X_T),
\end{equation}
where $f:\mathbb{R}^d \times [0,T] \to \mathbb{R}$ and $g:\mathbb{R}^d \to \mathbb{R}$ are as in Problem \ref{prob:optimal control}. Finally, $\mathcal{W}$ induces a \emph{reweighted} path measure $\mathbb{Q}$ on $\mathcal{C}$ via
\begin{equation}\label{eq:reweighted measure}
\frac{\mathrm d \Q}{\mathrm d \P} = \frac{e^{-\mathcal{W}}}{\mathcal{Z}}, \qquad \mathcal{Z} = \mathbb{E} \left[ \exp(-\mathcal{W}(X)) \right],
\end{equation}
assuming $f$ and $g$ are such that $\mathcal{Z}$ is finite (we shall tacitly make this assumption from now on).
We may ask whether $\mathbb{Q}$ can be obtained as the path measure related to a controlled SDE of the form \eqref{eq:controlled SDE}:
\begin{problem}[Conditioning]
\label{prob:conditioning}
Find $u^* \in \mathcal{U}$ such that the path measure $\mathbb{P}^{u^*}$ associated to \eqref{eq:controlled SDE} coincides with $\mathbb{Q}$.
\end{problem}
Referring to the above as a  conditioning problem is justified by the fact that \eqref{eq:reweighted measure} may be viewed as an instance of Bayes' formula relating conditional probabilities \cite{reich2019data}. This connection can be formalised using Doob's $h$-transform \cite{doob1957conditional,doob2012classical} and applied to diffusion bridges and quasistationary distributions, for instance (see \cite{chetrite2015nonequilibrium} and references therein). 
\begin{example}[Rare event simulation]
\label{ex:rare events}
Let us consider SDEs of the form \eqref{eq:uncontrolled SDE}, where the drift is a gradient, i.e. $b = - \nabla \Psi$, and the potential $\Psi$ is of multimodal type.
As an example we shall discuss the one-dimensional case $d=1$ and assume that $\Psi \in C^\infty(\mathbb{R})$ is given by
\begin{equation}
\Psi(x) = \kappa (x^2-1)^2,
\end{equation}
with $\kappa > 0$. Furthermore, let us fix the initial conditions $x_{\mathrm{init}} = -1$ and $t=0$, and 
assume a constant diffusion coefficient of size unity, $\sigma = 1$.
Observe that $\Psi$ exhibits two local minima at $x = \pm 1$, separated by a barrier at $x=0$, the height of which is modulated by the parameter $\kappa$ (see Figure \ref{fig: double well illustration} in Section \ref{sec:double well} for an illustration). When $\kappa$ is sufficiently large, the dynamics induced by \eqref{eq:uncontrolled SDE} exhibits metastable behaviour: transitions between the two basins happen very rarely as the transition time depends exponentially on the height of the barrier \cite{berglund2011kramers,kramers1940brownian}. Applications such as molecular dynamics are often concerned with statistics and derived quantities from these rare events as those are typically directly linked to biological functioning \cite{weinan2004metastability,schutte2003biomolecular,schutte2013metastability}. At the same time, computational approaches face a difficult sampling problem as transitions are hard to obtain by direct simulation from \eqref{eq:uncontrolled SDE}. Choosing $f  = 0$ and $g$ such that $e^{-g}$ is concentrated around $x=1$ (consider, for instance, $g(x) = \nu(x-1)^2$ with $\nu > 0$ sufficiently large), we see that $\mathbb{Q}$ as defined in \eqref{eq:reweighted measure} predominantly charges paths initialised in $x=-1$ at $t=0$ and enter a neighbourhood of $x=1$ at final time $T$. Problem \ref{prob:conditioning} can then be understood as the task of finding a control $u$ that allows efficient simulation of transition paths. Similar issues arise in the context of stochastic filtering, where the objective is to sample paths that are compatible with available data \cite{reich2019data}.
\end{example}
    
\subsection{Sampling problems}
    
The \emph{free energy} \cite{hartmann2017variational} associated to the dynamics \eqref{eq:uncontrolled SDE} and the work functional \eqref{eq:work} is given by
\begin{equation}
\label{eqn: free energy}
\gamma = - \log \mathbb{E} \left[\exp (-\mathcal{W}(X)) \right] = -\log \mathcal{Z},
\end{equation}
where the normalising constant $\mathcal{Z}$ has been defined in \eqref{eq:reweighted measure}. The problem of computing $\mathcal{Z}$ is ubiquitous in nonequilibrium thermodynamics and statistics \cite{blei2017variational,stoltz2010free}, and, quite often, the variance associated to the random variable $\exp (-\mathcal{W}(X))$ is so large as to render direct estimation of the expectation $\mathbb{E} \left[\exp (-\mathcal{W}(X)) \right]$ computationally infeasible\footnote{In fact, the variance is particularly large in metastable scenarios such as those sketched in Example \ref{ex:rare events}.}. A natural approach is then to use the identity
\begin{equation}
\label{eq:importance sampling}
\mathbb{E}\left[\exp (-\mathcal{W}(X))\right] = \mathbb{E} \left[\exp(-\mathcal{W}(X^u)) \frac{\mathrm{d}\mathbb{P}}{\mathrm{d}\mathbb{P}^u}  \right], \qquad u \in \mathcal{U},
\end{equation}
where we recall that $X$ and $X^u$ refer to the strong solutions to \eqref{eq:uncontrolled SDE} and \eqref{eq:controlled SDE}, respectively, and $\frac{\mathrm{d}\mathbb{P}}{\mathrm{d}\mathbb{P}^u}$ denotes the Radon-Nikodym derivative, explicitly given by Girsanov's theorem\footnote{By a slight abuse of notation, \eqref{eq:Girsanov IS} is to be interpreted as a random variable on $\Omega$ provided by the measurable map $\omega \mapsto X^u$ induced by \eqref{eq:controlled SDE}. In other words, the left-hand side should be read as $\frac{\mathrm{d}\P}{\mathrm{d}\P^u}(X^u(\omega))$.} \cite[Theorem 2.1.1]{ustunel2013transformation},
\begin{equation}
\label{eq:Girsanov IS}
\frac{\mathrm{d}\mathbb{P}}{\mathrm{d}\mathbb{P}^u}  = \exp \left( -\int_0^T u(X_s^u,s) \cdot \, \mathrm{d}W_s - \frac{1}{2} \int_0^T \vert u(X_s^u,s) \vert^2 \, \mathrm{d}s \right),
\end{equation}
see the proof of Theorem \ref{thm:connections}.
As explained in \cite{hartmann2017variational}, techniques leveraging \eqref{eq:importance sampling} may be thought of as instances of importance sampling on path space. Given that \eqref{eq:importance sampling} holds for all $u \in \mathcal{U}$, it is clearly desirable to choose the control such as to guarantee favourable statistical properties:
\begin{problem}[Variance minimisation]
\label{prob:variance minimisation}
Find $u^* \in \mathcal{U}$ such that
\begin{equation}
\label{eq:variance}
\Var \left(\exp(-\mathcal{W}(X^{u^*})) \frac{\mathrm{d}\mathbb{P}}{\mathrm{d}\mathbb{P}^{u^*}}  \right) = \inf_{u \in \mathcal{U}}  \Var \left(\exp(-\mathcal{W}(X^{u})) \frac{\mathrm{d}\mathbb{P}}{\mathrm{d}\mathbb{P}^{u}}  \right).
\end{equation}
\end{problem}
Under suitable conditions, it turns out that there exists $u^* \in \mathcal{U}$ such the variance expression \eqref{eq:variance} is  in fact zero (see Theorem \ref{thm:connections}, \eqref{thm: optimal control implies variance zero}), providing a perfect sampling scheme. 
\par\bigskip    
The problem formulations detailed so far are intimately connected as summarised by the following theorem:

\begin{theorem}[Connections and equivalences]
\label{thm:connections}
The following holds:
\begin{enumerate}
\item Let $V \in C_b^{2,1}(\mathbb{R}^d \times [0,T];\mathbb{R})$ be a solution to Problem \ref{prob:HJB}, i.e. solve the HJB-PDE \eqref{eq:HJB}. Set 
\begin{equation}
\label{eq:connections u star}
u^* =-\sigma^\top \nabla V.
\end{equation}
Then
\begin{enumerate}
\item
\label{equiv_a}
the control $u^*$ provides a solution to Problem \ref{prob:optimal control}, i.e. $u^*$ minimises the objective \eqref{eq:cost functional},
\item
\label{it:HJB->FBSDE}
the pair 
\begin{equation}
\label{eq:YZ in terms of V}
Y_s = V(X_s, s), \qquad Z_s = \sigma^\top \nabla V(X_s, s)
\end{equation}
solves the FBSDE \eqref{eq:FBSDE}, i.e.  Problem \ref{prob:FBSDE},
\item
\label{it:measure}
the measure $\mathbb{P}^{u^*}$ associated to the controlled SDE \eqref{eq:controlled SDE} coincides with $\mathbb{Q}$, i.e. $u^*$ solves Problem \ref{prob:conditioning},
\item
\label{thm: optimal control implies variance zero}
the control $u^*$ provides the minimum-variance estimator in \eqref{eq:variance}, i.e. $u^*$ solves Problem \ref{prob:variance minimisation}. Moreover, the variance is in fact zero, i.e. the random variable
\begin{equation}
\exp(-\mathcal{W}(X^{u^*})) \frac{\mathrm{d}\mathbb{P}}{\mathrm{d}\mathbb{P}^{u^*}}
\end{equation}
is almost surely constant.
\end{enumerate}
Furthermore, we have that
\begin{equation}
\label{eq:equivalent quantities}
J(u^*; x_{\mathrm{init}},0) = V(x_{\mathrm{init}},0) = Y_0 = - \log \mathcal{Z}.
\end{equation}
\item
Conversely, let $u^* \in \mathcal{U}$ solve Problem \ref{prob:conditioning}, i.e. assume that $\mathbb{P}^{u^*}$ coincides with $\mathbb{Q}$. Then the statement \eqref{thm: optimal control implies variance zero} holds. Furthermore, setting 
\begin{equation}
\label{eq:Y0 Z u}
Y_0 = -\log \mathcal{Z}, \qquad Z_s = -u^*(X_s,s),
\end{equation}
solves the backward SDE \eqref{eq:backward SDE} from Problem \ref{prob:FBSDE}, i.e. \eqref{eq:Y0 Z u} together with the first equation in \eqref{eq:backward SDE} determines a process $(Y_s)_{0 \le s \le T}$ that satisfies the final condition $Y_T = g(X_T)$, almost surely.
\end{enumerate}
\end{theorem}

\begin{remark}
We extend the connections between the optimal control formulation (Problem \ref{prob:optimal control}) and FBSDEs (Problem \ref{prob:FBSDE}) in Proposition \ref{prop:log var and KL}, see also Remark \ref{rem:log var and KL}.
\end{remark}

\begin{remark}[Regularity, uniqueness, and further connections]
\label{rem:viscosity}
Going beyond classical solvability of the HJB-PDE \eqref{eq:HJB} and introducing the notion of \emph{viscosity solutions} \cite{fleming2006controlled,pardoux1998backward}, the strong regularity and boundedness assumptions on $V$ in the first statement could be much relaxed and the connections exposed in Theorem \ref{thm:connections} could be extended \cite{pham2009continuous,yong1999stochastic}. As a case in point, we note that in the current setting, neither a solution to Problem \ref{prob:optimal control} nor to Problem \ref{prob:FBSDE} necessarily provides a classical solution to the PDE \eqref{eq:HJB}, as optimal controls are known to be non-differentiable, in general.

However, assuming classical well-posedness of the HJB-PDE \eqref{eq:HJB}, Theorem \ref{thm:connections} implies that the solution can be found by addressing one of the Problems \ref{prob:optimal control}, \ref{prob:FBSDE}, \ref{prob:conditioning} or \ref{prob:variance minimisation} and using the formulas \eqref{eq:connections u star} and \eqref{eq:YZ in terms of V}, as long as those problems admit \emph{unique} solutions, in an appropriate sense. For the latter issue, we refer the reader to \cite{kobylanski2000backward} and \cite[Chapter 11]{touzi2012optimal} in the context of forward-backward SDEs and to \cite{bierkens2014explicit} in the context of measures on path space. We note that, in particular, the forward SDE \eqref{eq:forward SDE} can be thought of as providing a random grid for the solution of the HJB-PDE \eqref{eq:HJB}, obtained through the backward SDE \eqref{eq:backward SDE}.   
\end{remark}
\begin{remark}
[Random initial conditions]
\label{rem:randomisation}
The equivalence between Problems \ref{prob:HJB} and \ref{prob:FBSDE} shows that $u^*$ does not depend on $x_{\mathrm{init}}$. Consequently, the initial condition in \eqref{eq:forward SDE} can be random rather than deterministic. In Section \ref{sec: LQGC} we demonstrate potential benefits of this extension for FBSDE-based algorithms. 
\end{remark}
\begin{remark}[Variational formulas and duality] The identities \eqref{eq:equivalent quantities} connect key quantities pertaining to the problem formulations \ref{prob:optimal control}, \ref{prob:HJB}, \ref{prob:FBSDE} and \ref{prob:conditioning}. The fact that $J(u^*; x_{\mathrm{init}},0) =  - \log \mathcal{Z}$ can moreover be understood in terms of the Donsker-Varadhan formula \cite{boue1998variational}, furnishing an explicit expression for the value function,
\begin{equation}
\label{eq:DV}
    V(x,t) = -\log \mathbb{E} \left[ \exp \left(-\int_t^T f(X_s,s)\,\mathrm{d}s - g(X_T) \right) \Bigg| X_t = x\right],
\end{equation}
 as discussed in \cite{daipra91,dai1996connections,hartmann2017variational}.
\end{remark}
\begin{remark}[Generalisations]
The problem formulations \ref{prob:optimal control}, \ref{prob:HJB} and \ref{prob:FBSDE} admit generalisations that keep parts of the connections expressed in Theorem \ref{thm:connections} intact. 
From the PDE-perspective (Problem \ref{prob:HJB}), it is possible to consider more general nonlinearities,
\begin{subequations}
\label{eq:nonlinear PDE}
\begin{align}
\label{eq:nonlinear 1st}
(L + \partial_t) V(x,t) + h(x,t,V(x,t),(\sigma^\top \nabla V)(x,t)) & = 0,   & (x,t) \in \mathbb{R}^d \times [0,T), \\
\label{eq:nonlinear terminal}
V(x,T) & = g(x),   & x \in \mathbb{R}^d,
\end{align}
\end{subequations}
with $h$ being a function satisfying appropriate regularity and boundedness assumptions. As in Theorem \ref{thm:connections} (\ref{it:HJB->FBSDE}), the nonlinear parabolic PDE \eqref{eq:nonlinear PDE} is related to a generalisation of the forward-backward system \eqref{eq:FBSDE},
\begin{subequations}
\label{eq:FBSDE generalised}
\begin{align}
\label{eq:generalised forward}
\mathrm{d} X_s & = b(X_s, s) \, \mathrm{d} s + \sigma(X_s, s) \, \mathrm{d} W_s,  \quad  &X_t   = x_{\mathrm{init}}, \\
\label{eq:generalised backward}
\mathrm{d} Y_s & = -h(X_s,s,Y_s,Z_s) \, \mathrm{d} s + Z_s \cdot \mathrm{d} W_s, \quad &  Y_T   = g(X_T),
\end{align}
\end{subequations}
where the connection is still given by \eqref{eq:YZ in terms of V}, see \cite[Section 6.3]{pham2009continuous}.
From the perspective 
of optimal control (Problem \ref{prob:optimal control}), it is possible to extend the discussion to SDEs of the form 
\begin{equation}
\mathrm d X_s^u = \widetilde{b}(X_s^u, s, u_s) \, \mathrm d s + \widetilde{\sigma}(X_s^u, s, u_s)\, \mathrm d W_s,
\end{equation}
replacing  \eqref{eq:controlled SDE}, and to running costs $\widetilde{f}(X^u_s, u_s, s)$ instead of $f(X^u_s, s) + \frac{1}{2}|u(X^u_s, s)|^2$ in \eqref{eq:cost functional}, assuming that $u_s \in \widetilde{U} \subset \mathbb{R}^m $, for some $m \in \mathbb{N}$.
This setting gives rise to more general HJB-PDEs,
\begin{equation}
\label{eq:general HJB}
\partial_t V(x,t) + H(x, t, \nabla V(x, t), \nabla^2 V(x, t)) = 0,
\end{equation}
where $\nabla^2 V$ denotes the Hessian of $V$, 
and the \emph{Hamiltonian} $H$ is given by
\begin{equation}
\label{eq:Ham}
H(x, t, p, A) = \inf_{u \in \widetilde{U}}   \left[ \widetilde{b}(x,t,u)\cdot p + \tfrac{1}{2} \Tr(\widetilde{\sigma} \widetilde{\sigma}^\top A)(x,t,u) + \widetilde{f}(x,t,u)
\right], 
\end{equation}
see
\cite{fleming2006controlled,pham2009continuous}.
In certain scenarios \cite[Section 4.5.2]{zhang2017backward}, it is then possible to relate \eqref{eq:general HJB} to \eqref{eq:FBSDE generalised}, noting however that typically $h$ will be given in terms of a minimisation problem as in \eqref{eq:Ham}.
The relationship to Problems \ref{prob:conditioning} and \ref{prob:variance minimisation} as well as the identity \eqref{eq:connections u star} rest on the particular structure\footnote{Note that this structure connects the PDEs \eqref{eq:general HJB} and \eqref{eq:HJB} in view of $H(x, t, \nabla V, \nabla^2 V) = LV +f +\min_{u\in \mathcal{U}}\left\{ \sigma u \cdot \nabla V + \frac{1}{2}|u|^2\right\}$ and $\min_{u\in \mathcal{U}}\left\{ \sigma u \cdot \nabla V + \frac{1}{2}|u|^2\right\} = -\frac{1}{2} |\sigma^\top \nabla V|^2 $.} inherent in \eqref{eq:controlled SDE} and \eqref{eq:cost functional}, enabling the use of Girsanov's theorem (see the Proof of Theorem \ref{thm:connections} below).
The methods developed in this paper based on the log-variance loss \eqref{def_log_variance_loss} can straightforwardly be extended to equations of the form \eqref{eq:nonlinear PDE} in the case when $h$ depends on $V$ only through $\nabla V$, owing to the invariance of the PDE under shifts of the form $V \mapsto V + \mathrm{const.}$, see Remark \ref{rem:shift V}. In order to address optimal control problems involving additional minimisation tasks posed by Hamiltonians such as \eqref{eq:Ham} it might be feasible to include appropriate penalty terms in the loss functional. We leave this direction for future work.
\end{remark}
\begin{proof}[Proof of Theorem \ref{thm:connections}]
The statement \eqref{equiv_a} is a classical result in stochastic optimal control theory, often referred to as a  \textit{verification theorem}, and can for instance be found in \cite[Theorem IV.4.4]{fleming2006controlled} or \cite[Theorem 3.5.2]{pham2009continuous}.
The implication \eqref{it:HJB->FBSDE} is a direct consequence of It\^{o}'s formula, cf. \cite[Proposition 6.3.2]{pham2009continuous} or \cite[Proposition 2.14]{carmona2016lectures}. Before proceeding to \eqref{it:measure}, we note that the first equality in \eqref{eq:equivalent quantities} now follows from \eqref{eq:value_function_definition} (for background, see \cite[Section  IV.2]{fleming2006controlled}), while the second equality is a direct consequence of \eqref{it:HJB->FBSDE}. Using \eqref{eq:FBSDE} and \eqref{it:HJB->FBSDE}, the third equality follows from
\begin{equation}
\mathcal{Z} = \mathbb{E} \left[ \exp (-\mathcal{W}(X)\right] = \exp(-Y_0) \cdot  \mathbb{E}\left[ \exp \left( \int_0^T u^*(X_s,s) \cdot \mathrm{d}W_s - \frac{1}{2} \int_0^T \vert u^*(X_s,s)\vert^2 \mathrm ds \right)\right] = \exp(-Y_0),
\end{equation}
relying on the facts that $Y_0$ is deterministic (again using \eqref{it:HJB->FBSDE}), and that the term inside the second expectation is a martingale (as $u^*$ is assumed to be bounded). 
Turning to \eqref{it:measure}, let us define an equivalent measure $\widetilde{\Theta}$ on $(\Omega,\mathcal{F})$ via 
\begin{equation}
\label{eq:Theta tilde}
\frac{\mathrm{d}\widetilde{\Theta}}{\mathrm{d}\Theta} = \exp \left( \int_0^T u^*(X_s,s) \cdot \mathrm{d}W_s - \frac{1}{2} \int_0^T \vert u^*(X_s,s) \vert^2 \, \mathrm{d}s \right).
\end{equation}
Since $u^*$ is assumed to be bounded, Novikov's condition is satisfied, and hence Girsanov's theorem asserts that the process $(\widetilde{W}_t)_{0 \le t \le T}$ defined by
\begin{equation}
\widetilde{W}_t = W_t - \int_0^t u^*(X_s,s) \, \mathrm{d}s 
\end{equation}
is a Brownian motion with respect to $\widetilde{\Theta}$. Consequently, we have that
\begin{equation}
\label{eq:dPu}
\frac{\mathrm{d}\mathbb{P}^{u^*}}{\mathrm{d}\mathbb{P}}(X(\omega)) = \frac{\mathrm{d}\widetilde{\Theta}}{\mathrm{d}\Theta}(\omega) = \exp \left(Y_0 - \mathcal{W}(X(\omega))\right) = \frac{\mathrm{d}\mathbb{Q}}{\mathrm{d}\mathbb{P}}(X(\omega)), \qquad \omega \in \Omega, 
\end{equation}
using \eqref{eq:FBSDE} and \eqref{eq:equivalent quantities} in the last step. We note that similar arguments can be found in \cite{kebiri2017adaptive}, \cite[Section 3.3.1]{carmona2018probabilistic}.
    
For the proof of \eqref{thm: optimal control implies variance zero} we refer to \cite[Theorem 2]{hartmann2017variational}.  
The proof of the second statement is very similar to the argument presented for \eqref{it:measure}, resting primarily on \eqref{eq:Theta tilde} and \eqref{eq:dPu}, and is therefore omitted. 
    \end{proof}
    
\subsection{Algorithms and previous work}
\label{sec: previous work}
    
The numerical treatment of optimal control problems has been an active area of research for many decades and multiple perspectives on solving Problem \ref{prob:optimal control} have been developed. The monographs \cite{bertsekas2011dynamic} and \cite{kushner2013numerical} provide good overviews to \textit{policy iteration} and \textit{$Q$-learning}, strategies that have been further investigated in the machine learning literature and that are generally subsumed under the term \textit{reinforcement learning} \cite{powell2019reinforcement}. We also recommend \cite{kappen2007introduction} as an introduction to the specific setting considered in this paper. To cope with the key issue of high dimensionality, the authors of \cite{oster2019approximating} suggest solving a certain type of control problem in the framework of hierarchical tensor products. Another strategy of dealing with the curse of dimensionality is to first apply a model reduction technique and only then solve for the reduced model. Here, recent results on \textit{balanced truncation} for controlled linear S(P)DEs have for instance been suggested in \cite{becker2019feedback}, and approaches for systems with a slow-fast scale separation via the \textit{homogenisation} method can be found in \cite{zhang2014optimal}.
\par\bigskip
    
Solutions to Problem \ref{prob:HJB}, i.e. to HJB-PDEs of the type \eqref{eq:HJB}, can be approximated through finite difference or finite volume methods \cite{achdou2013finite,oberman2006convergent,peyrl2005numerical}. However, these approaches are usually not applicable in high-dimensional settings.
In contrast, the recently introduced \emph{Multilevel Picard} method \cite{hutzenthaler2016multilevel} based on a combination of the Feynman-Kac and Bismut-Elworthy-Li formulas has been proven to beat the curse of dimensionality in a variety of settings, see \cite{beck2019overcoming,hutzenthaler2019overcoming,hutzenthaler2018overcoming,hutzenthaler2019financial,hutzenthaler2020multilevel}. 
\par\bigskip

The FBSDE formulation (Problem \ref{prob:FBSDE}) has opened the door for Monte Carlo based methods that have been developed since the early 90s. We mention in particular \textit{least-squares Monte Carlo}, where $(Z_s)_{0 \le s \le T}$ is approximated iteratively backwards in time by solving a regression problem in each time step, along the lines of the dynamic programming principle \cite[Chapter 3]{pham2009continuous}. A good introduction can be found in \cite{gobet2016monte}; for extensive analysis on numerical errors we refer the reader to \cite{gobet2005regression, zhang2004numerical}. Recently, this approach has also been connected with deep learning, replacing Galerkin approximations by neural networks \cite{hure2019some}, as well as with the tensor train format, exploiting inherent low rank structures \cite{richter2021solving}. \par\bigskip
    
Another method leveraging the FBSDE perspective has been put forward in \cite{weinan2017deep,han2018solving} and further developed in \cite{beck2018solving,beck2019machine}. Here, the main idea is to enforce the terminal condition $Y_T = g(X_T)$ in \eqref{eq:backward SDE} by iteratively minimising the loss function
\begin{equation}
\label{moment_loss_first_mention}
\mathcal{L}(u,y_0) = \E\left[(Y_T(y_0,u) - g(X_T))^2\right],
\end{equation}
using a stochastic gradient descent IDO scheme. The notation $Y_T(y_0,u)$ indicates that the process in \eqref{eq:backward SDE} is to be simulated with given initial condition $y_0$ and control $u$ (these representing a priori guesses or current approximations, typically relying on neural networks), hence viewing \eqref{eq:backward SDE} as a forward process. Consequently, the approach thus described can be classified as a \emph{shooting method} for boundary value problems. We note that this idea allows treating rather general parabolic and elliptic PDEs \cite{grohs2018proof,hutzenthaler2019proof}, as well as -- with some modifications -- optimal stopping problems \cite{becker2019deep,becker2019solving}, going beyong the setting considered in this paper. Using neural network approximations in conjunction with FBSDE-based Monte-Carlo techniques holds the promise of alleviating the curse of dimensionality; understanding this phenomenon and proving rigorous mathematical  statements has been been the focus of intense current research    \cite{berner2018analysis,grohs2018proof,grohs2019deep,hutzenthaler2019proof,jentzen2018proof}. Let us also mention that similar algorithms have been suggested in \cite{raissi2018forward,raissi2019physics}, in particular  proposing to modify the loss function \eqref{moment_loss_first_mention} in order to encode the backward dynamics \eqref{eq:backward SDE}, and extensive investigation of optimal network design and choice of tunable parameters has been carried out \cite{chan2019machine}. Furthermore, we refer to \cite{carmona2019convergence1,carmona2019convergence2} for convergence results in the broader context of mean field control. In \cite[Section III.B]{hartmann2019variational} it has been proposed to modify the forward dynamics \eqref{eq:forward SDE} (and, to compensate, also the backward dynamics \eqref{eq:backward SDE}) by an additional control term. This idea is central for the main results of this paper, see Section \ref{sec:BSDEs}. Similar ideas for other types of PDEs have been proposed as well, see for instance \cite{eigel2019variational,raissi2019physics}.
\\\\
Conditioned diffusions (Problem \ref{prob:conditioning}) have been considered in a large deviation context \cite{dupuis2004importance} as well as in a variational setting \cite{hartmann2019variational,hartmann2017variational} motivated by free energy computations, building on earlier work in \cite{boue1998variational, dai1996connections}, see also \cite{baudoin2002conditioned,chetrite2015nonequilibrium,daipra91,  ferre2018adaptive}. The simulation of diffusion bridges has been studied in \cite{mider2019simulating} and conditioning via Doob's $h$-transform has been employed in a sequential Monte Carlo context \cite{heng2017controlled}. The formulation in Problem \ref{prob:conditioning} identifies the target measure $\mathbb{Q}$, motivating approaches that seek to minimise certain divergences on path space. This perspective will be developed in detail in Section \ref{sec:divergences}, building bridges to Problems \ref{prob:optimal control}, \ref{prob:HJB}, \ref{prob:FBSDE} and \ref{prob:variance minimisation}. Prior work following this direction includes \cite{bierkens2014explicit,gomez2014policy,hartmann2012efficient,kappen2012optimal,rawlik2013stochastic}, in particular relying on a connection between the $\KL$-divergence (or relative entropy) on path space and the cost functional \eqref{eq:cost functional}, see also Proposition \ref{prop:relative entropy}. A similar line of reasoning leads to the \textit{cross-entropy method} \cite{hartmann2017variational,kappen2016adaptive,rubinstein2013cross,zhang2014applications}, see Proposition \ref{prop:cross entropy} and equation \eqref{eq:CE estimator} in Section \ref{sec:ensemble averages}.   
\\\\
Problem \ref{prob:variance minimisation} motivates minimising the variance of importance sampling estimators. We refer the reader to \cite[Section 5.2]{muller2018neural} for a recent attempt based on neural networks, to \cite{deniz2019convergence} for a theoretical analysis of convergence rates, to \cite{hartmann2021nonasymptotic} for potential non-robustness issues, and to \cite{bugallo2017adaptive} for a general overview regarding adaptive importance sampling techniques. The relationship between optimal control and importance sampling (see Theorem \ref{thm:connections}) has been exploited by various authors to construct efficient samplers \cite{kappen2016adaptive, thijssen2015path}, in particular also with a view towards the sampling based estimation of hitting times, in which case optimal controls are governed by elliptic rather than parabolic PDEs \cite{hartmann2014characterization,hartmann2019variational,hartmann2012efficient,hartmann2018importance}. Similar sampling problems have been addressed in the context of sequential Monte Carlo \cite{del2000branching,heng2017controlled} and generative models \cite{tzen2019neural,tzen2019theoretical}. The latter works examine the potential of the controlled SDE \eqref{eq:controlled SDE} as a sampling device targeting a suitable distribution of the final state $X^u_T$. 
    
\section{Approximating probability measures on path space}
\label{sec: divergences section}
    
In this section we demonstrate that many of the algorithmic approaches encountered in the previous section can be recovered as minimisation procedures of certain divergences between probability measures on path space. Similar perspectives (mostly discussing the relative entropy and cross-entropy in Definition \ref{def_RE_and_CE_loss} below) can be found in the literature, see \cite{hartmann2012efficient,kappen2012optimal,zhang2014applications}. 
Recall from Section \ref{sec:change-of-path-measure} that we denote by $\mathcal{C}$ the space of $\mathbb{R}^d$-valued paths on the time interval $[0,T]$ with fixed initial point $x_{\mathrm{init}} \in \mathbb{R}^d$. As before, the probability measures on $\mathcal{C}$ induced by \eqref{eq:uncontrolled SDE} and \eqref{eq:controlled SDE} will be denoted by $\mathbb{P}$ and $\mathbb{P}^u$, respectively.
From now on, let us assume that there exists a unique optimal control with convenient regularity properties:
\begin{assumption}
\label{ass:u*}
The HJB-PDE \eqref{eq:HJB} admits a unique solution $V \in C_b^{2,1}(\mathbb{R}^d \times [0,T])$. We set
\begin{equation}
\label{eq:u_star}
u^* = - \sigma^\top \nabla V.
\end{equation}
\end{assumption}
For Assumption \ref{ass:u*} to be satisfied, it is sufficient to impose the regularity and boundedness conditions $b,\sigma,f \in C_b^{2,1}(\mathbb{R}^d)$ and $g \in C_b^{3}(\mathbb{R}^d)$, see\footnote{This result requires the boundedness of the controls in $\mathcal{U}$. However, applying \cite[Chapter II, Theorem 3.1] {kunita1984stochastic} to \eqref{eq:DV}, we see that $\nabla V$ is bounded and hence $\mathcal{U}$ can be restricted appropriately.} \cite[Theorem 4.2]{fleming2006controlled}. The strong boundedness assumption on $V$ could be weakened and for instance be replaced by the condition $\sigma^\top \nabla V \in \mathcal{U}$. For existence and uniqueness results involving unbounded controls we refer to \cite{fleming1969controlled}, and for specific examples to Sections \ref{section_LLQC} and \ref{sec: LQGC}.
In the sense made precise in Theorem \ref{thm:connections}, the control $u^*$ defined above provides solutions to the Problems \ref{prob:optimal control}-\ref{prob:variance minimisation} considered in Section \ref{sec: connections}. Moreover, there exists a corresponding optimal path measure $\mathbb{Q}$ (in the following also called the \emph{target measure}) defined in \eqref{eq:reweighted measure} and satisfying $\mathbb{Q} = \mathbb{P}^{u^*}$. We further note that Assumption \ref{ass:u*} together with the results from \cite[Chapter 11]{touzi2012optimal} imply that the solution to the FBSDE \eqref{eq:FBSDE} is unique.

\subsection{Divergences and loss functions}
\label{sec:divergences}
The SDE \eqref{eq:controlled SDE} establishes a measurable map $\mathcal{U} \ni u \mapsto \mathbb{P}^u \in \mathcal{P}(\mathcal{C})$ that can be made explicit in terms of Radon-Nikodym derivatives using Girsanov's theorem (see Lemma \ref{lem:Girsanov} in Appendix \ref{app:divergences}). 
Consequently, we can elevate divergences between path measures to loss functions on vector fields. To wit, let $D: \mathcal{P}(\mathcal{C})\times \mathcal{P}(\mathcal{C}) \rightarrow \mathbb{R}_{\ge 0}\ \cup \{+\infty\}$ be a divergence\footnote{The defining property of a divergence between probability measures is the equivalence between $D(\mathbb{P}_1 \vert \mathbb{P}_2) = 0$ and $\mathbb{P}_1 = \mathbb{P}_2$. Prominent examples include the $\KL$-divergence and, more generally, the $f$-divergences \cite{liese2006divergences}.}, where, as before, $\mathcal{P}(\mathcal{C})$ denotes the set of probability measures on $\mathcal{C}$. Then, setting
\begin{equation}
\mathcal{L}_D(u) = D(\mathbb{P}^u \vert \mathbb{Q}), \qquad u \in \mathcal{U},
\end{equation}
we immediately see that $\mathcal{L}_D \ge 0$, with Theorem \ref{thm:connections} implying that $\mathcal{L}_D(u) = 0$ if and only if $u = u^*$. Consequently, an approximation of the optimal control vector field $u^*$ can in principle be found by minimising the loss $\mathcal{L}_D$. In the remainder of the paper, we will suggest possible losses and study some of their properties. \par\bigskip

Starting with the $\KL$-divergence, we introduce the \emph{relative entropy loss} and the \emph{cross-entropy loss}, corresponding to the divergences
\begin{equation}
D^{\RE}(\mathbb{P}_1 \vert \mathbb{P}_2) = \KL(\mathbb{P}_1 \vert \mathbb{P}_2) \qquad \text{and} \qquad D^{\CE}(\mathbb{P}_1 \vert \mathbb{P}_2) = \KL(\mathbb{P}_2 \vert \mathbb{P}_1).
\end{equation}
    
\begin{definition}[Relative entropy and cross-entropy losses]
\label{def_RE_and_CE_loss}
The \emph{relative entropy loss}
is given by
\begin{equation}
\label{eq:relative entropy}
\mathcal{L}_{\RE}(u) = {\E}_{\P^u}\left[\log \frac{\mathrm d \P^u}{\mathrm d \Q} \right], \qquad  u \in \mathcal{U},
\end{equation}
and the \emph{cross-entropy loss} by
\begin{equation}
\label{eq:cross entropy}
\mathcal{L}_{\CE}(u) = {\E}_{\Q}\left[\log \frac{\mathrm d \Q}{\mathrm d \P^u} \right], \qquad u \in \mathcal{U},
\end{equation}
where the target measure $\mathbb{Q}$ has been defined in \eqref{eq:reweighted measure}.
\end{definition}
\begin{remark}[Notation]
\label{rem:notation}
Note that, by definition, the expectations in \eqref{eq:relative entropy} and \eqref{eq:cross entropy} are understood as integrals on $\mathcal{C}$, i.e.
\begin{equation}
\mathcal{L}_{\RE}(u) = \int_{\mathcal{C}}\left(\log \frac{\mathrm d \P^u}{\mathrm d \Q} \right) \mathrm{d}\mathbb{P}^u, \qquad \mathcal{L}_{\CE}(u) = \int_{\mathcal{C}}\left(\log \frac{\mathrm d \Q}{\mathrm d \P^u} \right) \mathrm{d}\mathbb{Q}.  
\end{equation}
In contrast, the expectation operator $\mathbb{E}$ (without subscript, as used in \eqref{eq:cost functional} and \eqref{eq:importance sampling}, for instance) throughout denotes integrals on the underlying abstract probability space $(\Omega, \mathcal{F},(\mathcal{F}_t)_{t \ge 0}, \Theta)$.
\end{remark}

For $\widetilde{\mathbb{P}} \in \mathcal{P}(\mathcal{C})$, it is straightforward to verify that 
\begin{equation}
\label{eq:var divergence}
D^{\mathrm{Var}}_{\widetilde{\mathbb{P}}}(\mathbb{P}_1 \vert \mathbb{P}_2) = 
\begin{cases}
{\Var}_{\widetilde{\mathbb{P}}} \left( \frac{\mathrm{d}\mathbb{P}_2}{\mathrm{d}\mathbb{P}_1}\right), \quad & \text{if } \mathbb{P}_1 \sim \mathbb{P}_2 \quad\text{and}\quad \mathbb{E}_{\widetilde{\mathbb{P}}}\left[\left| \frac{\mathrm{d}\mathbb{P}_2}{\mathrm{d}\mathbb{P}_1}\right| \right] < \infty,\\
+ \infty, \qquad &\text{otherwise,} 
\end{cases}
\end{equation}
and 
\begin{equation}
\label{eq:var log divergence}
D^{\mathrm{Var(log)}}_{\widetilde{\mathbb{P}}}(\mathbb{P}_1 \vert \mathbb{P}_2) = 
\begin{cases}
{\Var}_{\widetilde{\mathbb{P}}} \left( \log \frac{\mathrm{d}\mathbb{P}_2}{\mathrm{d}\mathbb{P}_1}\right), \quad & \text{if } \mathbb{P}_1 \sim \mathbb{P}_2 \quad\text{and}\quad \mathbb{E}_{\widetilde{\mathbb{P}}}\left[\left| \log \frac{\mathrm{d}\mathbb{P}_2}{\mathrm{d}\mathbb{P}_1}\right| \right] < \infty,\\
+ \infty, \qquad &\text{otherwise,}
\end{cases}
\end{equation}
define divergences on the set of probability measures equivalent to $\widetilde{\mathbb{P}}$. Henceforth, these quantities shall be called \emph{variance divergence} and \emph{log-variance divergence}, respectively.
\begin{remark}
Setting $\widetilde{\mathbb{P}} = \mathbb{P}_1$, the quantity $D^{\mathrm{Var}}_{\mathbb{P}_1}(\mathbb{P}_1 \vert \mathbb{P}_2)$ coincides with the Pearson $\chi^2$-divergence \cite{dieng2017variational,liese2006divergences} measuring the importance sampling relative error \cite{deniz2019convergence,hartmann2021nonasymptotic}, hence relating to Problem \ref{prob:variance minimisation}.  The divergence $D^{\mathrm{Var(log)}}_{\widetilde{\mathbb{P}}}$ seems to be new; it is motivated by its connections to the forward-backward SDE formulation of optimal control (see Problem  \ref{prob:FBSDE}), as will be explained in Section \ref{sec:BSDEs}. Let us already mention that inserting the $\log$ in \eqref{eq:var divergence} to obtain \eqref{eq:var log divergence} has the potential benefit of making sample based estimation more robust in high dimensions (see Section \ref{sec:products}). Furthermore, we point the reader to Proposition \ref{prop:log var and KL} revealing close connections between $D^{\mathrm{Var(log)}}_{\widetilde{\mathbb{P}}}$ and the relative entropy.
\end{remark}
Using \eqref{eq:var divergence} and \eqref{eq:var log divergence} with $\widetilde{\mathbb{P}} = \mathbb{P}^v$,  
we obtain two additional families of losses, indexed by $v  \in \mathcal{U}$:
    
\begin{definition}[Variance and log-variance losses]
\label{def_var_loss}
For $v  \in \mathcal{U}$, the \emph{variance loss} is given by 
\begin{equation}
\label{def_variance_loss}
\mathcal{L}_{\text{Var}_v}(u) = {\Var}_{\P^v}\left( \frac{\mathrm d \Q}{\mathrm d \P^u} \right), \qquad  u \in \mathcal{U},
\end{equation}
and the \emph{log-variance loss} by
\begin{equation}
\label{def_log_variance_loss}
\mathcal{L}^{\log}_{\text{Var}_v}(u) = {\Var}_{\P^v}\left( \log\frac{ \mathrm d \Q}{\mathrm d \P^u} \right), \qquad  u \in \mathcal{U},
\end{equation}
whenever $\mathbb{E}_{{\mathbb{P}^v}}\left[\left| \frac{\mathrm{d}\mathbb{Q}}{\mathrm{d}\mathbb{P}^u}\right| \right] < \infty$ or $\mathbb{E}_{{\mathbb{P}^v}}\left[\left| \log \frac{\mathrm{d}\mathbb{Q}}{\mathrm{d}\mathbb{P}^u}\right| \right] < \infty$, respectively\footnote{These integrability conditions can readily be checked using the formulas provided in Proposition \ref{prop:variance} below.}. 
The notation ${\Var}_{\P^v}$ is to be interpreted in line with Remark \ref{rem:notation}. 
\end{definition}

By direct computations invoking Girsanov's theorem, the losses defined above admit explicit representations in terms of solutions to SDEs of the form \eqref{eq:uncontrolled SDE} and \eqref{eq:controlled SDE}. Crucially, the propositions that follow replace the expectations on $\mathcal{C}$ used in the definitions \eqref{eq:relative entropy}, \eqref{eq:cross entropy}, \eqref{eq:var divergence} and \eqref{eq:var log divergence} by expectations on $\Omega$ that are more amenable to direct probabilistic interpretation and Monte Carlo simulation (see also Remark \ref{rem:notation}). Recall that the target measure $\mathbb{Q}$ is assumed to be of the type \eqref{eq:reweighted measure}, where $\mathcal{W}$ has been defined in \eqref{eq:work}. We start with the relative entropy loss:
    
\begin{proposition}[Relative entropy loss]
\label{prop:relative entropy}
For $u \in \mathcal{U}$, let $(X_s^u)_{0 \le s \le T}$ denote the unique strong solution to \eqref{eq:controlled SDE}.
Then 
\begin{equation}
\label{eq:RE explicit}
\mathcal{L}_{\mathrm{RE}}(u) = \mathbb{E} \left[ \frac{1}{2} \int_0^T \vert u(X_s^u, s) \vert^2 \, \mathrm{d}s + \int_0^T f(X_s^u, s)\, \mathrm ds + g(X_T^u) \right] + \log \mathcal{Z}.
\end{equation}
\end{proposition}
\begin{proof}
See \cite{hartmann2012efficient,kappen2012optimal}. For the reader's convenience, we provide a self-contained proof in  
Appendix \ref{app:divergences}.
\end{proof}
\begin{remark}
Up to the constant $\log \mathcal{Z}$, the loss $\mathcal{L}_{\RE}$ coincides with the cost functional \eqref{eq:cost functional} associated to the optimal control formulation in Problem \ref{prob:optimal control}. The approach of minimising the $\KL$-divergence between $\mathbb{P}^u$ and $\mathbb{Q}$ as defined in \eqref{eq:relative entropy} is thus directly linked to the perspective outlined in Section \ref{sec:optimal control}. We refer to \cite{hartmann2012efficient,kappen2012optimal} for further details.
\end{remark}
The cross-entropy loss admits a family of representations, indexed by $v \in \mathcal{U}$:
\begin{proposition}[Cross-entropy loss]
\label{prop:cross entropy}
For $v \in \mathcal{U}$, let $(X_s^v)_{0 \le s \le T}$ denote the unique strong solution to \eqref{eq:controlled SDE}, with $u$ replaced by $v$.
Then there exists a constant $C \in \mathbb{R}$ (not depending on $u$ in the next line) such that
\begin{subequations}
\label{CE_formula}
\begin{align}
\mathcal{L}_{\CE}(u) = \frac{1}{\mathcal{Z}} \mathbb{E} \Bigg[ & \left( \frac{1}{2} \int_0^T \vert u(X^v_s, s) \vert^2 \,\mathrm{d}s - \int_0^T (u \cdot v)(X_s^v, s) \, \mathrm{d}s - \int_0^T u(X_s^v, s) \cdot \mathrm{d}W_s
\right) \\
\label{eq:CE2}
& \exp \left(- \int_0^T v(X_s^v, s) \cdot \mathrm{d}W_s - \frac{1}{2} \int_0^T \vert v(X_s^v, s) \vert^2 \, \mathrm{d}s - \mathcal{W}(X^v) \right) \Bigg] + C,
\end{align}
\end{subequations}
for all $u \in \mathcal{U}$.
\end{proposition}
\begin{proof}
See \cite{zhang2014applications} or Appendix \ref{app:divergences} for a self-contained proof.
\end{proof}
\begin{remark}
\label{remark_CE_additional_v}
The appearance of the exponential term in \eqref{eq:CE2} can be traced back to the reweighting\footnote{Note that, by slightly abusing notation, here and in the following $\P$ often denotes an arbitrary (path) measure and does not necessarily relate to the uncontrolled dynamics \eqref{eq:uncontrolled SDE}.}
\begin{equation}
D^{\CE}(\P \vert \Q) = \mathbb{E}_{\Q} \left[ \log \left( \frac{\mathrm{d}\Q}{\mathrm{d}\P} \right) \right] = \mathbb{E}_{\P^v} \left[ \log \left( \frac{\mathrm{d}\Q}{\mathrm{d}\P} \right) \frac{\mathrm{d}\Q}{\mathrm{d}\P^v}\right],
\end{equation}
recalling that $\P^v$ denotes the path measure associated to \eqref{eq:controlled SDE} controlled by $v$.
While the choice of $v$ evidently does not affect the loss function, judicious tuning may have a significant impact on the numerical performance by means of altering the statistical error for the associated estimators (see Section \ref{sec:ensemble averages}). We note that the expression \eqref{eq:RE explicit} for the relative entropy loss can similarly be augmented by an additional control $v \in \mathcal{U}$. However, Proposition \ref{prop:robustness} in Section \ref{sec:products} discourages this approach and our numerical experiments using a reweighting for the relative entropy loss have not been promising. In general, we feel that exponential terms of the form appearing in \eqref{eq:CE2} often have a detrimental effect on the variance of estimators, which should also be compared to an analysis in \cite{hartmann2021nonasymptotic}. Therefore, an important feature of both the relative entropy loss and the log-variance loss (see Proposition \ref{prop:variance}) seems to be that expectations can be taken with respect to controlled processes $(X_s^v)_{0 \le s \le T}$ without incurring exponential factors as in \eqref{eq:CE2}.
\end{remark}
    
\begin{remark}
Setting $v = 0$ leads to the simplification 
\begin{equation}
\mathcal{L}_{\CE}(u) = \frac{1}{\mathcal{Z}} \mathbb{E} \Bigg[ \left( \frac{1}{2} \int_0^T \vert u(X_s, s) \vert^2 \, \mathrm ds  - \int_0^T u(X_s, s) \cdot \mathrm{d}W_s
\right) \exp(-\mathcal{W}(X)) \Bigg] + C,
\end{equation}
where $(X_s)_{0 \le s \le T}$ solves the uncontrolled SDE \eqref{eq:uncontrolled SDE}. The quadratic dependence of  $\mathcal{L}_{\CE}$ on $u$ has been exploited in \cite{zhang2014applications} to construct efficient Galerkin-type approximations of $u^*$.
\end{remark}
Finally, we derive corresponding representations for the variance and log-variance losses:
\begin{proposition}[Variance-type losses]
\label{prop:variance}
For $v \in \mathcal{U}$, let $(X_s^v)_{0 \le s \le T}$ denote the unique strong solution to \eqref{eq:controlled SDE}, with $u$ replaced by $v$.
Furthermore, define
\begin{equation}
\label{def_Y_T}
\widetilde{Y}_T^{u,v} = - \int_0^T (u \cdot v)(X_s^v, s)\, \mathrm{d}s - \int_0^T f(X_s^v, s)\,\mathrm ds - \int_0^T u(X_s^v, s) \cdot \mathrm{d}W_s + \frac{1}{2} \int_0^T |u(X_s^v, s)|^2\,\mathrm{d}s.
\end{equation}
Then
\begin{equation}
\label{eq:var explicit}
\mathcal{L}_{\mathrm{Var}_v}(u) = \frac{1}{\mathcal{Z}^2} \,{\Var} \left(e^{\widetilde{Y}_T^{u, v} - g(X_T^v)}\right),
\end{equation}
and 
\begin{equation}
\label{eq:log var explicit}
\mathcal{L}^{\log}_{\Var_v}(u) = {\Var}\left(\widetilde{Y}_T^{u,v} - g(X_T^v)\right),
\end{equation}
for all $u \in \mathcal{U}$.
\end{proposition}
\begin{proof}
See Appendix \ref{app:divergences}.
\end{proof}
Setting $v=u$ in \eqref{eq:var explicit} recovers the importance sampling objective in \eqref{eq:importance sampling}, i.e. the variance divergence $D^{\Var}_{\P^u}$ encodes the formulation from Problem \ref{prob:variance minimisation}. See also \cite{muller2018neural,hartmann2021nonasymptotic}.

\begin{remark}
\label{remark_var_loss_additional_v}
While different choices of $v$ merely lead to distinct representations for the cross-entropy loss $\mathcal{L}_{\CE}$ according to Proposition \ref{prop:cross entropy} and Remark \ref{remark_CE_additional_v}, the variance losses $\mathcal{L}_{\mathrm{Var}_v}$ and $\mathcal{L}^{\log}_{\Var_v}$ do indeed depend on $v$. However, the property $\mathcal{L}_{\mathrm{Var}_v}(u) = 0 \iff u = u^*$ (and similarly for $\mathcal{L}^{\log}_{\Var_v}$) holds for all $v \in \mathcal{U}$, by construction.
\end{remark}
\subsection{FBSDEs and the log-variance loss}
\label{sec:BSDEs}
    
As it turns out, the log-variance loss
$\mathcal{L}^{\log}_{\Var_v}$ as computed in \eqref{eq:log var explicit} is intimately connected to the FBSDE formulation in Problem \ref{prob:FBSDE} (and we already used the notation $\widetilde{Y}_T^{u,v}$ in hindsight). Indeed, setting $v = 0$ in Proposition \ref{prop:variance} and writing 
\begin{equation}
\label{eq:var loss 0}
{\Var}\left(\widetilde{Y}_T^{u,0} - g(X_T^0)\right) = {\Var}\Big(\underbrace{\widetilde{Y}_T^{u,0} + y_0}_{=:Y^{u,0}_{T}} - g(X_T^0)\Big), 
\end{equation}
for some (at this point, arbitrary) constant $y_0 \in \mathbb{R}$, we recover the forward SDE \eqref{eq:forward SDE} from \eqref{eq:uncontrolled SDE} and the backward SDE \eqref{eq:backward SDE} from \eqref{def_Y_T} in conjunction with the optimality condition $\mathcal{L}^{\log}_{\Var_v}(u) = 0$, using also the identification $u^*(X_s,s) =: -Z_s$ suggested by \eqref{eq:YZ in terms of V}.
For arbitrary $v \in \mathcal{U}$, we similarly obtain the generalised FBSDE system
\begin{subequations}
\label{eq:general FBSDE}
\begin{alignat}{2}
\label{eq:general FBSDE forward}
\mathrm d X^v_s & = \left(b(X^v_s,s) + \sigma(X^v_s,s) v(X^v_s,s)\right) \mathrm ds + \sigma(X^v_s,s) \, \mathrm dW_s, \qquad \qquad \qquad \qquad X^v_0 && = x_0, \\
\label{eq:general FBSDE backward}
\mathrm{d}Y_s^{u^{*},v} & = -f(X^v_s,s) \, \mathrm{d}s + v(X^v_s,s) \cdot Z_s \, \mathrm{d}s + \frac{1}{2} \vert Z_s \vert^2 \, \mathrm{d}s + Z_s \cdot \mathrm{d}W_s, \qquad \qquad \,\,\,\, Y_T^{u^{*}, v} && = g(X^v_T), 
\end{alignat}
\end{subequations}
again setting 
\begin{equation}
\label{eq:shift}
Y_T^{u,v} = \widetilde{Y}_T^{u,v} + y_0.
\end{equation}
In this sense, the divergence $D^{\Var(\log)}_{\P^v}(\P^u|\Q)$ encodes the dynamics \eqref{eq:general FBSDE}. 
Let us again insist on the fact that by construction the solution $(Y_s,Z_s)_{0 \le s \le T}$ to \eqref{eq:general FBSDE} does not depend on $v \in \mathcal{U}$ (the contribution $\sigma(X^v_s,s) v(X^v_s,s) \, \mathrm ds$ in \eqref{eq:general FBSDE forward} being compensated for by the term $v(X^v_s,s) \cdot Z_s \, \mathrm{d}s$ in \eqref{eq:general FBSDE backward}), whereas clearly $(X_s^v)_{0 \le s \le T}$ does. When $u^*(X_s,s)=-Z_s$ is approximated in an iterative manner (see Section \ref{sec:computational_aspects}), the choice $v = u$ is natural as it amounts to applying the currently obtained estimate for the optimal control to the forward process \eqref{eq:general FBSDE forward}.  In this context, the system \eqref{eq:general FBSDE} was put forward in  \cite[Section III.B]{hartmann2019variational}. The bearings of appropriate choices for $v$ will be further discussed in Section \ref{sec:finite sample properties}.
    
It is instructive to compare the expression \eqref{eq:var loss 0} for the log-variance loss to the `moment loss'
\begin{equation}
\label{eq:moment}
\mathcal{L}_{\mathrm{moment}} (u,y_0) = \mathbb{E} \left[ \left( Y_T^{u,0}(y_0) - g(X_T^0) \right)^2 \right]
\end{equation}
suggested in 
\cite{weinan2017deep,han2018solving} in the context of solving more general nonlinear parabolic PDEs\footnote{We have employed the notation $Y_T^{u,0}(y_0)$ in order to stress the dependence on $y_0$ through \eqref{eq:shift}.}. More generally, we can define
\begin{equation}
\label{eq:moment_v}
\mathcal{L}_{\mathrm{moment}_v} (u,y_0) = \mathbb{E} \left[ \Big( Y_T^{u,v}(y_0) - g(X_T^v) \Big)^2 \right]
\end{equation}
as a counterpart to the expression \eqref{eq:log var explicit}.
Note that unlike the losses considered so far, the moment losses depend on the additional parameter $y_0 \in \mathbb{R}$, which has implications in numerical implementations. Also, these losses do not admit a straightforward interpretation in terms of divergences between path measures. As we show in Proposition \ref{prop: equivalence moment log-variance}, algorithms based on $\mathcal{L}_{\mathrm{moment}_v}$ are in fact equivalent to their counterparts based on $\mathcal{L}^{\log}_{\Var_v}$ in the limit of infinite batch size when $y_0$ is chosen optimally or when the forward process is controlled in a certain way. We already anticipate that optimising an additional parameter $y_0$ can slow down convergence towards the solution $u^*$ considerably (see Section \ref{sec:numerics}).

\begin{remark}
\label{rem:shift V}
Reversing the argument, the log-variance loss can be obtained from \eqref{eq:moment} by replacing the second moment by the variance and using the translation invariance \eqref{eq:var loss 0} to remove the dependence on $y_0$. The fact that this procedure leads to a viable loss function (i.e. satisfying $\mathcal{L}(u)=0 \iff u=u^*$) can be traced back to the fact that the Hamilton-Jacobi PDE \eqref{eq:HJB 1st} is itself translation invariant (i.e. it remains unchanged under the transformation $V \mapsto V + \mathrm{const}$). 
Following this argument, the log-variance loss can be applied for solving more general PDEs of the form \eqref{eq:nonlinear PDE} in the case when $h$ depends on $V$ only through $\nabla V$. Furthermore, our interpretation in terms of divergences between probability measures on path space remains valid, at least in the case when $\sigma$ is constant (in the following we let $\sigma = I_{d \times d}$ for simplicity)\footnote{For more general diffusion coefficients, we can make similar arguments considering measures on the path space associated to $(W_t)_{t\ge 0}$, however departing slightly from the set-up in this paper.}. 
Indeed, denoting as before the path measure associated to \eqref{eq:generalised forward} by $\mathbb{P}$, defining the target $\mathbb{Q}$ via $\tfrac{\mathrm{d}\mathbb{Q}}{\mathrm{d}\mathbb{P}} \propto e^{-g}$, and introducing the neural network approximation $\widetilde{u} \approx -\sigma^\top\nabla V$, the backward SDE \eqref{eq:generalised backward} induces a $\widetilde{u}$-dependent path measure $\mathbb{P}^{\widetilde{u}}$,
\begin{equation} \frac{\mathrm{d}\mathbb{P}^{\widetilde{u}}}{\mathrm{d}\mathbb{P}}(X) \propto \exp \left( \int_0^T h(X_s,s,-\widetilde{u}(X_s,s)) \, \mathrm{d}s -\int_0^T \widetilde{u}(X_s,s)\cdot \left(b(X_s,s)\, \mathrm{d}s - \mathrm{d}X_s \right) \right),
\end{equation}
assuming that the right-hand side is $\mathbb{P}$-integrable.
Using $Z \approx -\widetilde{u}$ in \eqref{eq:generalised backward} and denoting the corresponding process by $Y^{\widetilde{u}}$, we then obtain 
\begin{equation}
\label{eq:var general}
\mathcal{L}(\widetilde{u}) = {\Var}_{\mathbb{P}}\left( \log \frac{\mathrm{d}\mathbb{Q}}{\mathrm{d}\mathbb{P}^{\widetilde{u}}}\right) = \Var \left( Y^{\widetilde{u}}_T - g(X_T)\right)
\end{equation}
as an implementable loss function, with straightforward modifications to \eqref{eq:FBSDE generalised} when $\mathbb{P}$ is replaced by $\mathbb{P}^v$, see \eqref{eq:general FBSDE}. Note, however, that in general the vector field $\widetilde{u}$ does not lend itself to a straightforward interpretation in terms of a control problem.
The PDEs treated in \cite{weinan2017deep,han2018solving} do not possess the shift-invariance property (that is, $h$ depends on $V$), and thus the vanishing of \eqref{eq:var general} does not characterise the solution to the PDE \eqref{eq:nonlinear 1st} uniquely (not even up to additive constants). Uniqueness may be restored by including appropriate terms in \eqref{eq:var general} enforcing the terminal condition \eqref{eq:nonlinear terminal}. Theoretical and numerical properties of such extensions may be fruitful directions for future work.   
\end{remark}
    
\subsection{Algorithmic outline and empirical estimators}
\label{sec:ensemble averages}
In order to motivate the theoretical analysis in the following sections, let us give a brief overview of algorithmic implementations based on the loss functions developed so far. We refer to Section \ref{sec:computational_aspects} for a more detailed account. Recall that by the construction outlined in Section \ref{sec:divergences}, the solution $u^*$ as defined in \eqref{eq:u_star} is characterised as the global minimum of $\mathcal{L}$, where $\mathcal{L}$ represents a generic loss function. Assuming a parametrisation $\mathbb{R}^p \ni \theta \mapsto u_{\theta}$ (derived from, for instance, a Galerkin truncation or a neural network), we apply gradient-descent type methods to the function $\theta \mapsto \mathcal{L}(u_\theta)$, relying on the explicit expressions obtained in Propositions \ref{prop:relative entropy}, \ref{prop:cross entropy} and \ref{prop:variance}. It is an important aspect that those expressions involve expectations that need to be estimated on the basis of ensemble averages. To approximate the loss $\mathcal{L}_{\RE}$, for instance, we use the estimator    
\begin{equation}
\label{eq:RE estimator}
\widehat{\mathcal{L}}_{\RE}^{(N)} (u)= \frac{1}{N}  \sum_{i=1}^N \left[ \frac{1}{2} \int_0^T \vert u(X_s^{u,(i)},s) \vert^2 \, \mathrm{d}s + \int_0^T f(X_s^{u,(i)}, s)\, \mathrm ds + g(X_T^{u,(i)})\right],
\end{equation}
where $(X^{u,(i)}_s)_{0 \le s \le T}$, $i=1, \ldots, N$ denote independent realisations of the solution to \eqref{eq:controlled SDE}, and $N \in \mathbb{N}$ refers to the batch size. The estimators $\widehat{\mathcal{L}}_{\CE}^{(N)}(u)$, $\widehat{\mathcal{L}}_{\Var}^{(N)}(u)$, $\widehat{\mathcal{L}}_{\Var}^{\log,(N)}(u)$ and $\widehat{\mathcal{L}}^{ (N)}_{\mathrm{moment}_v}(u,y_0)$ are constructed analogously, i.e. the estimator for the cross-entropy loss is given by
\begin{subequations}
\label{eq:CE estimator}
\begin{align}
\widehat{\mathcal{L}}^{(N)}_{\CE,v}(u) = \frac{1}{N} \sum_{i=1}^N \Bigg[ & \left( \frac{1}{2} \int_0^T \vert u(X^{v,(i)}_s, s) \vert^2 \,\mathrm{d}s - \int_0^T (u \cdot v)(X_s^{v,(i)}, s) \, \mathrm{d}s - \int_0^T u(X^{v,(i)}, s) \cdot \mathrm{d}W^{(i)}_s
\right) \\
& \exp \left(- \int_0^T v(X_s^{v,(i)}, s) \cdot \mathrm{d}W^{(i)}_s - \frac{1}{2} \int_0^T \vert v(X_s^{v,(i)}, s) \vert^2 \, \mathrm{d}s - \mathcal{W}(X^{v,(i)}) \right) \Bigg],
\end{align}
\end{subequations}
the estimator for the variance loss is given by
\begin{equation}
\label{eq:var estimator} \widehat{\mathcal{L}}^{(N)}_{\mathrm{Var}_v}(u) = \frac{1}{N-1}\sum_{i=1}^N \left(e^{\widetilde{Y}_T^{u, v, (i)} - g(X_T^{v,(i)})} - \left(\overline{e^{\widetilde{Y}_T^{u, v} - g(X_T^v)}}\right) \right)^2,
\end{equation}
the estimator for the log-variance loss by
\begin{equation}
\label{eq:log var estimator}
\widehat{\mathcal{L}}^{\log (N)}_{\Var_v}(u) = \frac{1}{N-1} \sum_{i=1}^N \left(\widetilde{Y}_T^{u,v,(i)} - g(X_T^{v,(i)}) -\left(\overline{\widetilde{Y}_T^{u,v} - g(X_T^{v})}\right) \right)^2,
\end{equation}
and the estimator for the moment loss by
\begin{equation}
\label{eq:moment estimator}
\widehat{\mathcal{L}}^{ (N)}_{\mathrm{moment}_v}(u,y_0) = \frac{1}{N} \sum_{i=1}^N \left(\widetilde{Y}_T^{u,v,(i)} + y_0 - g(X_T^{v,(i)}) \right)^2.
\end{equation}
In the previous displays, the overline denotes an empirical mean, for example
\begin{equation}
\overline{\widetilde{Y}_T^{u,v} - g(X_T^{v})} = \frac{1}{N} \sum_{i=1}^N  \left(\widetilde{Y}_T^{u,v,(i)} - g(X_T^{v,(i)}) \right),
\end{equation}
and $(W_t^{(i)})_{t \ge 0}$, $i=1,\ldots, N$ denote independent Brownian motions associated to $(X_t^{u,(i)})_{t \ge 0}$.
By the law of large numbers, the convergence $\widehat{\mathcal{L}}^{(N)} (u) \rightarrow \mathcal{L}(u)$ holds almost surely up to additive and multiplicative constants\footnote{More precisely, $\widehat{\mathcal{L}}_{\RE}^{(N)} (u) \rightarrow \mathcal{L}_{\RE}(u)- \log \mathcal{Z}$ and $\widehat{\mathcal{L}}^{(N)}_{\CE,v}(u) \rightarrow \mathcal{Z}(\mathcal{L}_{\CE}(u) - C)$. The fact that the estimators $\widehat{\mathcal{L}}_{\RE}^{(N)}$ and $\widehat{\mathcal{L}}_{\CE, v}^{(N)}$ do not depend on the intractable constants $\mathcal{Z}$ and $C$ is crucial for the implementability of the associated methods.}, but as we show in Section \ref{sec:numerics}, the fluctuations for finite $N$ play a crucial role for the overall performance of the method. The variance associated to empirical estimators will hence be analysed in Section \ref{sec:finite sample properties}. 
\begin{remark}
The estimators introduced in this section are standard, and more elaborate constructions, for instance involving control variates \cite[Section 4.4.2]{robert2013monte}, can be considered to reduce the variance. We leave this direction for future work. It is noteworthy, however, that the log-variance estimator \eqref{eq:log var estimator} appears to act as a control variate in natural way, see Propositions \ref{prop:log var and KL} and \ref{prop: equivalence moment log-variance} and Remark \ref{remark: control variate}.
\end{remark}
\begin{remark}
Note that the estimator $\widehat{\mathcal{L}}^{(N)}_{\CE,v}$ depends on $v \in \mathcal{U}$, in contrast to its target $\mathcal{L}_{\CE}$; in other words, the limit $\lim_{N \rightarrow \infty} \widehat{\mathcal{L}}^{(N)}_{\CE,v}(u)$ does not depend on $v$. This contrasts the pairs $(\widehat{\mathcal{L}}^{(N)}_{\mathrm{Var}_v},\mathcal{L}_{\mathrm{Var}_v}) $ and $(\widehat{\mathcal{L}}^{\log,(N)}_{\mathrm{Var}_v},\mathcal{L}^{\log}_{\mathrm{Var}_v})$, see also Remark \ref{remark_CE_additional_v}.
\end{remark}
We provide a sketch of the algorithmic procedure in Algorithm \ref{alg:sketch}. Clearly, choosing different loss functions (and corresponding estimators) at every gradient step as indicated leads to viable algorithms. In particular, we have in mind the option of adjusting the forward control $v \in \mathcal{U}$ using the current approximation $u_\theta$. More precisely, denoting by $u_\theta^{(j)}$ the approximation at the $j^{\text{th}}$ step, it is reasonable to set $v= u^{(j)}_\theta$ in the iteration yielding $u^{(j+1)}_\theta$. In the remainder of this paper, we will focus on this strategy for updating $v$, leaving differing schemes for future work. 
\par\bigskip
\begin{algorithm}[H]
\label{alg:sketch}
\SetAlgoLined
Choose a parametrisation $\mathbb{R}^p \ni \theta \mapsto u_{\theta}$.\\
Initialise $u_\theta$ (with a parameter vector $\theta \in \R^p$). \\
Choose an optimisation method $\mathit{descent}$, a batch size $N \in \mathbb{N}$ and a learning rate $\eta > 0$.
\\
\Repeat{ convergence}{
Choose a loss function $\mathcal{L}$ and a corresponding estimator $\widehat{\mathcal{L}}^{(N)}$.\\
Compute $\widehat{\mathcal{L}}^{(N)}(u_\theta)$ according to either \eqref{eq:RE estimator}, \eqref{eq:CE estimator}, \eqref{eq:var estimator}, \eqref{eq:log var estimator} or \eqref{eq:moment estimator}.  \\
Compute $\nabla_{\theta}\widehat{\mathcal{L}}^{(N)}(u_\theta)$ using automatic differentiation.\\
Update parameters: $\theta \gets \theta - \eta \, \mathit{descent}(\nabla_\theta\widehat{\mathcal{L}}^{(N)}(u_\theta))$.
}
\caption{Approximation of $u^*$}
\KwResult{$u_\theta \approx u^*$.}
\end{algorithm}
    
\section{Equivalence properties in the limit of infinite batch size}
\label{sec: infinite batch size}
    
In this section we will analyse some of the properties of the losses defined in Section \ref{sec:divergences}, not taking into account the approximation by ensemble averages described in Section \ref{sec:ensemble averages}. In other words, the results in this section are expected to be valid when the batch size $N$ used to compute the estimators $\widehat{\mathcal{L}}^{(N)}$ is sufficiently large. The derivatives relevant for the gradient-descent type methodology described in Section \ref{sec:ensemble averages} can be computed as follows,
\begin{equation}
\label{eq: phi <-> gradient}
\frac{\partial}{\partial \theta_i} \mathcal{L}(u_\theta) = \frac{\delta}{\delta u} \mathcal{L}(u;\phi_i) \Big\vert_{u = u_\theta}, \qquad \phi_i = \frac{\partial u_\theta}{\partial \theta_i},
\end{equation}
where $\frac{\delta}{\delta u} \mathcal{L}(u;\phi)$ denotes the G{\^a}teaux derivative in direction $\phi$. We recall its definition \cite[Section 5.2]{siddiqi1986functional}:
\begin{definition}[G{\^a}teaux derivative]
\label{def:Gateaux}
Let $u \in \mathcal{U}$ and $\phi \in C_b^1(\mathbb{R}^d \times [0,T]; \mathbb{R}^d)$. A loss function $\mathcal{L}:\mathcal{U} \to \mathbb{R}$ is called \emph{G{\^a}teaux-differentiable} at $u$, if, for all $\phi \in C_b^1(\mathbb{R}^d \times [0,T]; \mathbb{R}^d)$, the real-valued function $\varepsilon \mapsto \mathcal{L}(u + \varepsilon \phi)$ is differentiable at $\varepsilon = 0$. In this case we define the \emph{G{\^a}teaux derivative in direction $\phi$} to be
\begin{equation}
\frac{\delta}{\delta u} \mathcal{L}(u; \phi) := \frac{\mathrm d}{\mathrm d \varepsilon}\Big|_{\varepsilon = 0}\mathcal{L}(u + \varepsilon \phi).
\end{equation}
\end{definition}
\begin{remark}
The functions $\phi_i$ defined in \eqref{eq: phi <-> gradient} depend on the chosen parametrisation for $u$. In the case when a Galerkin truncation is used, $
u_\theta = \sum_{i} \theta_i \alpha_i,$
these coincide with the chosen ansatz functions (i.e. $\phi_i = \alpha_i$). Concerning neural networks, the family $(\phi_i)_i$ reflects the choice of the architecture, the function $\phi_i$ encoding the response to a a change in the $i^{\text{th}}$ weight. For convenience, we will throughout work under the assumption (implicit in Definition \ref{def:Gateaux}) that the functions $\phi_i$ are bounded, noting however that this could be relaxed with additional technical effort. Furthermore, note that Definition \ref{def:Gateaux} extends straightforwardly to the estimator versions $\widehat{\mathcal{L}}^{(N)}$.
\end{remark}

The following result shows that algorithms based on $\frac{1}{2}\mathcal{L}_{\Var_v}^{\log}$ and $\mathcal{L}_{\RE}$ behave equivalently in the limit of infinite batch size, provided that the update rule $v=u$ for the log-variance loss is applied (see the discussion towards the end of Section \ref{sec:ensemble averages}), and that \emph{`all other things being equal'}, for instance in terms of network architecture and choice of optimiser. Furthermore, we provide an analytical expression for the gradient for future reference.
\begin{proposition}[Equivalence of log-variance loss and relative  entropy loss]
\label{prop:log var and KL}
Let $u,v \in \mathcal{U}$ and $\phi \in C_b^1(\mathbb{R}^d \times [0,T] ; \mathbb{R}^d)$. Then $\mathcal{L}_{\Var_v}^{\log}$ and $\mathcal{L}_{\mathrm{RE}}$ are G{\^a}teaux-differentiable at $u$ in direction $\phi$. Furthermore, 
\begin{equation}
\label{eq:exact gradient}
\frac{1}{2}\left(  \frac{\delta}{\delta u} \mathcal{L}_{\Var_v}^{\log}(u;\phi) \right)\Big \vert_{v = u} =  \frac{\delta}{\delta u} \mathcal{L}_{\mathrm{RE}}(u;\phi) = {\E}\left[\left(g(X_T^u) - \widetilde{Y}_T^{u, u}\right) \int_0^T \phi(X_s^u,s)\cdot \mathrm dW_s \right].
\end{equation}
\end{proposition}
\begin{remark}
\label{rem:log var and KL}
Proposition \ref{prop:log var and KL} extends the connection between the cost functional \eqref{eq:cost functional} and the FBSDE formulation \eqref{eq:FBSDE} exposed in Theorem \ref{thm:connections}. Indeed, the Problems \ref{prob:optimal control} and \ref{prob:FBSDE} do not only agree on identifying the solution $u^*$; it is also the case that the gradients of the corresponding loss functions agree for $u \neq u^*$.
    
Moreover, it is instructive to compare the expressions \eqref{eq:RE explicit} and \eqref{eq:log var explicit} (or their sample based variants \eqref{eq:RE estimator} and \eqref{eq:log var estimator}). Namely, computing the derivatives associated to the relative entropy loss entails differentiating both the SDE-solution $X^u$ as well as $f$ and $g$, determining the running and terminal costs. Perhaps surprisingly, the latter is not necessary for obtaining the derivatives of the log-variance loss, opening the door for gradient-free implementations.
\end{remark}
    
\begin{proof}[Proof of Proposition \ref{prop:log var and KL}]
We present a heuristic argument based on the perspective introduced in Section \ref{sec:divergences} and refer to Appendix \ref{app:infinite batch size} for a rigorous proof.

For fixed $\P \in \mathcal{P}(\mathcal{C})$, let us consider perturbations $\P + \varepsilon \mathbb{U}$, where $\mathbb{U}$ is a signed measure with $\mathbb{U}(\mathcal{C}) = 0$. Assuming sufficient regularity, we then expect
\begin{equation}
\label{eq:RE abstract gradient}
\frac{\mathrm{d}}{\mathrm{d}\varepsilon} \Big|_{\varepsilon = 0} D^{\RE} (\P + \varepsilon \U | \Q) = 
\frac{\mathrm{d}}{\mathrm{d}\varepsilon} \Big|_{\varepsilon = 0} {\E}_{\P} \left[ \log \left( \frac{\mathrm{d}(\P + \varepsilon \U)}{ \mathrm{d}\Q} \right) \frac{\mathrm{d}(\P + \varepsilon \U)}{\mathrm{d}\P}\right]
= \underbrace{{\E}_{\P} \left[ \frac{\mathrm{d} \U}{\mathrm{d}\P}\right]}_{=0} + {\E}_{\P} \left[ \log \left( \frac{\mathrm{d}\P}{\mathrm{d}\Q} \right) \frac{\mathrm{d}\U}{\mathrm{d}\P}\right],
\end{equation}
where the first term on the right-hand side vanishes because of $\mathbb{U}(\mathcal{C}) = 0$. Likewise,
\begin{subequations}
\label{eq:abstract gradient computation}
\begin{align}
\frac{\mathrm{d}}{\mathrm{d}\varepsilon} \Big|_{\varepsilon = 0} D^{\mathrm{Var(log)}}_{\widetilde{\mathbb{P}}}(\mathbb{P} + \varepsilon \U \vert \Q)  & = 
\frac{\mathrm{d}}{\mathrm{d}\varepsilon} \Big|_{\varepsilon = 0} \left( {\E}_{\widetilde{\P}} \left[\log^2 \left( \frac{\mathrm{d}(\P + \varepsilon \U)}{\mathrm{d}\Q}\right) \right] -
{\E}_{\widetilde{\P}} \left[\log \left( \frac{\mathrm{d}(\P + \varepsilon \U)}{\mathrm{d}\Q}\right) \right]^2
\right)
\\
\label{eq:log var abstract gradient}
& = 2 \,{\E}_{\widetilde{\P}} \left[ \log \left( \frac{\mathrm{d}\P}{\mathrm{d}\Q}\right) \frac{\mathrm{d}\U}{\mathrm{d}\P}\right] - 2\, {\E}_{\widetilde{\P}} \left[ \log \left( \frac{\mathrm{d}\P}{\mathrm{d}\Q}\right)\right] {\E}_{\widetilde{\P}} \left[ \frac{\mathrm{d}\U}{\mathrm{d}\P}\right]. 
\end{align}
\end{subequations}
For $\widetilde{\P} = \P$, the second term in \eqref{eq:log var abstract gradient} vanishes (again, because of $\U(\mathcal{C}) = 0$), and hence \eqref{eq:log var abstract gradient} agrees with \eqref{eq:RE abstract gradient} up to a factor of $2$.
\end{proof}
\begin{remark}[Local minima]
It is interesting to note that \eqref{eq:abstract gradient computation} can be expressed as 
\begin{equation}
\frac{\mathrm{d}}{\mathrm{d}\varepsilon} \Big|_{\varepsilon = 0} D^{\mathrm{Var(log)}}_{\widetilde{\mathbb{P}}}(\mathbb{P} + \varepsilon \U \vert \Q) = 2\, \mathrm{Cov}_{\widetilde{\P}} \left(\log \frac{\mathrm{d}\P}{\mathrm{d}\Q}, \frac{\mathrm{d}\U}{\mathrm{d}\P} \right). 
\end{equation}
In particular, the derivative is zero for all $\U$ with $\U(\mathcal{C}) = 0$ if and only if $\P = \Q$. In other words, we expect the loss landscape associated to losses based on the log-variance divergence to be free of local minima where the optimisation procedure could get stuck. A more refined analysis concerning the relative entropy loss can be found in \cite{lie2016convexity}.
\end{remark}
In the following proposition, we gather results concerning the moment loss $\mathcal{L}_{\mathrm{moment}_v}$ defined in \eqref{eq:moment}. The first statement is analogous to Proposition \ref{prop:log var and KL} and shows that $\mathcal{L}_{\mathrm{moment}_v}$ and $\mathcal{L}^{\log}_{\Var_v}$  are equivalent in the infinite batch size limit, provided that the update strategy $v=u$ is employed. The second statement deals with the alternative $v \neq u$. In this case, $y_0 = -\log \mathcal{Z}$ (i.e. finding the optimal $y_0$ according to Theorem \ref{thm:connections}) is necessary for $\mathcal{L}_{\mathrm{moment}_v}$ to identify the correct $u^*$. Consequently, approximation of the optimal control will be inaccurate unless the parameter $y_0$ is determined without error.
\begin{proposition}[Properties of the moment loss]
\label{prop: equivalence moment log-variance}
Let $u,v \in \mathcal{U}$ and $y_0 \in \mathbb{R}$. Then the following holds:
\begin{enumerate}
\item 
The losses  $\mathcal{L}_{\mathrm{moment},v}(\cdot, y_0)$ and $\mathcal{L}_{\Var_v}^{\log}$ are G{\^a}teaux-differentiable at $u$, and
\begin{equation}
\label{eq:moment log var gradient}
\left(\frac{\delta }{\delta u}\mathcal{L}_{\mathrm{moment}_v}(u, y_0;\phi) \right)\Big |_{v=u} = \left(\frac{\delta }{\delta u} \mathcal{L}^{\log}_{\Var_v}(u;\phi) \right) \Big |_{v=u}
\end{equation}
holds for all $\phi \in C_b^1(\mathbb{R}^d \times [0,T]; \mathbb{R}^d)$. In particular, \eqref{eq:moment log var gradient} is zero at $u = u^*$, independently of $y_0$. 
\item
If $v \ne u$, then
\begin{equation}
\frac{\delta }{\delta u}\mathcal{L}_{\mathrm{moment}_v}(u, y_0;\phi)= 0
\end{equation}
holds for all $\phi\in C_b^1(\mathbb{R}^d \times [0,T]; \mathbb{R}^d)$ if and only if $u = u^*$ and $y_0 = -\log \mathcal{Z}$.
\end{enumerate}
\end{proposition}
\begin{proof}
The proof can be found in Appendix \ref{app:infinite batch size}.
\end{proof}
    
\begin{remark}[Control variates]
\label{remark: control variate}
Inspecting the proofs of Propositions \ref{prop:log var and KL} and \ref{prop: equivalence moment log-variance}, we see that the identities \eqref{eq:exact gradient} and \eqref{eq:moment log var gradient} rest on the vanishing of terms of the form
$
\beta \, {\E} \left[\int_0^T \phi(X_s^u,s) \cdot \mathrm{d}W_s \right],
$
where $\beta = -y_0$ for the moment loss and $\beta = - \E\left[g(X_T^u) - \widetilde{Y}^{u,u}_T\right]$ for the log-variance loss. The corresponding Monte Carlo estimators (see Section \ref{sec:ensemble averages}) hence include terms that are zero in expectation and act as control variates \cite[Section 4.4.2]{robert2013monte}. Using the explicit expression for the derivative in \eqref{eq:exact gradient}, the optimal value for $\beta$ in terms of variance reduction is given by
\begin{subequations}
\begin{align}
\beta^* & = - \frac{\operatorname{Cov}\left(\left( g(X_T^u) - \widetilde{Y}_T^{u,u}  \right) \int_0^T \phi(X_s^u,s) \cdot \mathrm{d}W_s, \int_0^T \phi(X_s^u,s) \cdot \mathrm{d}W_s\right)}{\Var\left( \int_0^T \phi(X_{s}^{u},s) \cdot \mathrm{d}W_s \right)}\\
&= - \E\left[ g(X_T^u) - \widetilde{Y}_T^{u,u} \right] - \frac{\operatorname{Cov}\left(  g(X_T^u) - \widetilde{Y}_T^{u,u}, \left(\int_0^T \phi(X_s^u,s) \cdot \mathrm{d}W_s\right)^2 \right)}{\E \left[ \left(\int_0^T \phi(X^u_s,s) \cdot \mathrm{d}W_s\right)^2\right]},
\end{align}
\end{subequations}
which splits into a $\phi$-independent (i.e. shared across network weights) and a $\phi$-dependent (i.e. weight-specific) term. The $\phi$-independent term is reproduced in expectation by the log-variance estimator. Numerical evidence suggests that the $\phi$-dependent term is often small and fluctuates around zero, but implementations that include this contribution (based on Monte Carlo estimates) hold the promise of further variance reductions. We note however that determining a control variate for every weight carries a significant computational overhead and that Monte Carlo errors need to be taken into account. Finally, if $y_0$ in the moment loss differs greatly from  $- \E\left[ g(X_T^u) - \widetilde{Y}_T^{u,u} \right]$, we expect the corresponding variance to be large, hindering algorithmic performance. In our follow-up paper \cite{richter2020vargrad}, we have provided a more detailed analysis of the connections between the log-variance divergences and variance reduction techniques in the context of computational Bayesian inference.
\end{remark}
    
\section{Finite sample properties and the variance of estimators}
    
\label{sec:finite sample properties}

In this section we investigate properties of the sample versions of the losses as outlined in Section \ref{sec:ensemble averages} and, in particular, study their variances and relative errors. We will highlight two different types of robustness, both of which prove significant for convergence speed and stability concerning  practical implementations of Algorithm \ref{alg:sketch}, see the numerical experiments in Section \ref{sec:numerics}.
    
\subsection{Robustness at the solution $u^*$}
    
By construction, the optimal control solution $u^*$ represents the global minimum of all considered losses. Consequently,  the associated directional derivatives vanish at $u^*$, i.e.
\begin{equation}
\frac{\delta}{\delta u}\Big|_{u=u^*}\mathcal{L}(u; \phi) = 0,
\end{equation}
for all $\phi \in C_b^1(\mathbb{R}^d \times [0,T]; \mathbb{R}^d)$.
A natural question is whether similar statements can be made with respect to the corresponding Monte Carlo estimators. We make the following definition.
    
\begin{definition}[Robustness at the solution $u^*$]
\label{def:robust u}
We say that an estimator $\widehat{\mathcal{L}}^{(N)}$ is \emph{robust at the solution $u^*$} if 
\begin{equation}
\Var\left(\frac{\delta}{\delta u}\Big|_{u=u^*}\widehat{\mathcal{L}}^{(N)}(u; \phi) \right) = 0,
\end{equation}
for all $\phi \in C_b^1(\mathbb{R}^d \times [0,T]; \mathbb{R}^d)$ and $N \in \mathbb{N}$.
\end{definition}
\begin{remark}
Robustness at the solution $u^*$ implies that fluctuations in the gradient due to Monte Carlo errors are suppressed close to $u^*$, facilitating accurate approximation. Conversely, if robustness at $u^*$ does not hold, then the relative error (i.e. the Monte Carlo error relative to the size of the gradients \eqref{eq: phi <-> gradient}) grows without bounds near $u^*$, potentially incurring instabilities of the gradient-descent type scheme. We refer to Figure \ref{rel_error_gradients_ma} and the corresponding discussion for an illustration of this phenomenon.
\end{remark}    

\begin{proposition}[Robustness and non-robustness at $u^*$] 
\label{prop:robustness at u}
The following holds:
\begin{enumerate}
\item
\label{it:u log variance}
The variance estimator $\widehat{\mathcal{L}}^{(N)}_{\Var_v}$ and the log-variance estimator $\widehat{\mathcal{L}}^{\log (N)}_{\Var_v}$ are robust at $u^*$, for all $v \in \mathcal{U}$.
\item
\label{it:u moment}
For all $v \in \mathcal{U}$, the moment estimator $\widehat{\mathcal{L}}^{(N)}_{\operatorname{moment}_v}(\cdot,y_0)$ is robust at $u^*$, i.e.
\begin{equation}
\Var\left(\frac{\delta}{\delta u}\Big|_{u=u^*}\widehat{\mathcal{L}}_{\mathrm{moment}_v}^{(N)}(u, y_0; \phi) \right) = 0,\qquad \text{for all} \,\ \phi \in  C_b^1(\mathbb{R}^d \times [0,T]; \mathbb{R}^d),
\end{equation}
if and only if $y_0 = - \log \mathcal{Z}$.
\item
\label{it:u relative entropy}
The relative entropy estimator $\widehat{\mathcal{L}}_{\RE}^{(N)}$ is not robust at $u^*$. More precisely, for $\phi \in C_b^1(\mathbb{R}^d \times [0,T]; \mathbb{R}^d)$, 
\begin{equation}
\label{eq:variance RE loss}
\Var\left(  \frac{\delta}{\delta u}\Big|_{u=u^*}\widehat{\mathcal{L}}_{\RE}^{(N)}(u; \phi) \right) = \frac{1}{N} \mathbb{E} \left[ \int_0^T \vert(\nabla u^*)^\top(X_s^{u^*},s)   A_s\vert^2 \,\mathrm ds \right],
\end{equation}
where $(A_s)_{0 \le s \le T}$ denotes the unique strong solution to the SDE
\begin{equation}
\label{eq:A equation}
\mathrm{d}A_s = (\sigma\phi)(X_s^{u^*},s) \, \mathrm{d}s +
\left[(\nabla b + \nabla (\sigma u^{*}))(X_s^{u^*},s)\right]^\top  A_s \, \mathrm{d}s + A_s \cdot \nabla \sigma(X_s^{u^*},s)\, \mathrm{d}W_s, \qquad A_0 = 0.
\end{equation}
\item
\label{it:u CE}
For all $v \in \mathcal{U}$, the cross-entropy estimator  $\widehat{\mathcal{L}}^{(N)}_{\CE, v}$ is not robust at $u^*$.
\end{enumerate}
\end{proposition}
\begin{remark}
The fact that robustness of the moment estimator at $u^*$ requires $y_0 = -\log \mathcal{Z}$ might lead to instabilities in practice as this relation is rarely satisfied exactly. Note that the variance of the relative entropy estimator at $u^*$ depends on $\nabla u^*$. We thus expect instabilities in metastable settings, where often this quantity is fairly large. For numerical confirmation, see Figure \ref{rel_error_gradients_ma} and the related discussion.
\end{remark}    
\begin{proof}
For illustration, we show the robustness of the log-variance estimator $\widehat{\mathcal{L}}^{\log (N)}_{\Var_v}$. The remaining proofs are deferred to Appendix \ref{app:products}. 
By a  straightforward calculation (essentially equivalent to \eqref{eq:calculation variance loss} in Appendix \ref{app:divergences}), we see that
\begin{subequations}
\begin{align}
\frac{\delta}{\delta u} 
\widehat{\mathcal{L}}^{\log (N)}_{\Var_v}(u;\phi) & =  \frac{2}{N-1} \sum_{i=1}^N \left[ \left(  g\left(X_T^{v,(i)}\right)-\widetilde{Y}_T^{u, v,(i)}  \right) \frac{\delta \widetilde{Y}_T^{u,v,(i)}}{\delta u} (u;\phi) \right]
\\
& - \frac{2}{N(N-1)} \sum_{i=1}^N \left[ \left(  g\left(X_T^{v,(i)}\right)-\widetilde{Y}_T^{u, v,(i)}  \right) \right] \sum_{i=1}^N \left[ \frac{\delta \widetilde{Y}_T^{u,v,(i)}}{\delta u} (u;\phi) \right],
\end{align}
\end{subequations}
where
\begin{equation}
\label{eq:dYdu}
\frac{\delta \widetilde{Y}_T^{u,v,(i)}}{\delta u} (u;\phi) = \int_0^T \phi(X_s^{v,(i)},s) \cdot \mathrm{d}W^{(i)}_s - \int_0^T \left(\phi \cdot (u - v) \right)(X_s^{v, (i)}, s) \, \mathrm ds. 
\end{equation}
The claim now follows from observing that
\begin{equation}
\left(g\left(X_T^{v,(i)}\right)-\widetilde{Y}_T^{u,v,(i)}\right)\Big|_{u = u^*} 
\end{equation}
is almost surely constant (i.e. does not depend on $i$), according to the second equation in \eqref{eq:general FBSDE backward}.
\end{proof}
    
\subsection{Stability  in high dimensions -- robustness under tensorisation}
\label{sec:products}
    
In this section we study the robustness of the proposed algorithms in high-dimensional settings. As a motivation, consider the case when the drift and diffusion coefficients in the uncontrolled SDE \eqref{eq:uncontrolled SDE} split into separate contributions along different dimensions,
\begin{equation}
b(x,s) = \sum_{i=1}^d b_i(x_i,s), \qquad \sigma(x,s) = \sum_{i=1}^d \sigma_i(x_i,s),
\end{equation}
for $x=(x_1,\ldots,x_d) \in \mathbb{R}^d$, and analogously for  the running and terminal costs $f$ and $g$ as well as for the control vector field $u$. It is then straightforward to show that the path measure $\mathbb{P}^u$ associated to the controlled SDE \eqref{eq:controlled SDE} and the target measure $\mathbb{Q}$ defined in \eqref{eq:reweighted measure} factorise,
\begin{equation}
\label{eq:factorisation}
\mathbb{P}^u = \bigotimes_{i=1}^d \mathbb{P}^{u_i}, \qquad \mathbb{Q} = \bigotimes_{i=1}^d \mathbb{Q}_i.
\end{equation}
From the perspective of statistical physics, \eqref{eq:factorisation} corresponds to the scenario where non-interacting systems are considered simultaneously. 
To study the case when $d$ grows large, we leverage the perspective put forward in Section \ref{sec:divergences}, recalling that $D(\mathbb{P}\vert \mathbb{Q})$ denotes a generic divergence. In what follows, we will denote corresponding estimators based on a sample of size $N$ by $\widehat{D}^{(N)}(\mathbb{P}\vert \mathbb{Q})$, and study the quantity
\begin{equation}
\label{eq:relative error}
r^{(N)}(\mathbb{P} \vert \mathbb{Q}) :=  \frac{\sqrt{\Var \left( \widehat{D}^{(N)}(\mathbb{P}\vert \mathbb{Q})\right)}}{D(\mathbb{P}\vert \mathbb{Q})},
\end{equation}
measuring the relative statistical error when estimating $D(\mathbb{P}\vert \mathbb{Q})$ from samples, noting that $r^{(N)}(\mathbb{P} \vert \mathbb{Q}) = \mathcal{O}(N^{-1/2})$.
As $r^{(N)}$ is clearly linked to algorithmic performance and stability, we are interested in divergences, corresponding loss functions and estimators whose relative error remains controlled when the number of independent factors in \eqref{eq:factorisation} increases:
    	
\begin{definition}[Robustness under tensorisation]
\label{def:products}
We say that a divergence $D: \mathcal{P}(\mathcal{C}) \times \mathcal{P}(\mathcal{C}) \rightarrow \mathbb{R} \cup \{+ \infty \}$ and a corresponding estimator $\widehat{D}^{(N)}$ are \emph{robust under tensorisation} if, for all $\mathbb{P},\mathbb{Q} \in \mathcal{P}(\mathcal{C})$ such that $D(\mathbb{P} \vert \mathbb{Q}) < \infty$ and $N \in \mathbb{N}$, there exists $C > 0$ such that 
\begin{equation}
r^{(N)} \left(\bigotimes_{i=1}^M \mathbb{P}_i \Big\vert \bigotimes_{i=1}^M \mathbb{Q}_i \right)  < C,
\end{equation}
for all $M \in \mathbb{N}$.
Here, $\mathbb{P}_i$ and $\mathbb{Q}_i$ represent identical copies of $\mathbb{P}$ and $\mathbb{Q}$, respectively, so that $\bigotimes_{i=1}^M \mathbb{P}_i$ and $\bigotimes_{i=1}^M \mathbb{Q}_i$ are measures on the product space $\bigotimes_{i=1}^M C([0,T],\mathbb{R}^d) \simeq C([0,T],\mathbb{R}^{Md})$.     
\end{definition}
Clearly, if $\P$ and $\Q$ are measures on $C([0,T],\mathbb{R})$, then $M$ coincides with the dimension of the combined problem.
\begin{remark}
The variance and log-variance divergences defined in \eqref{eq:var divergence} and \eqref{eq:var log divergence} depend on an auxiliary measure $\widetilde{\mathbb{P}}$. Definition \ref{def:products} extends straightforwardly by considering the product measures $\bigotimes _{i=1}^d\widetilde{\mathbb{P}}_i$. In a similar vein, the relative entropy and cross-entropy divergences admit estimators that depend on a further probability measure $\widetilde{\P}$, \begin{equation}
\label{eq:RE and CE estimators}
\widehat{D}^{\RE,(N)}_{\widetilde{\P}}(\P \vert \Q) = \frac{1}{N} \sum_{j=1}^N \left[ \log \left( \frac{\mathrm{d}\P}{\mathrm{d}\Q}\right) \frac{\mathrm{d}\P}{\mathrm{d}\widetilde{\P}}\right] (X^j), \quad \widehat{D}^{\CE,(N)}_{\widetilde{\P}}(\P \vert \Q) = \frac{1}{N} \sum_{j=1}^N \left[ \log \left( \frac{\mathrm{d}\Q}{\mathrm{d}\P}\right) \frac{\mathrm{d}\P}{\mathrm{d}\widetilde{\P}}\right] (X^j), 
\end{equation}  
where $X^j \sim \widetilde{\P}$, motivated by the identities $D^{\RE}(\P \vert \Q) = \mathbb{E}_{\widetilde{\P}} \left[ \log \left( \frac{\mathrm{d}\P}{\mathrm{d}\Q} \right) \frac{\mathrm{d}\P}{\mathrm{d}\widetilde{\P}}\right]$ and $D^{\CE}(\P \vert \Q) = \mathbb{E}_{\widetilde{\P}} \left[ \log \left( \frac{\mathrm{d}\Q}{\mathrm{d}\P} \right) \frac{\mathrm{d}\Q}{\mathrm{d}\widetilde{\P}}\right]$. We refer to Remark \ref{remark_CE_additional_v} for a similar discussion.
\end{remark}
    
\begin{proposition}
\label{prop:robustness}
We have the following robustness and non-robustness properties:
\begin{enumerate}
\item
\label{it:log var robust} 
The log-variance divergence $D^{\mathrm{Var(log)}}_{\widetilde{\mathbb{P}}}$, approximated using the standard Monte Carlo estimator, is robust under tensorisation, for all $\widetilde{\mathbb{P}} \in \mathcal{P}(\mathcal{C})$.
\item
\label{it:relative entropy robust}
The relative entropy divergence $D^{\RE}$, estimated using $\widehat{D}^{\RE,(N)}_{\widetilde{\P}}$, is robust under tensorisation if and only if $\widetilde{\P} = \P$.
    \item
    \label{it:var robust}
    The variance divergence $D^{\mathrm{Var}}_{\widetilde{\mathbb{P}}}$ is not robust under tensorisation when approximated using the standard Monte Carlo estimator. More precisely, if $\frac{\mathrm{d}\mathbb{Q}}{\mathrm{d}\mathbb{P}}$ is not $\widetilde{\mathbb{P}}$-almost surely constant, then, for fixed $N \in \mathbb{N}$, there exist constants $a > 0$ and $C>1$ such that
    \begin{equation}
    \label{eq:r var}
    r^{(N)} \left(\bigotimes_{i=1}^M \mathbb{P}_i \Big\vert \bigotimes_{i=1}^M \mathbb{Q}_i \right) \ge a \,C^M,
    \end{equation}
for all $M\ge1$.
\item
\label{it:cross entropy robust}
The cross-entropy divergence $D^{\RE}$, estimated using $\widehat{D}^{\RE,(N)}_{\widetilde{\P}}$, is not robust under tensorisation. More precisely, for fixed $N \in \mathbb{N}$ there exists a constant $a>0$ such that 
\begin{equation}
\label{eq:r CE}
    r^{(N)} \left(\bigotimes_{i=1}^M \mathbb{P}_i \Big\vert \bigotimes_{i=1}^M \mathbb{Q}_i \right) \ge a \left(\sqrt{ \chi^2 (\Q \vert \widetilde{\P}) + 1} \right)^M,
\end{equation}
for all $M \ge 1$. Here 
\begin{equation}
\chi^2(\Q \vert \widetilde{\P}) = {\E}_{\widetilde{\P}} \left[ \left( \frac{\mathrm{d} \Q }{\mathrm{d} \widetilde{\P}}\right)^2 - 1 \right]
\end{equation}
denotes the $\chi^2$-divergence between $\Q$ and $\widetilde{\P}$. 
\end{enumerate}
\end{proposition}  
\begin{proof}
See Appendix \ref{app:products}. 
\end{proof}
\begin{remark}
Proposition \ref{prop:robustness} suggests that the variance and cross-entropy losses perform poorly in high-dimensional settings as the relative errors \eqref{eq:r var} and \eqref{eq:r CE} scale exponentially in $M$. Numerical support can be found in Section \ref{sec:numerics}. We note that in practical scenarios we have that $\widetilde{\P} \neq \Q$ as it is not feasible to sample from the target, and hence $\sqrt{ \chi^2 (\Q \vert \widetilde{\P}) + 1} > 1$.
\end{remark}    	
    
\section{Numerical experiments}
\label{sec:numerics}
In this section we illustrate our theoretical results on the basis of numerical experiments. In Subsection \ref{sec:computational_aspects} we discuss computational details of our implementations, complementing the discussion in Section \ref{sec:ensemble averages}.  The Subsections \ref{section_LLQC} and \ref{sec: LQGC} focus on the case when the uncontrolled SDE  \eqref{eq:uncontrolled SDE} describes an Ornstein-Uhlenbeck process and the dimension is comparatively large. In Section \ref{sec:double well} we consider metastable settings (of both low and moderate dimensionality), representative of those typically encountered in rare event simulations (see Example \ref{ex:rare events}).
We rely on PyTorch as a tool for automatic differentiation and refer to the code at \url{https://github.com/lorenzrichter/path-space-PDE-solver}.    
    
\subsection{Computational aspects}
\label{sec:computational_aspects}
    
The numerical treatment of the Problems \ref{prob:optimal control}-\ref{prob:variance minimisation} using the IDO-methodology is based on the explicit loss function representations in Section \ref{sec:divergences}, together with a gradient descent scheme relying on automatic differentiation\footnote{Note that for the gradients of the process $(X_s^u)_{0 \le s \le T}$ alternative computational methods can be considered (see \cite{gobet2005sensitivity} for an overview). A numerical analysis of the approach we rely on can be found in \cite{yang1991monte}.}. Following the discussion in Section \ref{sec:ensemble averages}, a particular instance of an IDO-algorithm is determined by the choice of a loss function, and, in the case of the cross-entropy, moment and variance-type losses, by a strategy to update the control vector field $v$ in the forward dynamics (see Propositions \ref{prop:cross entropy} and \ref{prop:variance}). As mentioned towards the end of Section \ref{sec:ensemble averages}, we focus on setting $v=u$ at each gradient step, i.e. to use the current approximation as a forward control. Importantly, we do not differentiate the loss with respect to $v$; in practice this can be achieved by removing the corresponding variables from the autodifferentiation computational graph (for instance using the \texttt{detach} command in the PyTorch package). Including differentiation with respect to $v$ as well as more elaborate choices of the forward control might be rewarding directions for future research.

Practical implementations require approximations at three different stages: first, the time discretisation of the SDEs \eqref{eq:uncontrolled SDE} or \eqref{eq:controlled SDE}; second, the Monte Carlo approximation of the losses (as outlined in Section \ref{sec:ensemble averages}), or, to be precise, the approximation of their respective gradients; and third, the function approximation of either the optimal control vector field $u^*$ or the value function $V$. Moreover, implementations vary according to the choice of an appropriate gradient descent method.
    
Concerning the first point, we discretise the SDE \eqref{eq:controlled SDE} using the Euler-Maruyama scheme \cite{kloeden2013numerical} along a time grid $0 = t_0 < \dots < t_K = T$, namely iterating 
\begin{equation}
\label{discrete_controlled_SDE}
\widehat{X}^u_{n+1} = \widehat{X}^u
_n\, + \,\left(b(\widehat{X}^u_n, t_n) + \sigma(\widehat{X}^u_n, t_n) u(\widehat{X}^u_n, t_n) \right) \Delta t\, +\, \sigma(\widehat{X}^u_n, t_n) \xi_{n+1} \sqrt{\Delta t},\qquad \widehat{X}_0 = x_\text{init},
\end{equation}
where $\Delta t > 0$ denotes the step size, and $\xi_n \sim \mathcal{N}(0, I_{d \times d})$ are independent standard Gaussian random variables. 
Recall that the initial value can be random rather than deterministic (see Remark \ref{rem:randomisation}). We demonstrate the potential benefit of sampling $\widehat{X}_0$ from a given density in Section \ref{sec: LQGC}.

We next discuss the approximation of $u^*$. First, note that a viable and straightforward alternative is to instead approximate $V$ and compute $u^* = - \sigma^\top \nabla V$ whenever needed (for instance by automatic differentiation), see \cite{raissi2018forward}. However, this approach has performed slightly worse in our experiments, and, furthermore, $V$ can be recovered from $u^* $ by integration along an appropriately chosen curve.
To approximate $u^*$, a classic option is a to use a Galerkin truncation, i.e. a linear combination of ansatz functions
\begin{equation}
u(x, t_n) = \sum_{m=1}^M \theta_m^n \alpha_m(x),
\end{equation}
for $n \in \{0, \dots, K-1\}$ with parameters $\theta_m^n \in \R$. Choosing an appropriate set $\{ \alpha_m \}_{m=1}^M$ is crucial for algorithmic performance -- a task that in high-dimensional settings requires detailed a priori knowledge about the problem at hand. Instead, we focus on approximations of $u^*$  realised by neural networks.
    
\begin{definition}[Neural networks]
\label{def_NN}
We define a standard \textit{feed-forward neural network} $\Phi_\varrho:\R^k \to \R^m$ by
\begin{equation}
\Phi_\varrho(x) = A_L \varrho(A_{L-1} \varrho(\cdots  \varrho(A_1 x + b_ 1) \cdots) + b_{L-1}) + b_L,
\end{equation}
with matrices $A_l \in \R^{n_{l} \times n_{l-1}}$, vectors $b_l \in \R^{n_l}, 1 \le l \le L$, and a nonlinear activation function $\varrho: \R \to \R$ that is to be applied componentwise.
We further define the \textit{DenseNet} \cite{weinan2018deepRitz, huang2017} containing additional skip connections,
\begin{equation}
\Phi_\varrho(x) = A_L x_L + b_L,
\end{equation}
where $x_{L}$ is defined recursively by
\begin{align}
y_{l+1} = \varrho(A_l x_l + b_l), \qquad 
x_{l+1} = (x_l, y_{l+1})^\top,
\end{align}
with $A_l \in \R^{n_l \times \sum_{i=0}^{l-1} n_i}, b_l \in \R^l$ for $1 \le l \le L-1$ and $x_1 = x$, $n_0 = d$. In both cases the collection of matrices $A_l$ and vectors $b_l$ comprises the learnable parameters $\theta$.
\end{definition}
    
Neural networks are known to be universal function approximators \cite{cybenko1989approximation, hornik1989multilayer}, with recent results indicating favourable properties in high-dimensional settings \cite{elbrachter2018dnn, eldan2016power, grohs2018proof, petersen2018optimal, schwab2019deep}. The control $u$ can be represented by either $u(x,t) = \Phi_\varrho(y)$ with $y=(x,t)^\top$, i.e. using one neural network for both the space and time dependence, or by $u(x,t_n) = \Phi^n_\varrho(x)$, using one neural network per time step. The former alternative led to better performance in our experiments, and the reported results rely on this choice. For the gradient descent step we either choose SGD with constant learning rate \cite[Algorithm 8.1]{goodfellow2016deep} or Adam \cite[Algorithm 8.7]{goodfellow2016deep}, \cite{kingma2014adam}, a variant that relies on adaptive step sizes and momenta. Further numerical investigations on network architectures and optimisation heuristics can be found in \cite{chan2019machine}. 

To evaluate algorithmic choices we monitor the following two performance metrics:
\begin{enumerate}
\item The \emph{importance sampling relative error}, namely
\begin{equation}
\label{eq: importance sampling relative error}
\delta(u) :=
\frac{\sqrt{\operatorname{Var}\left( e^{-\mathcal{W}(X^u)} \frac{\mathrm d \P}{\mathrm d \P^u} \right)}}{\E[e^{-\mathcal{W}(X)}]},
\end{equation}
where $u$ is the approximated control in the corresponding iteration step. This quantity is zero if and only if $u = u^*$ (cf. Theorem \ref{thm:connections}) and measures the quality of the control in terms of the objective introduced in Problem \ref{prob:variance minimisation}. Since its Monte Carlo version fluctuates heavily if $u$ is far away from $u^*$ we usually estimate this quantity with additional samples not being used in the gradient computation.
\item An \emph{$L^2$-error}, 
\begin{equation}
{\E}\left[ \int_0^T |u - u^*_\text{ref}|^2(X^u_s, s) \, \mathrm ds \right],
\end{equation}
where $u^*_\text{ref}$ is computed either analytically or using a finite difference scheme for the HJB-PDE \eqref{eq:HJB}. This quantity is more robust w.r.t. deviations from $u^*$ and therefore we compute the Monte Carlo estimator using just the samples from the training iteration.
\end{enumerate}

\subsection{Ornstein-Uhlenbeck dynamics with linear costs}
\label{section_LLQC}
    
Let us consider the controlled Ornstein-Uhlenbeck process
\begin{equation}
\label{Ornstein-Uhlenbeck}
\mathrm dX_s^u = \left(AX_s^u + B u(X_s^u, s)\right) \mathrm d s + B \,\mathrm d W_s, \quad X_0^u = 0,
\end{equation}
where $A,B \in \R^{d \times d}$. Furthermore, we assume zero running costs, $f = 0$, and linear terminal costs $g(x) = \gamma \cdot x$, for a fixed vector $\gamma \in \R^d$. As shown in Appendix \ref{OU_linear_solution}, the optimal control is given by
\begin{equation}
u^*(x, t) = -B^\top e^{A^\top (T-t)}\gamma,
\end{equation}
which remarkably does not depend on $x$. Therefore, not only the variance and log-variance losses are robust at $u^*$ in the sense of Definition \ref{def:robust u}, but also the relative entropy loss, according to \eqref{eq:variance RE loss} in Proposition \ref{prop:robustness at u}.
    
We choose $A = -I_{d \times d} + (\xi_{ij})_{1\le i,j \le d}$ and  $B = I_{d \times d} + (\xi_{ij})_{1\le i,j \le d}$, where $\xi_{ij} \sim \mathcal{N}(0, \nu^2)$ are sampled i.i.d. once at the beginning of the simulation. Note that this choice corresponds to a small perturbation of the product setting from Section \ref{sec:products}. We set $T = 1, \nu = 0.1$, $\gamma = (1, \dots, 1)^\top$ and as function approximation take the DenseNet from Definition \ref{def_NN} using two hidden layers, each with a width of $n_1 = n_2 = 30$, and $\varrho = \max(0, x)$ as the nonlinearity. Lastly, we choose the Adam optimiser as a gradient descent scheme. Figure \ref{LLGC_d1_loss_log} shows the algorithm's performance for $d = 1$ with batch size $N = 200$, learning rate $\eta = 0.01$ and step size $\Delta t = 0.01$. We observe that log-variance, relative entropy and moment loss perform similarly and converge well to a suitable approximation. The cross-entropy loss decreases, but at later gradient steps fluctuates more than the other losses (we note that the fluctuations appear to be less pronounced when using SGD, however at the cost of substantially slowing down the overall speed of convergence). The inferior quality of the control obtained using the cross-entropy loss may be explained by its non-robustness at $u^*$, see Proposition \ref{prop:robustness at u}.
\begin{figure}[H]
\centering
\includegraphics[width=0.9\linewidth]{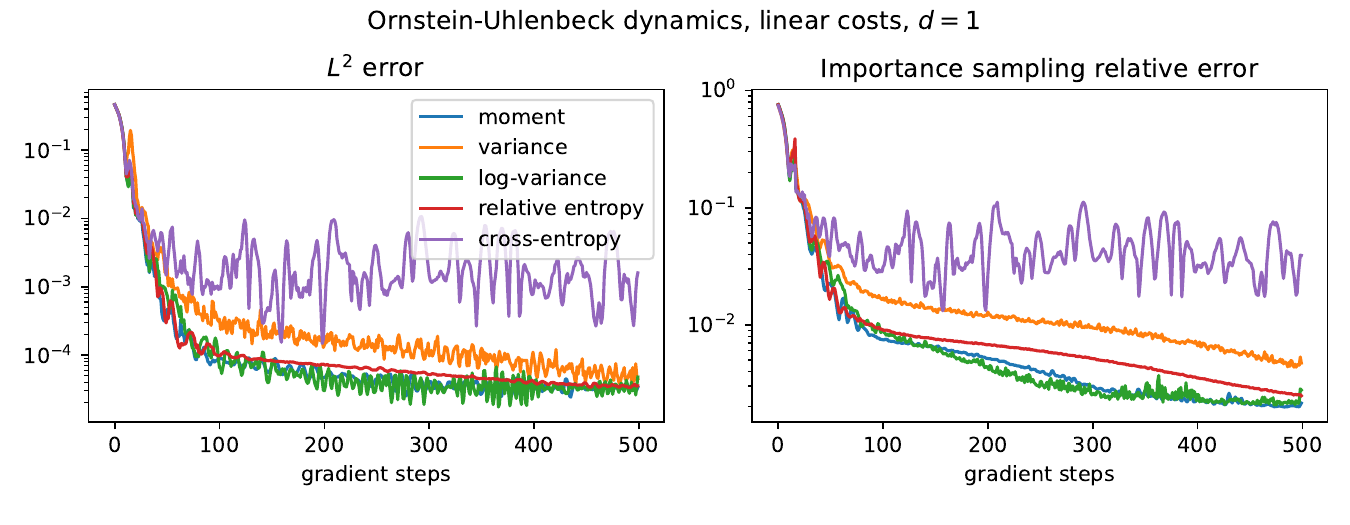}
\caption{Performance of the algorithm using five different loss functions according to the metrics introduced in Section \ref{sec:computational_aspects} as a function of the iteration step.}
\label{LLGC_d1_loss_log}
\end{figure}
    
Figure \ref{LLGC_d40_loss_log} shows the algorithm's performance in a high-dimensional case, $d = 40$, where we now choose $N = 500$ as the batch size, $\eta = 0.001$ as the learning rate, $\Delta t = 0.01$ as the time step, and as before rely on a DenseNet with two hidden layers. We observe that relative entropy loss and log-variance loss perform best, and that the moment  and cross-entropy losses converge at a significantly slower rate. The variance loss is numerically unstable and hence not represented in Figure \ref{LLGC_d40_loss_log}. We encounter similar problems in the subsequent experiments and thus do not consider the variance loss in what follows. In Figure \ref{LLGC_d40_u_approx} we plot some of the components of the $40$-dimensional approximated optimal control vector field as well as the analytic solution $u_{\mathrm{ref}}^*(x, t)$ for a fixed value of $x$ and varying time $t$, showcasing the inferiority of the approximation obtained using the cross-entropy loss. The comparatively poor performance of the cross-entropy and the variance losses can be attributed to their non-robustness with respect to tensorisations, see Section \ref{sec:products}.
To further illustrate these results, Figure \ref{rel-error-dimension} displays the relative error associated to the loss estimators computed from $N = 15\cdot 10^6$ samples in different dimensions. The dimensional dependence agrees with what is expected from Proposition \ref{prop:robustness}, but we note that our numerical experiment goes beyond the product case.
    
\begin{figure}[H]
\centering
\includegraphics[width=0.9\linewidth]{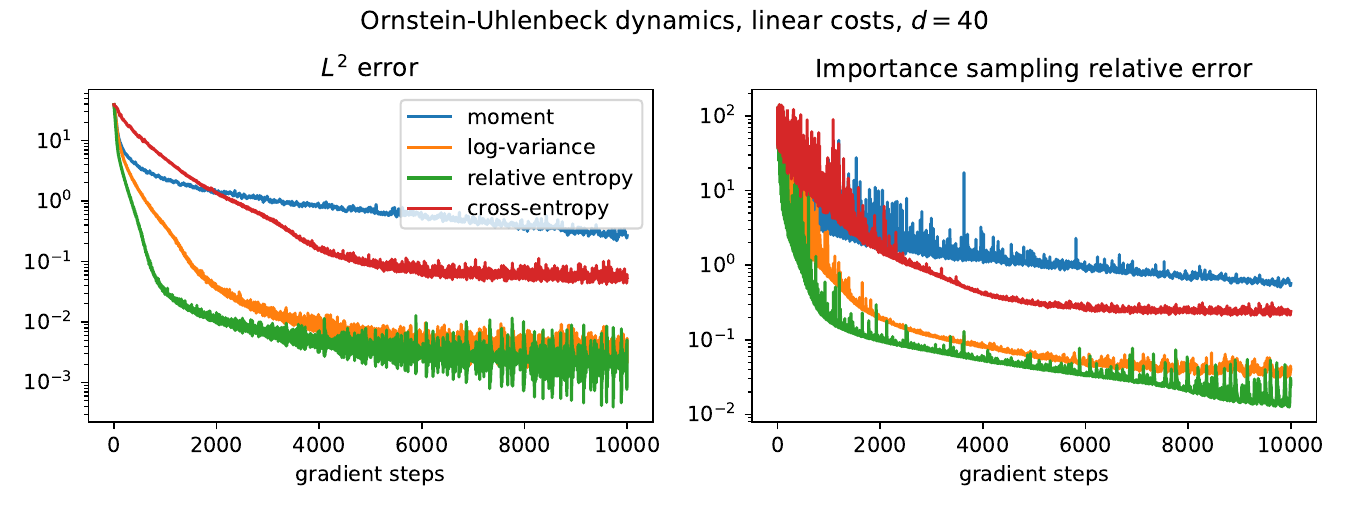}
\caption{Performance of the algorithm using four different loss functions in a high-dimensional setting.}
\label{LLGC_d40_loss_log}
\end{figure}
    
\begin{figure}[H]
\centering
\includegraphics[width=0.9\linewidth]{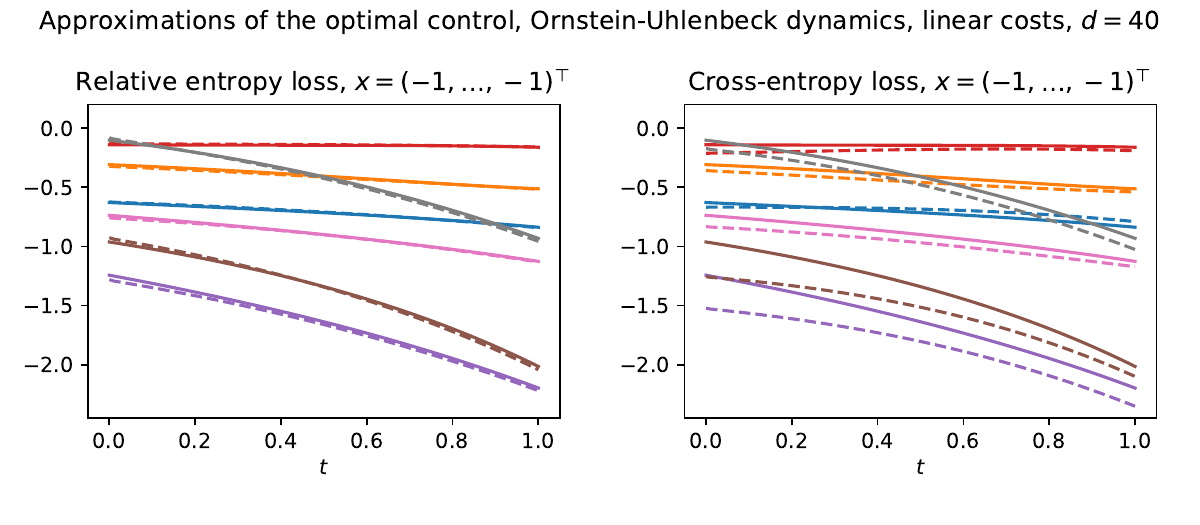}
\caption{Approximation $u$ (dashed lines) and reference solution $u^*_\text{ref}$ (straight lines) for the optimal control obtained using the relative entropy and cross-entropy losses, respectively. $7$ out of the $40$ components of $u$ and $u^*_\text{ref}$ are plotted.}
\label{LLGC_d40_u_approx}
\end{figure}
    
\begin{figure}[H]
\centering
\includegraphics[width=0.40\linewidth]{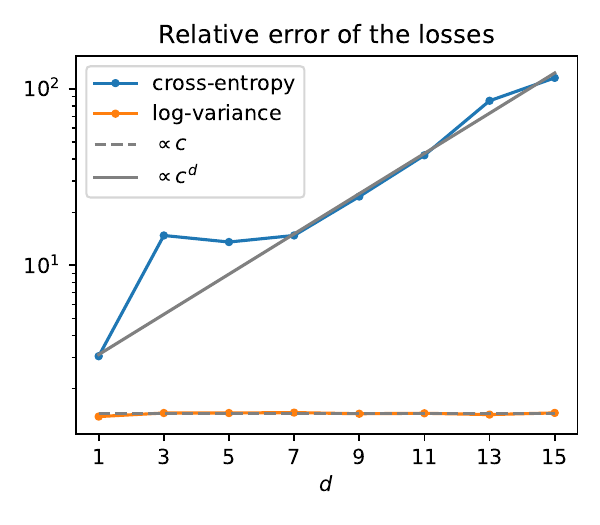}
\caption{Relative error of the log-variance and cross-entropy losses depending on the dimension.}
\label{rel-error-dimension}
\end{figure}

Lastly, let us investigate the effect of the additional parameter $y_0$ in the moment loss. For a first experiment, we initialise $y_0$ with either the naive choice $y_0^{(1)} = 0$, or $y_0^{(2)} = 10$, a starting value which differs considerably from $-\log \mathcal{Z}$ or the optimal choice $y_0^{(3)} = - \log \mathcal{Z} \approx -5.87$. Let us insist that in practical scenarios the value of $-\log \mathcal{Z}$ is usually not known. 
Additionally, we contrast using Adam and SGD as an optimisation routine -- in both cases we choose $N = 200$, $\eta = 0.01$, $\Delta t = 0.01$,  and the same DenseNet architecture as in the previous experiments. 
    
Figure \ref{LLGC_d20_adam_SGD} shows that the initialisation of $y_0$ can have a significant impact on the convergence speed. Indeed, 
with the initialisation $y_0 = -\log \mathcal{Z}$,  the moment and  log-variance losses perform very similarly, in accordance with Proposition \ref{prop: equivalence moment log-variance}. In contrast, choosing the initial value $y_0$ such that the discrepancy $|y_0 +\log \mathcal{Z}|$ is large incurs a much slower convergence.
    
Comparing the two plots in Figure \ref{LLGC_d20_adam_SGD} shows that the Adam optimiser achieves a much faster convergence overall in comparison to SGD. Moreover, the difference in performance between $y_0$-initialisations is more pronounced when the Adam optimiser is used. The observations in these experiments are in agreement with those in \cite{chan2019machine}.
    
\begin{figure}[H]
\centering
\includegraphics[width=0.90\linewidth]{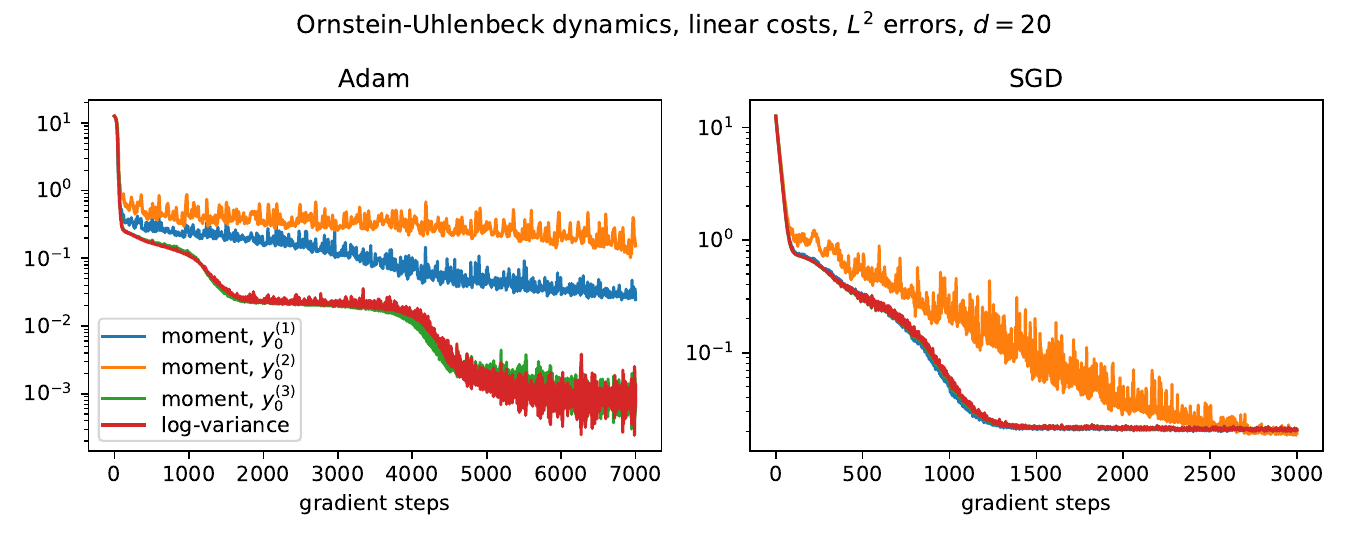}
\caption{Performance of the algorithm with the moment loss and different initialisations for $y_0$, using Adam and SGD.}
\label{LLGC_d20_adam_SGD}
\end{figure}
    
\subsection{Ornstein-Uhlenbeck dynamics with quadratic costs}
\label{sec: LQGC}
We consider the Ornstein-Uhlenbeck process described by \eqref{Ornstein-Uhlenbeck} with quadratic running and terminal costs, i.e. $f(x, s) = x^\top P x$ and $g(x) = x^\top R x$, with $P,R \in \R^{d \times d}$. This setting is known as the \textit{linear quadratic Gaussian control} problem \cite{vanHandel2007}. The optimal control is given by \cite[Section 6.5]{vanHandel2007}
\begin{align}
\label{eq:LQGC solution}
u^*(x, t) &=  -2 B_t^\top F_t x,
\end{align}
where the matrices $F_t$ fulfill the matrix Riccati equation
\begin{align}
\frac{\mathrm d}{\mathrm d t}F_t + A^\top_t F_t + F_t A_t - 2 F_tB_tB_t^\top F_t + P = 0,\qquad F_T = R.
\end{align}
In this example, we demonstrate an approach leveraging a priori knowledge about the structure of the solution.
Motivated by \eqref{eq:LQGC solution}, we consider the linear ansatz functions
\begin{equation}
u(x, t_n) = \Xi_n x,
\end{equation}
where the entries of the matrices $\Xi_n \in \R^{d \times d}$, $n = 0,\ldots, K-1$ represent the parameters to be learnt. The matrices $A$ and $B$ are chosen as in Subsection \ref{section_LLQC} and we set $P = \frac{1}{2} I_{d\ \times d}$, $R = I_{d \times d}$ and $T=0.5$. Figure \ref{LQGC_d10_loss_log_SGD_adam} shows the performance using Adam with learning rate $\eta = 0.001$ and SGD with learning rate $\eta = 0.01$, respectively. 
The relative entropy losses converges fastest, followed by the log-variance loss. The convergence of the cross-entropy loss is significantly slower, in particular in the SGD case. We also note that the cross-entropy loss diverges if larger learning rates are used. These findings are in line with the results from Proposition \ref{prop:robustness}. When SGD is used, the moment loss experiences fluctuations in later gradient steps. This can be explained by the fact that the moment loss is robust at $u^*$ only if $y_0 = - \log \mathcal{Z}$ is satisfied exactly (see Proposition \ref{prop: equivalence moment log-variance}).
\par\bigskip

Let us illustrate the potential benefit of sampling $X_0$ from a predescribed density (see Remark \ref{rem:randomisation}), here $X_0 \sim \mathcal{N}(0, I_{d \times d})$. The overall convergence is hardly affected and the $L^2$ error dynamics agrees qualitatively with the one shown in Figure \ref{LQGC_d10_loss_log_SGD_adam}. However, the approximation is more accurate at initial time $t=0$, see Figure \ref{LQGC_X_0_random}. This phenomenon appears to be particularly pronounced in this example, as independent ansatz functions are used at each time step.
    
\begin{figure}[H]
\centering
\includegraphics[width=0.90\linewidth]{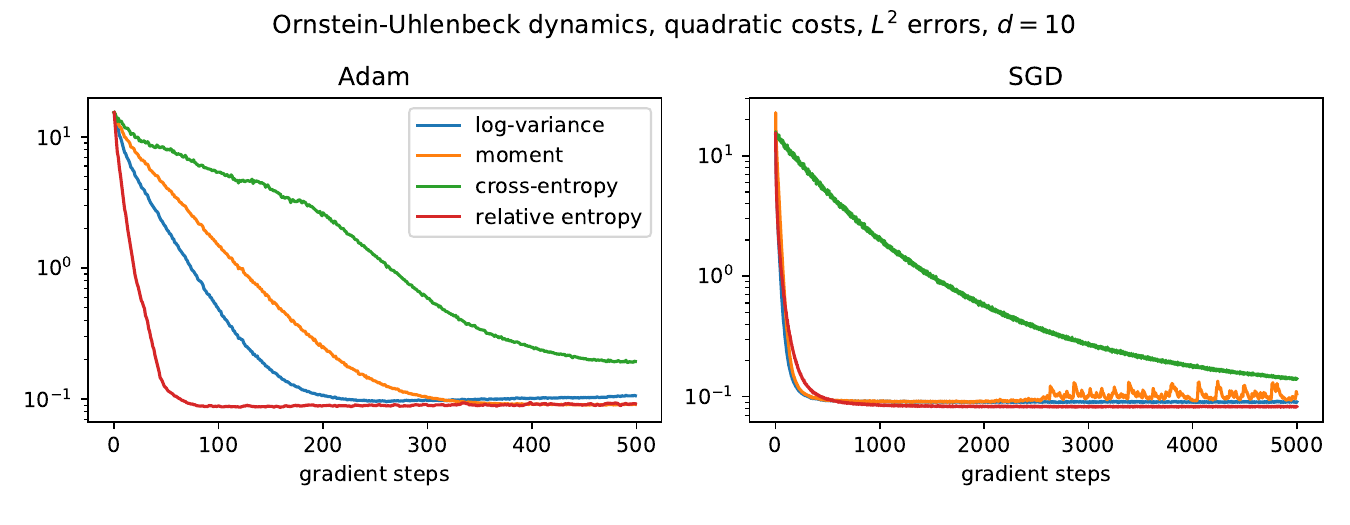}
\caption{Performance of the losses for the Ornstein-Uhlenbeck process with quadratic costs, using Adam and SGD.}
\label{LQGC_d10_loss_log_SGD_adam}
\end{figure}
    
\begin{figure}[H]
\centering
\includegraphics[width=0.90\linewidth]{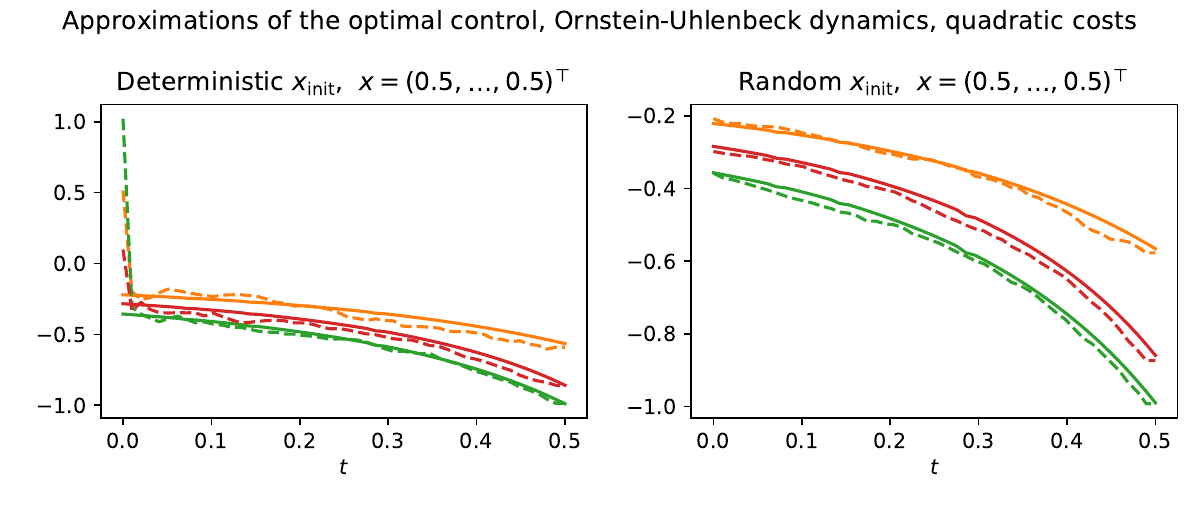}
\caption{Approximation and reference solution of the optimal control with either deterministic or random initialisations of $x_\text{init}$. Three components of $u$ and $u_\text{ref}^*$ are plotted.}
\label{LQGC_X_0_random}
\end{figure}
    
\subsection{Metastable dynamics in low and high dimensions}
\label{sec:double well}
We now come back to the double well potential from Example \ref{ex:rare events} and consider the SDE
\begin{equation}
\label{SDE_double_well}
\mathrm dX_s = -\nabla \Psi(X_s) \, \mathrm ds + B \, \mathrm dW_s, \quad X_0 = x_\text{init},
\end{equation}
where $B \in \R^{d \times d}$ is the diffusion coefficient,  $\Psi(x) = \sum_{i=1}^d \kappa_i(x_i^2-1)^2$ is the potential (with $\kappa_i > 0$ being a set of parameters) and $x_\text{init} = (-1, \dots, -1)^\top$ is the initial condition. We consider zero running costs, $f = 0$, terminal costs $g(x) = \sum_{i=1}^d \nu_i (x_i-1)^2$, where $\nu_i > 0$, and a terminal time $T=1$. Recall from Example \ref{ex:rare events} that choosing higher values for $\kappa_i$ and $\nu_i$ accentuates the metastable features, making sample-based   estimation of $ \E\left[\exp(-g(X_T))\right]$
more challenging. For an illustration, Figure \ref{fig: double well illustration} shows the potential $\Psi$ and the weight at final time $e^{-g}$ (see \eqref{eq:reweighted measure}), for different values of $\nu$ and $\kappa$, in dimension $d=1$ and for $B=1$. We furthermore plot the `optimally tilted potentials' $\Psi^* = \Psi + BB^\top V$, noting that $-\nabla \Psi^* = -\nabla \Psi + Bu^*$. Finally, the right-hand side shows the gradients $\nabla u^*$ at initial time $t=0$.
\begin{figure}[H]
\centering
\includegraphics[width=1.0\linewidth]{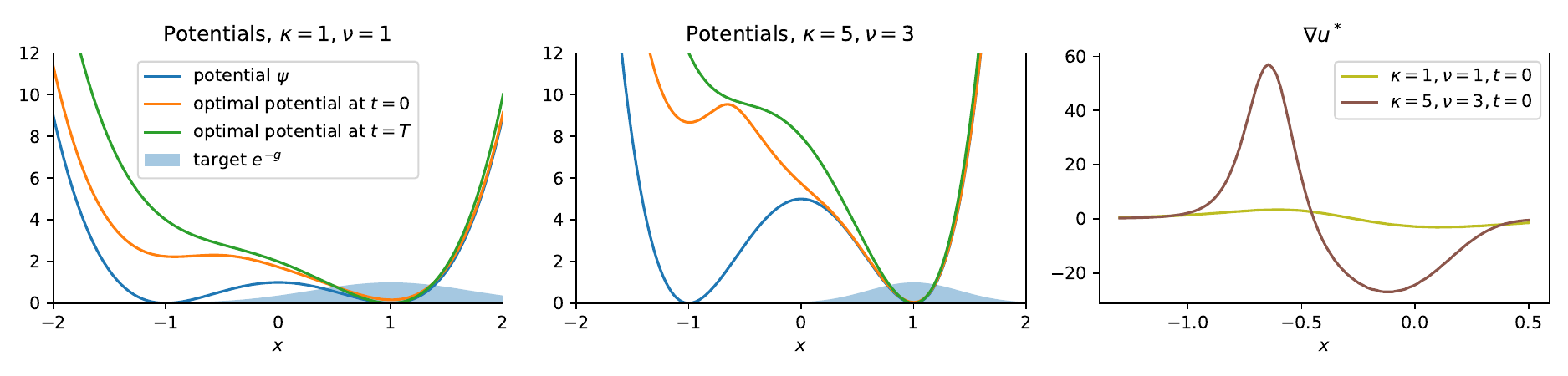}
\caption{The double well potential and  the weight $e^{-g}$, for different values of $\kappa$ and $\nu$ as well as optimal controls (inducing `tilted potentials') and their gradients.}
\label{fig: double well illustration}
\end{figure}
    
For an experiment, let us first consider the one-dimensional case, choosing $B = 1$, $\kappa = 5$ and $\nu = 3$. In this setting the relative error associated to the standard Monte Carlo estimator, i.e. the estimator version of \eqref{eq: importance sampling relative error}, which we denote by $\widehat{\delta}$, is roughly $\widehat{\delta}(0) = 63.86$ for a batch size of $N = 10^7$ trajectories, from which only about $2 \cdot 10^3$ (i.e. 0.02\%) cross the barrier. Given that  $e^{-g}$ is supported mostly in the right well, the optimal control $u^*$ steers the dynamics across the barrier. Using an approximation of $u^*$ obtained by a finite difference scheme, we achieve a relative error of $\widehat{\delta}(u^*) = 1.94$ (the theoretical optimum being zero, according to Theorem \ref{thm:connections}) and a crossing ratio of approximately 87.28\%.

To run IDO-based algorithms, we use the standard feed-forward neural network (see Definition \ref{def_NN}) with the activation function $\varrho = \tanh$ and choose $\Delta t = 0.005$, $\eta = 0.05$. We try batch sizes of $N = 50$ and $N = 1000$ and plot the training progress in Figures \ref{double_well_d_1_K_300} and \ref{double_well_d_1_K_1000}, respectively.
In Figure \ref{double_well_1d_approx} we display the approximation obtained using the log-variance loss and compare with the reference solution $u^*_{\mathrm{ref}}$.

\begin{figure}[H]
\centering
\includegraphics[width=0.90\linewidth]{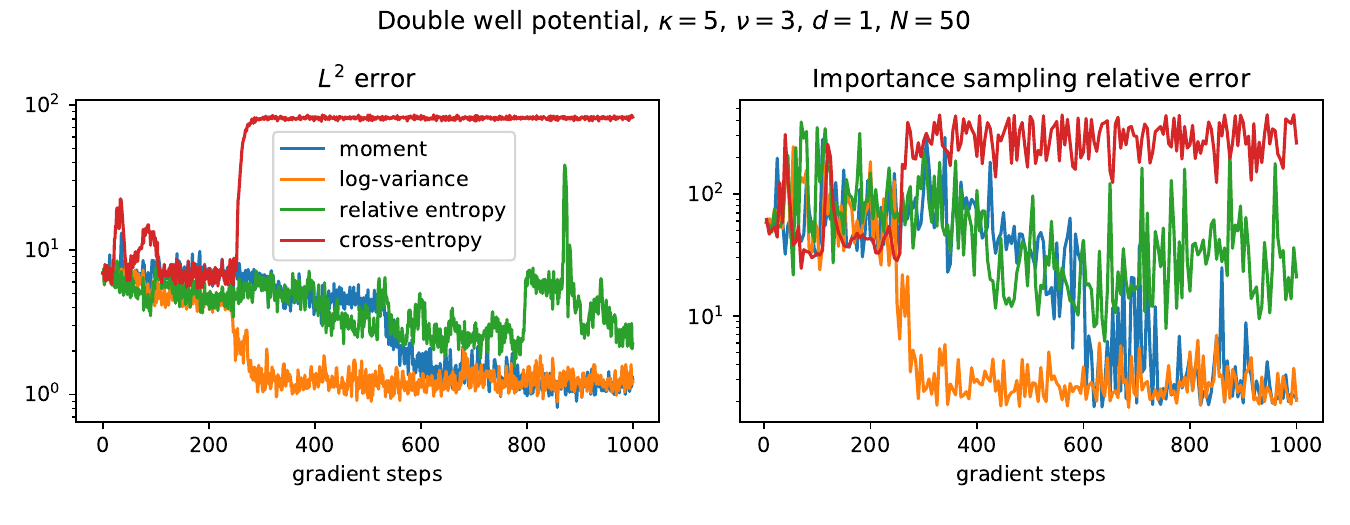}
\caption{Training iterations for the one-dimensional metastable double well example for a small batch size.}
\label{double_well_d_1_K_300}
\end{figure}
    
\begin{figure}[H]
\centering
\includegraphics[width=0.90\linewidth]{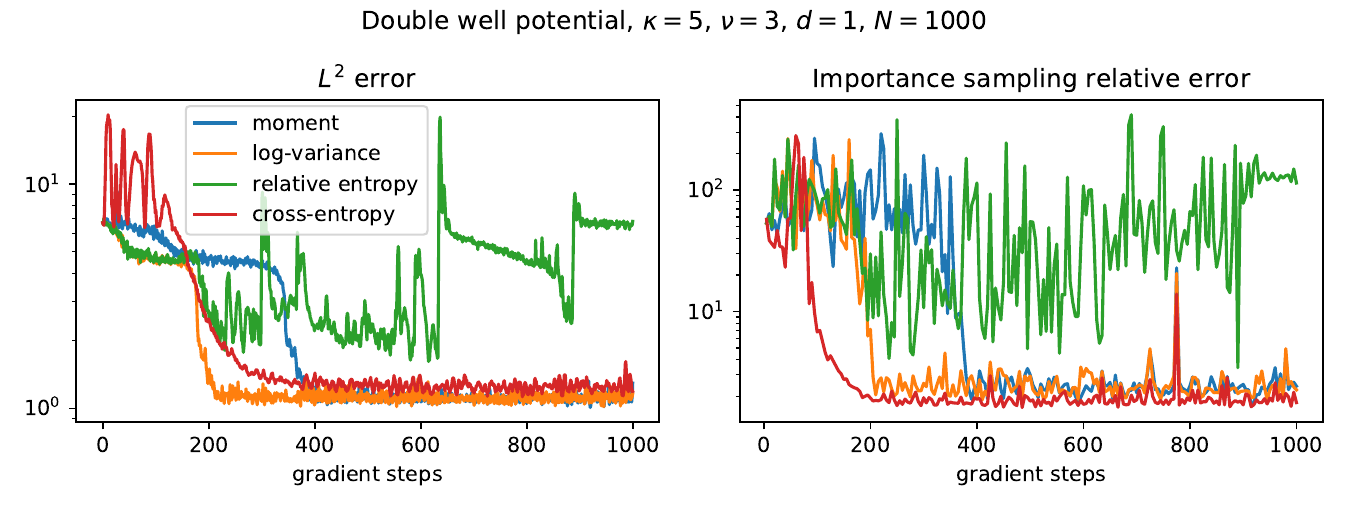}
\caption{Training iterations for the one-dimensional metastable double well example for a large batch size.}
\label{double_well_d_1_K_1000}
\end{figure}

It can be observed that the log-variance and moment losses perform well with both batch sizes, with the log-variance loss however achieving a satisfactory approximation with fewer gradient steps. The cross-entropy loss appears to work well only if the batch size is sufficiently large. We attribute this observation to the non-robustness at $u^*$ (see Proposition \ref{prop:robustness at u}) and, tentatively, to the exponential factor appearing in \eqref{eq:CE2}, see Remark \ref{remark_CE_additional_v}. 

The optimisation using the relative entropy loss is frustrated by instabilities in the vicinity of the solution $u^*$. In order to further investigate this aspect we numerically compute the variances of the gradients and the associated relative errors with respect to the mean, using $50$ realisations at each gradient step.
Figure \ref{rel_error_gradients_ma} shows the averages of the relative errors and variances over weights in the network\footnote{In order to lessen the impact of Monte Carlo errors and numerical instabilities, we take moving averages comprising $30$ gradient steps and discard partial derivatives with an average magnitude of less than $0.01$. We note that the plateaus present in Figure \ref{rel_error_gradients_ma} are an artefact due to the moving averages, but insist that this procedure does not alter the main results in a qualitative way.}, confirming that the gradients associated to the log-variance loss have significantly lower variances.  
This phenomenon is in accordance with Proposition \ref{prop:robustness at u} (in particular noting that  $|\nabla u^*|^2$ is expected to be rather large in a metastable setting, see Figure \ref{fig: double well illustration}) and explains the unsatisfactory behaviour of the relative entropy loss observed in Figures \ref{double_well_d_1_K_300} and \ref{double_well_d_1_K_1000}.
    
\begin{figure}[H]
\centering
\includegraphics[width=0.90\linewidth]{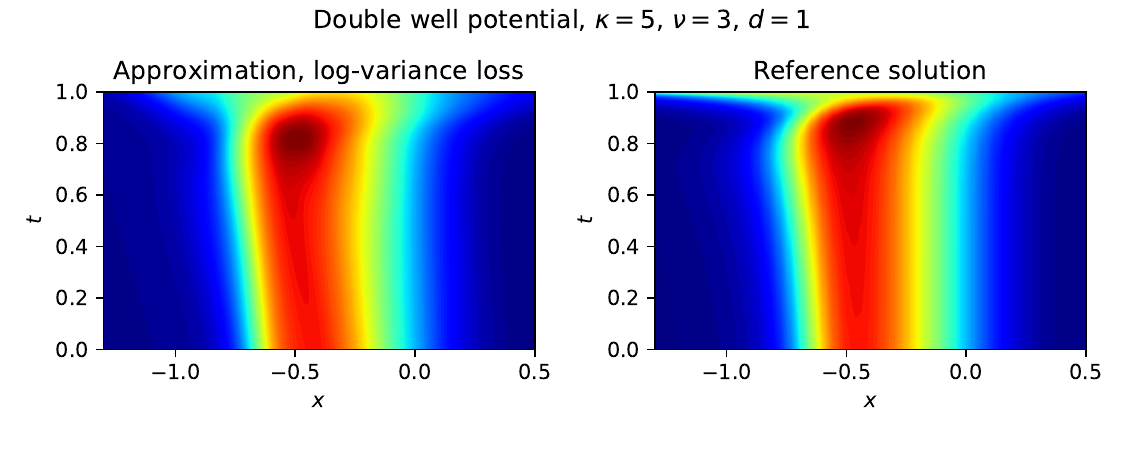}
\caption{Approximation and reference solution for the double well control problem in $d=1$.}
\label{double_well_1d_approx}
\end{figure}

\begin{figure}[H]
\centering
\includegraphics[width=1.0\linewidth]{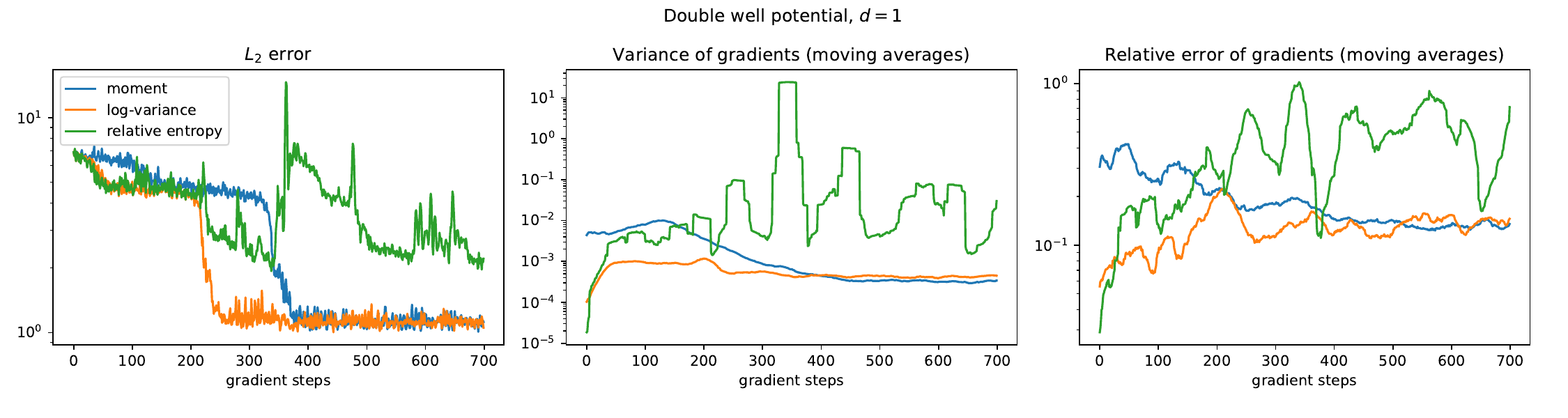}
\caption{We display the $L_2$ error pertaining to the one-dimensional double well experiment, along with the estimated averages of the variances and relative errors of the gradients along the training iterations for different losses.}
\label{rel_error_gradients_ma}
\end{figure}
    
Let us now consider the multidimensional setting, namely $d=10$, where the dynamics exhibits `highly' metastable characteristics in 3 dimensions and `weakly' metastable characteristics in the remaining 7 dimensions. 
To be precise, we set $\kappa_i = 5$, $\nu_i = 3$ for $i \in \{1, 2, 3\}$ and $\kappa_i = 1$, $\nu_i = 1$ for $i \in \{4, \dots, 10\}$. Moreover, we choose the diffusion coefficient to be $B = I_{d \times d}$ and conduct the experiment with a batch size of $N=500$. 
    
In Figure \ref{double_well_d_10} we see that only the log-variance loss achieves a reasonable approximation. Interestingly, the training progresses in stages, successively overcoming the potential barriers in the highly metastable directions. 
On the right-hand side we display the components of the approximated optimal control associated to one highly and one weakly metastable direction, for fixed $t=0$. We observe that the approximation is fairly accurate, and that comparatively large control forces are needed to push the dynamics over the highly metastable potential barrier.
     
\begin{figure}[H]
\centering
\includegraphics[width=0.90\linewidth]{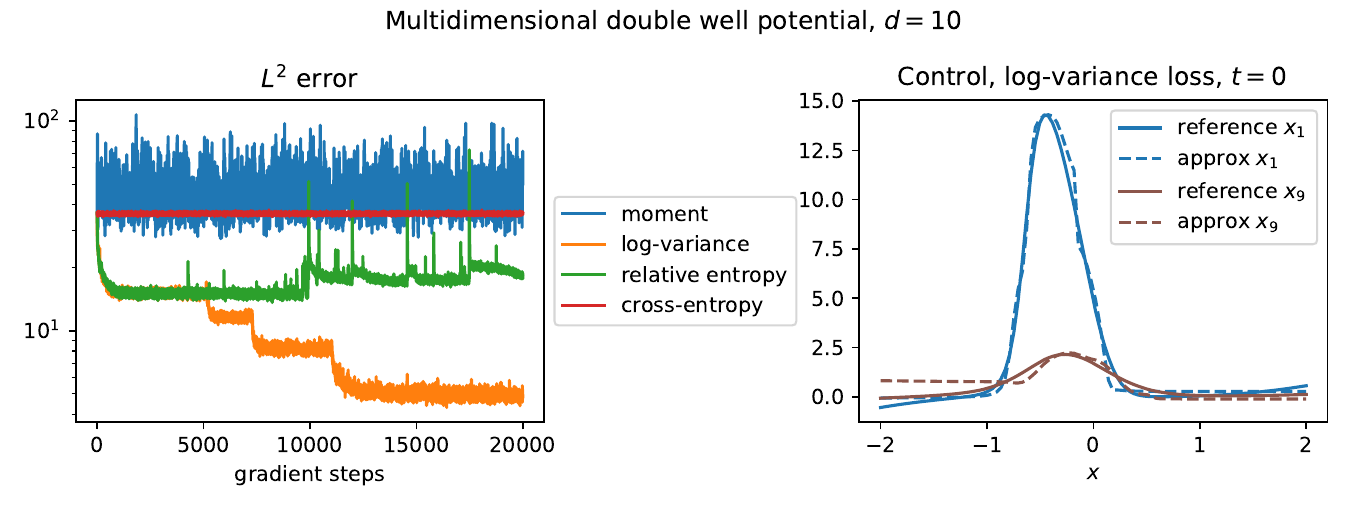}
\caption{Training iterations for the multidimensional metastable double well along with the approximated solution using the log-variance loss, from which we plot two components.
}
\label{double_well_d_10}
\end{figure}

\section{Conclusion and outlook}
\label{sec:outlook}
Motivated by the observation that optimal control of diffusions can be phrased in a number of different ways, we have provided a unifying framework based on divergences between path measures, encompassing various existing numerical methods in the class of IDO algorithms. In particular, we have shown that the novel log-variance divergences are closely connected to forward-backward SDEs. We have furthermore shown a fundamental equivalence between approaches based on the $\mathrm{KL}$-divergence and the log-variance divergences.

Turning to the variance of Monte Carlo gradient estimators, we have defined and studied two notions of stability -- robustness under tensorisation and robustness at the optimal control solution. Of the losses and estimators under consideration, only the log-variance loss is stable in both senses, often resulting in superior numerical performance. The consequences of robustness and non-robustness as defined have been exemplified by extensive numerical experiments.

The results presented in this paper can be extended in various directions. First, it would be interesting to consider other divergences on path space and construct and study the ensuing algorithms. In this respect, we may also mention the development of more elaborate schemes to update the control for the forward dynamics. Second, one may attempt to generalise the current framework to other types of control problems and PDEs (for instance to elliptic PDEs and hitting time problems as considered in \cite{hartmann2014characterization,hartmann2019variational,hartmann2012efficient,hartmann2018importance}, or to the Schr{\"o}dinger problem as discussed in \cite{reich2019data}).
Deeper understanding of the design of IDO algorithms could be achieved by extending our stability analysis beyond the product case and for controls that differ greatly from the optimal one. In particular, advances in this direction might help to develop more sophisticated variance reduction techniques. Finally, we envision applications of the log-variance divergences in other settings.
\par\bigskip
\textbf{Acknowledgements.}
This research has been funded by 
Deutsche Forschungsgemeinschaft (DFG) through the grant 
CRC 1114 \lq Scaling Cascades in Complex Systems\rq \,(projects A02 and A05, project number 235221301). We would like to thank Carsten Hartmann and Wei Zhang for many very useful discussions. We thank the referees for their useful comments and suggestions that have led to various improvements in the presentation of this paper.

\appendix
    
\section{Appendix}
\subsection{Proofs for Section \ref{sec:divergences}} 
\label{app:divergences}
    
The Radon-Nikodym derivatives appearing in the divergences defined in Section \ref{sec:divergences} can be computed explicitly:
\begin{lemma}
\label{lem:Girsanov}
For $u \in \mathcal{U}$, the measures $\mathbb{P}$ and $\mathbb{P}^u$ are equivalent. Moreover, the Radon-Nikodym derivative satisfies
\begin{equation}
\label{eq:Pu P}
\frac{\mathrm d \P^{u}}{\mathrm d\P}(X) = \exp \left( \int_0^T \left(u^\top  \sigma^{-1}\right)(X_s, s) \cdot \mathrm dX_s - \int_0^T (\sigma^{-1} b \cdot u)(X_s, s) \,\mathrm ds - \frac{1}{2} \int_0^T |u(X_s, s)|^2 \, \mathrm ds \right)
\end{equation}
\end{lemma}

\begin{proof}
The fact that the two measures are equivalent follows from the linear growth assumption on $u$ (see \eqref{eq:control set}), combining Bene\v{s}' theorem with Girsanov's theorem, see \cite[Proposition 2.2.1 and Theorem 2.1.1]{ustunel2013transformation}. According to a slight generalisation of  \cite[Theorem 2.4.2]{ustunel2013transformation}, we have 
\begin{equation}
\label{eq:P PW}
\frac{\mathrm{d} \mathbb{P}}{\mathrm{d}\mathbb{P}_{\mathrm{W}}} (X) = \exp \left( \int_0^T (b(X_s, s) \cdot \sigma^{-2}(X_s, s) \, \mathrm{d}X_s - \frac{1}{2} \int_0^T (b \cdot \sigma^{-2} b)(X_s, s) \, \mathrm{d}s \right),
\end{equation}
and 
\begin{equation}
\label{eq:Pu PW}
\frac{\mathrm{d} \mathbb{P}^u}{\mathrm{d}\mathbb{P}_{\mathrm{W}}} (X) = \exp \left( \int_0^T (b + \sigma u)(X_s, s) \cdot \sigma^{-2}(X_s, s) \, \mathrm{d}X_s - \frac{1}{2} \int_0^T \left((b + \sigma u) \cdot \sigma^{-2} (b + \sigma u)\right)(X_s, s) \, \mathrm{d}s \right),
\end{equation}
where $ \mathbb{P}_{\mathrm{W}}$ denotes the measure on $\mathcal{C}$ induced by
\begin{equation}
\mathrm d X_s = \sigma(X_s,s) \, \mathrm dW_s, \qquad X_0 = x_{\mathrm{init}}.
\end{equation}
Using
\begin{equation}
\frac{\mathrm d \P^{u}}{\mathrm d\P}(X) = \frac{\mathrm d \P^{u}}{\mathrm d\P_{\mathrm{W}}} \frac{\mathrm d \P_{\mathrm{W}}}{\mathrm d\P}(X),
\end{equation}
and inserting \eqref{eq:P PW} and \eqref{eq:Pu PW}, we obtain the desired result.
\end{proof}
    
\begin{proof}[Proof of Proposition \ref{prop:relative entropy}]
Using \eqref{eq:reweighted measure} and \eqref{eq:Pu P} (or arguing as in the proof of Theorem \ref{thm:connections}) we compute
\begin{subequations}
\begin{align}
\mathcal{L}_{\text{RE}}(u) &= {\E}_{\P^u}\left[\log \frac{\mathrm d \P^u}{\mathrm d \Q} \right] 
= {\E}_{\P^u}\left[\log \left(\frac{\mathrm d \P^u}{\mathrm d \P}\frac{\mathrm d \P}{\mathrm d \Q}\right) \right] \\
&= {\E}\left[ \int_0^T u(X_s^u, s) \cdot  \,\mathrm dW_s + \frac{1}{2} \int_0^T |u(X^u_s, s)|^2 \, \mathrm ds + \int_0^T f(X^u_s, s)\mathrm ds  + g(X^u_T)  \right] + \log \mathcal{Z} \\
&= {\E}\left[ \frac{1}{2} \int_0^T |u(X^u_s, s)|^2 \, \mathrm ds + \int_0^T f(X^u_s, s)\mathrm ds  + g(X^u_T)  \right] + \log \mathcal{Z}.
\end{align}
\end{subequations}
\end{proof}
    
\begin{proof}[Proof of Proposition \ref{prop:cross entropy}]
Similarly, we compute
\begin{subequations}
\begin{align}
\mathcal{L}_{\text{CE}}(u) &= {\E}_{\Q}\left[\log \frac{\mathrm d \Q}{\mathrm d \P^u} \right] 
= {\E}_{\P^v}\left[\log\left( \frac{\mathrm d \Q}{\mathrm d \P}\frac{\mathrm d \P}{\mathrm d \P^u}\right) \frac{\mathrm d \Q}{\mathrm d \P} \frac{\mathrm d \P}{\mathrm d \P^v} \right] \\
\begin{split}
&= {\E}\Bigg[\left(\frac{1}{2} \int_0^T |u(X^v_s, s)|^2 \, \mathrm ds - \int_0^T (u \cdot v)(X^v, s)\, \mathrm ds - \int_0^T u(X_s^v, s) \cdot  \,\mathrm dW_s  -\mathcal{W}(X^v) -\log\mathcal{Z} \right) \\
& \qquad\qquad \frac{1}{\mathcal{Z}} \exp\left(-\mathcal{W}(X^v) -\int_0^T v(X_s^v, s) \cdot  \,\mathrm dW_s  -\frac{1}{2} \int_0^T |v(X^v_s, s)|^2 \, \mathrm ds \right)\Bigg] \end{split} \\
\begin{split}
&= \frac{1}{\mathcal{Z}} {\E}\Bigg[\left(\frac{1}{2} \int_0^T |u(X^v_s, s)|^2 \, \mathrm ds - \int_0^T (u \cdot v)(X_s^v, s)\mathrm ds - \int_0^T u(X_s^v, s) \cdot  \,\mathrm dW_s \right) \\
& \qquad\qquad  \exp\left(-\int_0^T v(X_s^v, s) \cdot  \,\mathrm dW_s  -\frac{1}{2} \int_0^T |v(X^v_s, s)|^2 \, \mathrm ds-\mathcal{W}(X^v)  \right)\Bigg] + C\end{split} ,
\end{align}
\end{subequations}
where $C \in \R$ does not depend on $u$.
\end{proof}
    
\begin{proof}[Proof of Proposition \ref{prop:variance}]
With $\widetilde{Y}_T^{u,v}$ defined as in \eqref{def_Y_T}, we compute for the variance loss
\begin{align}
\mathcal{L}_{\text{Var}_v}(u) &= {\Var}_{\P^v}\left( \frac{\mathrm d \Q}{\mathrm d \P^u} \right) 
= {\Var}_{\P^v}\left( \frac{\mathrm d \Q}{\mathrm d \P}\frac{\mathrm d \P}{\mathrm d \P^u} \right) 
= \frac{1}{\mathcal{Z}^2} \,{\Var}_{\P^v}\left( e^{\widetilde{Y}_T^{u, v} - g(X_T^v)} \right).
\end{align}
Similarly, the log-variance loss equals
\begin{subequations}
\begin{align}
\mathcal{L}^{\log}_{\text{Var}_v}(u) &= {\Var}_{\P^v}\left( \log\frac{\mathrm d \Q}{ \mathrm d \P^u} \right) 
= {\Var}_{\P^v}\left( \log\left(\frac{\mathrm d \P}{\mathrm d \P^u}\frac{\mathrm d \Q}{\mathrm d \P} \right) \right) 
= {\Var}_{\P^v}\left( \widetilde{Y}_T^{u, v} - g(X_T^v) - \log\mathcal{Z} \right)\\
&= {\Var}_{\P^v}\left(\widetilde{Y}_T^{u, v} - g(X_T^v) \right).
\end{align}
\end{subequations}
\end{proof}

\subsection{Proofs for Section \ref{sec: infinite batch size}}
\label{app:infinite batch size}

\begin{proof}[Proof of Proposition \ref{prop:log var and KL}]
For $\varepsilon \in \mathbb{R}$ and $\phi \in C_b^1(\mathbb{R}^d \times [0,T] ; \mathbb{R}^d)$, let us define the change of measure
\begin{equation}
\Lambda_T(\varepsilon,\phi) = \exp \left( -\varepsilon \int_0^T \phi(X^u_s,s) \cdot \mathrm{d}W_s - \frac{\varepsilon^2}{2} \int_0^T \vert \phi(X^u_s,s)\vert^2 \, \mathrm{d}s\right), \qquad \frac{\mathrm{d}\widetilde{\Theta}}{\mathrm{d}\Theta} = \Lambda_T(\varepsilon,\phi).
\end{equation}
According to Girsanov's theorem, the process $(\widetilde{W}_s)_{0 \le s \le T}$, defined as
\begin{equation}
\widetilde{W}_t = W_t + \varepsilon \int_0^t \phi(X^u_s,s) \, \mathrm{d}s,
\end{equation}
is a Brownian motion under $\widetilde{\Theta}$. We therefore obtain
\begin{equation}\mathcal{L}_{\mathrm{RE}}(u + \varepsilon \phi) = \mathbb{E} \left[ \left( \frac{1}{2} \int_0^T \vert (u + \varepsilon \phi) (X_s^u,s) \vert^2 \, \mathrm{d}s + \int_0^T f(X_s^u, s)\, \mathrm ds + g(X_T^u)\right) \Lambda^{-1}_T(\varepsilon,\phi) \right] + \log \mathcal{Z}.
\end{equation}
Using dominated convergence, we can interchange derivatives and integrals (for technical details, we refer to \cite{lie2016convexity}) and compute
\begin{subequations}
\label{eq:variation RE}
\nonumber
\begin{align}
\frac{\mathrm{d}}{\mathrm{d} \varepsilon} \Big\vert_{\varepsilon = 0} \mathcal{L}_{\RE} (u + \varepsilon \phi) & = \mathbb{E} \left[\int_0^T (u\cdot \phi)(X_s^u,s) \, \mathrm{d}s + \left(\frac{1}{2} \int_0^T \vert u (X_s^u,s) \vert^2 \, \mathrm{d}s + \int_0^T f(X_s^u, s)\, \mathrm ds + g(X_T^u)\right) \int_0^T \phi(X_s^u,s) \cdot \mathrm{d}W_s \right] 
\\
\label{eqn: derivative RE loss}
& = {\E}\left[\left(g(X_T^u) - \widetilde{Y}_T^{u, u}\right) \int_0^T \phi(X_s^u,s)\cdot \mathrm dW_s \right],
\tag{\ref{eq:variation RE}}
\end{align}
\end{subequations}
where we have used It{\^o}'s isometry,
\begin{equation}
\mathbb{E} \left[ \int_0^T \phi(X_s^u,s)\cdot \mathrm dW_s \int_0^T u(X_s^u,s)\cdot \mathrm dW_s\right] = \mathbb{E} \left[\int_0^T (u\cdot \phi)(X_s^u,s) \, \mathrm{d}s \right].
\end{equation}
Turning to the log-variance loss,
we see that
\begin{subequations}
\label{eq:calculation variance loss}
\begin{align}
& \frac{\mathrm{d}}{\mathrm{d}\varepsilon} \Big\vert_{\varepsilon = 0}
\mathcal{L}^{\log}_{\Var_v}(u + \varepsilon \phi) = \frac{\mathrm{d}}{\mathrm{d}\varepsilon} \Big\vert_{\varepsilon = 0} \left( \mathbb{E} \left[ \left( \widetilde{Y}_T^{u+\varepsilon \phi, v} - g(X_T^v) \right)^2\right] -  \mathbb{E} \left[ \left( \widetilde{Y}_T^{u+\varepsilon \phi, v} - g(X_T^v) \right)\right]^2 \right)
\\
\label{eqn: derivative_log_variance}
= & 2\, \mathbb{E} \left[ \left( \widetilde{Y}_T^{u, v} - g(X_T^v) \right)\frac{\mathrm{d}}{\mathrm{d}\varepsilon}\Big\vert_{\varepsilon = 0}\widetilde{Y}_T^{u+\varepsilon \phi, v}\right] - 2\, \mathbb{E} \left[ \left( \widetilde{Y}_T^{u, v} - g(X_T^v) \right)\right] \mathbb{E}\left[\frac{\mathrm{d}}{\mathrm{d}\varepsilon}\Big\vert_{\varepsilon = 0}\widetilde{Y}_T^{u+\varepsilon \phi, v}\right], 
\end{align}
\end{subequations}
where
\begin{equation}
\frac{\mathrm{d}}{\mathrm{d}\varepsilon}\Big\vert_{\varepsilon = 0}\widetilde{Y}_T^{u+\varepsilon \phi, v} = \int_0^T (\phi\cdot (u-v))(X_s^v,s) \, \mathrm{d}s - \int_0^T \phi(X_s^v,s) \cdot \mathrm{d}W_s.
\end{equation}
Setting $v=u$, we obtain
\begin{equation}
\label{gradient_of_log_variance}
\left( \frac{\mathrm{d}}{\mathrm{d}\varepsilon} \Big\vert_{\varepsilon = 0}
\mathcal{L}^{\log}_{\Var_v}(u + \varepsilon \phi)\right)\Big\vert_{v=u} =2 \,{\E}\left[\left(g(X_T^u) - \widetilde{Y}_T^{u, u}\right) \int_0^T \phi(X_s^u,s)\cdot \mathrm dW_s \right],
\end{equation}
from which the result follows by comparison with \eqref{eq:variation RE}.
\end{proof}

\begin{proof}[Proof of Proposition \ref{prop: equivalence moment log-variance}]
We compute
\begin{equation}
\frac{\mathrm{d}}{\mathrm{d} \varepsilon} \Big\vert_{\varepsilon = 0} \mathcal{L}_{\text{moment}_v} (u + \varepsilon \phi) = 2 \, \mathbb{E} \left[\left( \widetilde{Y}_T^{u,v} + y_0 - g(X_T^v) \right) \left(\int_0^T (\phi\cdot (u-v))(X_s^v,s) \, \mathrm{d}s - \int_0^T \phi(X_s^v,s) \cdot \mathrm{d}W_s \right)  \right].
\end{equation}
Setting $v = u$ and using that ${\E}\left[y_0 \int_0^T \phi(X_s^v,s)\cdot \mathrm{d}W_s \right] = 0$, the first statement follows by comparison with \eqref{eq:exact gradient}.
The second statement follows from 
\begin{align}
\left(\frac{\delta }{\delta u}\mathcal{L}_{\mathrm{moment}_v}(u, y_0;\phi)\right)\Big|_{u=u^*} 
= 2 \, \mathbb{E} \left[\left(  y_0 + \log \mathcal{Z}  \right) \left(\int_0^T (\phi\cdot (u^*-v))(X_s^v,s) \, \mathrm{d}s  \right)  \right],
\end{align}
where we have used the fact that $\widetilde{Y}_T^{u^*,v} - g(X_T^v) = \log \mathcal{Z}$, almost surely.
\end{proof}

\subsection{Proofs for Section \ref{sec:finite sample properties}}
\label{app:products}

\begin{proof}[Proof of Proposition
\ref{prop:robustness at u}]
\ref{it:u log variance}.)
We compute
\begin{subequations}
\begin{align}
& \frac{\delta}{\delta u}\Big|_{u = u^*} \widehat{\mathcal{L}}^{(N)}_{\Var_v}(u;\phi) = 2 \, \Bigg(\frac{1}{N}\sum_{i=1}^N \left[ \exp\left(2\left(\widetilde{Y}_T^{u^*,v,(i)} - g\left(X_T^{v, (i)}\right)\right)\right)\frac{ \delta \widetilde{Y}_T^{u,v,(i)}}{\delta u} (u^*;\phi)\right] \\
& -  \frac{1}{N}\sum_{i=1}^N\left[\exp\left(\widetilde{Y}_T^{u^*,v,(i)} - g\left(X_T^{v, (i)}\right)\right) \frac{\delta\widetilde{Y}_T^{u,v,(i)}}{\delta u} (u^*;\phi) \right]\frac{1}{N}\sum_{i=1}^N\left[\exp\left(\widetilde{Y}_T^{u^*,v,(i)} - g\left(X_T^{v, (i)}\right)\right)\right]\Bigg),
\end{align}
\end{subequations}
where $\frac{ \delta \widetilde{Y}_T^{u,v,(i)}}{\delta u} (u;\phi)$ is given in \eqref{eq:dYdu}.
As in the proof for the log-variance estimator, the quantity
\begin{equation}
\exp\left(\widetilde{Y}_T^{u^*,v,(i)} - g\left(X_T^{v, (i)}\right)\right)
\end{equation}
is almost surely constant and thus the statement follows.\par\bigskip

\ref{it:u moment}.)
Similarly to the computations involved in \ref{it:u log variance}.) we have
\begin{subequations}
\begin{align}
& \frac{\delta}{\delta u}\Big|_{u=u^*}\widehat{\mathcal{L}}^{(N)}_{\operatorname{moment}_v}(u, y_0; \phi) =  \frac{2}{N}\sum_{i=1}^N \left(\widetilde{Y}_T^{u^*,v, (i)} + y_0 - g\left(X_T^{u^*, (i)}\right) \right)\frac{ \delta \widetilde{Y}_T^{u,v,(i)}}{\delta u} (u^*;\phi)  \\
&= \frac{2}{N} \left(- \log\mathcal{Z} + y_0  \right)  \sum_{i=1}^N \left(  \int_0^T \phi(X_s^{v,(i)},s) \cdot \mathrm{d}W^{(i)}_s - \int_0^T \left(\phi \cdot (u^* - v) \right)(X_s^{v, (i)}, s) \, \mathrm ds  \right),
\end{align}
\end{subequations}
where we have used the fact that $\widetilde{Y}_T^{u^*,v, (i)} - g\left(X_T^{u^*, (i)}\right) = - \log \mathcal{Z}$ according to \eqref{eq:equivalent quantities} and \eqref{eq:general FBSDE backward}. 
The variance of this expression equals
\begin{equation}
\label{eqn: variance gradient moment}     
\frac{4}{N} \left(\log\mathcal{Z} -  y_0 \right)^2 \E \left[  \left(  \int_0^T \phi(X_s^{v,(i)},s) \cdot \mathrm{d}W^{(i)}_s - \int_0^T \left(\phi \cdot (u^* - v) \right)(X_s^{v, (i)}, s) \, \mathrm ds  \right)^2 \right],
\end{equation}
implying the claim. \par\bigskip

\ref{it:u relative entropy}.)
Let $\phi \in C_b^1(\mathbb{R}^d \times [0,T] ; \mathbb{R}^d)$ and $\varepsilon \in \mathbb{R}$. As usual, we denote by $(X_s^{u^* + \varepsilon \phi})_{0 \le s \le T}$ the unique strong solution to \eqref{eq:controlled SDE}, with $u$ replaced by $u^* + \varepsilon \phi$. By a slight modification of \cite[Theorems 3.1 and 3.3]{kunita1984stochastic} detailed, for instance, in \cite[Section 10.2.2]{pages2018numerical}, $X_s^{u^* + \varepsilon \phi}$ is almost surely differentiable as a function of $\varepsilon$. Furthermore,  $\frac{\mathrm{d}X_s^{u^* + \varepsilon \phi}}{ \mathrm{d}\varepsilon} \Big \vert_{\varepsilon = 0} =: A_s$ satisfies the SDE \eqref{eq:A equation}.
We calculate
\begin{subequations}
\label{eq:pathwise derivative}
\begin{align}
& \frac{\mathrm{d}}{\mathrm{d}\varepsilon} \Big\vert_{\varepsilon = 0} \left[\frac{1}{2} \int_0^T \vert u^* + \varepsilon \phi \vert^2(X_s^{u^* + \varepsilon \phi},s) \, \mathrm{d}s + \int_0^T f(X_s^{u^* + \varepsilon \phi},s) \, \mathrm{d}s +  g(X_T^{u^* + \varepsilon \phi})\right] \\
\label{eq:path derivative}
& = \int_0^T (u^* \cdot \phi)(X_s^{u^*},s) \, \mathrm{d}s +\frac{1}{2} \int_0^T (\nabla \vert u^* \vert^2)(X_s^{u^*},s) \cdot  A_s \, \mathrm{d}s + \int_0^T \nabla f (X_s^{u^*},s)\cdot A_s \, \mathrm{d}s + \nabla g(X_T^{u^*}) \cdot A_T. 
\end{align}
\end{subequations}
From \eqref{eq:HJB final condition} and using integration by parts, we see that the last term in \eqref{eq:path derivative} satisfies
\begin{equation}
\label{eq:int by parts}
(\nabla g)(X_T^{u^*}) \cdot A_T  = \nabla V(X_T^{u^*},T) \cdot A_T =  \int_0^T \nabla V(X^{u^*}_s,s) \cdot \mathrm{d}A_s +   \int_0^T A_s \cdot \mathrm{d} (\nabla V (X^{u^*}_s,s))
+ \left\langle A_{\cdot},\nabla V(X_{\cdot}^{u^*},\cdot)\right\rangle_T.
\end{equation}
Next, we employ It{\^o}'s formula and Einstein's summation convention to compute
\begin{subequations}
\label{eq:Ito formula V}
\begin{align}
&\mathrm{d} (\partial_{x_i} V (X^{u^*}_s,s))=
\\
& =  \left[\partial_{x_i} \partial_s V + (\partial_{x_i} \partial_{x_j} V) (b+\sigma u^*)_j + \frac{1}{2} (\partial_{x_i}\partial_{x_j} \partial_{x_k} V) \sigma_{jl}\sigma_{kl}  \right](X_s^{u^*},s) \, \mathrm{d}s 
+ \left[(\partial_{x_i} \partial_{x_j}V)\sigma_{jk}\right](X_s^{u^*},s) \, \mathrm{d}W_s^k
\\
& = \partial_{x_i} \left[\partial_s V + LV - \frac{1}{2} (\partial_{x_j}V) \sigma_{jk} \sigma_{lk} (\partial_{x_l}V) \right](X_s^{u^*},s) \,\mathrm{d}s 
+  \left[(\partial_{x_i} \partial_{x_j}V)\sigma_{jk}\right](X_s^{u^*},s) \, \mathrm{d}W_s^k
\\
& + \left[\frac{1}{2} \left((\partial_{x_j}V)(\partial_{x_l}V)  -\partial_{x_j} \partial_{x_l}V\right) \partial_{x_i}(\sigma_{jk} \sigma_{lk}) - (\partial_{x_j}V) \partial_{x_i}b_j \right] (X_s^{u^*},s)\, \mathrm{d}s
\\
& = \left[\frac{1}{2} \left((\partial_{x_j}V)(\partial_{x_l}V)  -\partial_{x_j} \partial_{x_l}V\right) \partial_{x_i}(\sigma_{jk} \sigma_{lk}) - (\partial_{x_j}V) \partial_{x_i}b_j - \partial_{x_i}f\right] (X_s^{u^*},s)\, \mathrm{d}s
\\
&+  \left[(\partial_{x_i} \partial_{x_j}V)\sigma_{jk}\right](X_s^{u^*},s) \, \mathrm{d}W_s^k,
\end{align}
\end{subequations}
where we used \eqref{eq:u_star} from the second to the third line and \eqref{eq:HJB} to manipulate the first term in the third line.
Using \eqref{eq:A equation} and \eqref{eq:Ito formula V}, we see that the quadratic variation process satisfies
\begin{equation}
\label{eq:quadratic variation}
\left\langle A_{\cdot},\nabla V(X_{\cdot}^{u^*},\cdot)\right\rangle_T = \frac{1}{2} \int_0^T A_j \left[ \partial_{x_j} (\sigma_{ik}\sigma_{lk})(\partial_{x_i} \partial_{x_l}V)\right](X_s^{u^*},s)\, \mathrm{d}s.
\end{equation}
Combining \eqref{eq:A equation}, \eqref{eq:int by parts}, \eqref{eq:Ito formula V} and \eqref{eq:quadratic variation}, it follows that \eqref{eq:pathwise derivative} equals
\begin{equation}
\int_0^T \left[A_j (\partial_{x_i}V) \partial_{x_j} \sigma_{ik} + A_j (\partial_{x_i}\partial_{x_j}V)  \sigma_{ik}\right](X_s^{u^*},s) \, \mathrm{d}W^k_s = -\int_0^T A_s \cdot (\nabla u^*)(X_s^{u^*},s) \, \mathrm{d}W_s.
\end{equation}
The claim is now implied by It{\^o}'s isometry. \par\bigskip

\ref{it:u CE}.)
With the definition of the cross-entropy loss estimator as in \eqref{eq:CE estimator} we compute
\begin{subequations}
\label{eqn: grad_CE_v}
\begin{align}
\frac{\delta}{\delta u}\Big|_{u=u^*}\widehat{\mathcal{L}}_{\CE, v}(u; \phi) & = \frac{1}{N} \sum_{i=1}^N \Bigg[  \left( \int_0^T (\phi\cdot (u^*-v))(X_s^{v, (i)},s)\,\mathrm{d}s  - \int_0^T \phi(X_s^{v, (i)}, s) \cdot \mathrm{d}W_s^{(i)}
\right) \\
& \exp \left(- \int_0^T v(X_s^{v, (i)}, s) \cdot \mathrm{d}W_s^{(i)} - \frac{1}{2} \int_0^T \vert v(X_s^{v, (i)}, s) \vert^2 \, \mathrm{d}s - \mathcal{W}(X^{v, (i)}) \right) \Bigg].
\end{align}
\end{subequations}
Since $\E\left[\frac{\delta}{\delta u}\Big|_{u=u^*}\widehat{\mathcal{L}}_{\CE, v}(u; \phi) \right]= 0$ by construction, we see that
\begin{subequations}
\begin{align}
\Var\left(\frac{\delta}{\delta u}\Big|_{u=u^*}\widehat{\mathcal{L}}_{\CE, v}(u; \phi) \right)  & = \frac{1}{N} \E \Bigg[ \left( \int_0^T (\phi\cdot (u^*-v))(X_s^{v},s)\,\mathrm{d}s  - \int_0^T \phi(X_s^{v}, s) \cdot \mathrm{d}W_s
\right)^2 \\
& \exp \left(- 2\int_0^T v(X_s^{v}, s) \cdot \mathrm{d}W_s - \int_0^T \vert v(X_s^{v}, s) \vert^2 \, \mathrm{d}s - 2 \mathcal{W}(X^{v}) \right) \Bigg].
\end{align}
\end{subequations}
Let us assume for the sake of contradiction that $
\Var\left(\frac{\delta}{\delta u}\Big|_{u=u^*}\widehat{\mathcal{L}}_{\CE, v}(u; \phi) \right) = 0$, for all $\phi \in C_b^1(\mathbb{R}^d \times [0,T]; \mathbb{R}^d)$. It then follows that
\begin{equation}
\int_0^T (\phi\cdot (u^*-v))(X_s^{v},s)\,\mathrm{d}s = \int_0^T \phi(X_s^{v}, s) \cdot \mathrm{d}W_s,
\end{equation}
which is clearly false, in general.
\end{proof}

\begin{proof}[Proof of Proposition \ref{prop:robustness}]
Throughout the proof, we will use the notation     
\begin{equation}
\label{eq:product measures}
\mathbb{P}^M := \bigotimes_{i=1}^M \mathbb{P}_i, \qquad \mathbb{Q}^M := \bigotimes_{i=1}^M \mathbb{Q}_i, \qquad \widetilde{\mathbb{P}}^M = \bigotimes_{i=1}^M \widetilde{\mathbb{P}}_i
\end{equation}
to denote the product measures on $\bigotimes_{i=1}^M C([0,T],\mathbb{R}^d) \simeq C([0,T],\mathbb{R}^{Md})$ associated to $\mathbb{P}$, $\mathbb{Q}$ and $\widetilde{\mathbb{P}}$, where $\mathbb{P}_i$, $\mathbb{Q}_i$ and $\widetilde{\mathbb{P}}_i$ refer to identical copies. \par\bigskip

\ref{it:log var robust}.)
First note that
\begin{equation}
\label{eq:variance tensor}
D_{\widetilde{\mathbb{P}}^M}^{\Var (\log)}(\mathbb{P}^M\vert \mathbb{Q}^M) = \mathrm{Var}_{\widetilde{\mathbb{P}}^M} \left( \sum_{i=1}^M \log \left( \frac{\mathrm{d} \mathbb{Q}_i}{\mathrm{d} \mathbb{P}_i}\right) \right) = \sum_{i=1}^M \mathrm{Var}_{\widetilde{\mathbb{P}}_i} \left( \log \left( \frac{\mathrm{d} \mathbb{Q}_i}{\mathrm{d} \mathbb{P}_i}\right)\right) = M D_{\widetilde{\mathbb{P}}}^{\Var (\log)}(\mathbb{P}\vert \mathbb{Q}).
\end{equation}
The sample variance satisfies \cite{cho2005variance}
\begin{equation}
\Var\left(\widehat{D}^{\Var (\log),(N)}_{\widetilde{\mathbb{P}}^M}(\mathbb{P}^M\vert \mathbb{Q}^M)\right) = \frac{1}{N} \left( \mu_4 - \frac{N-3}{N-1}D_{\widetilde{\mathbb{P}}^M}^{\Var (\log)}(\mathbb{P}^M\vert \mathbb{Q}^M)^2 \right),
\end{equation}
where
\begin{equation}
\mu_4 = \mathbb{E}_{\widetilde{\mathbb{P}}^M} \left[ \left( \log \left( \frac{\mathrm{d}\mathbb{Q}^M}{\mathrm{d}\mathbb{P}^M}\right) - \mathbb{E}_{\widetilde{\mathbb{P}}^M}  \left[\log \left(\frac{\mathrm{d}{\mathbb{Q}^M}}{\mathrm{d}\mathbb{P}^M} \right)\right]\right)^4  \right].
\end{equation}
We calculate
\begin{subequations}
\begin{align}
\mu_4 & = \mathbb{E}_{\widetilde{\mathbb{P}}^M} \left[ \left( \sum_{i=1}^M  \left( \log \left( \frac{\mathrm{d}\mathbb{Q}_i}{\mathrm{d}\mathbb{P}_i}\right) - \mathbb{E}_{\widetilde{\mathbb{P}}_i}\left[ \log \left( \frac{\mathrm{d}\mathbb{Q}_i}{\mathrm{d}\mathbb{P}_i} \right)\right]  \right)\right)^4\right] \\
& = M  \mathbb{E}_{\mathbb{P}} \left[  \left( \log \left( \frac{\mathrm{d}\mathbb{Q}}{\mathrm{d}\mathbb{P}}\right) - \mathbb{E}_{\mathbb{P}}\left[ \log \left( \frac{\mathrm{d}\mathbb{Q}}{\mathrm{d}\mathbb{P}} \right)\right]  \right)^4\right] + 6 \begin{pmatrix}
M \\ 2 
\end{pmatrix}
\mathbb{E}_{\mathbb{P}}\left[\left( \log \left( \frac{\mathrm{d} \mathbb{Q}}{\mathrm{d}\P}\right) - \mathbb{E}_{\mathbb{P}}\left[ \log \left( \frac{\mathrm{d}\mathbb{Q}}{\mathrm{d}\mathbb{P}} \right)\right] \right)^2  \right]^2,
\end{align}
\end{subequations}
where we have used the fact that, for instance, 
\begin{equation}
\mathbb{E}_{\widetilde{\mathbb{P}}^M} \left[
\left( \log \left( \frac{\mathrm{d}\mathbb{Q}_i}{\mathrm{d}\mathbb{P}_i}\right) - \mathbb{E}_{\widetilde{\mathbb{P}}_i}\left[ \log \left( \frac{\mathrm{d}\mathbb{Q}_i}{\mathrm{d}\mathbb{P}_i} \right)\right]  \right)\left( \log \left( \frac{\mathrm{d}\mathbb{Q}_j}{\mathrm{d}\mathbb{P}_j}\right) - \mathbb{E}_{\widetilde{\mathbb{P}}_j} \left[\log \left( \frac{\mathrm{d}\mathbb{Q}_j}{\mathrm{d}\mathbb{P}_j} \right)\right]  \right)^3
\right] = 0,
\end{equation}
for $i \neq j$. Combining this with \eqref{eq:variance tensor}, it follows that $\mathrm{Var}\widehat{D}^{\Var (\log),(N)}_{\widetilde{\mathbb{P}}^M}(\mathbb{P}^M\vert \mathbb{Q}^M) = \mathcal{O}(M^2)$. The claim is then a consequence of the definition \eqref{eq:relative error}.\par\bigskip

\ref{it:relative entropy robust}.) We compute
\begin{equation}
D^{\RE}(\P^M | \Q^M) = {\E}_{\P^M}\left[\log \frac{\mathrm d \P^M}{\mathrm d \Q^M} \right] = M {\E}_{{\P}}\left[\log \frac{\mathrm d {\P}}{\mathrm d {\Q}} \right] = M D^{\RE}({\P} | {\Q}).
\end{equation}
For $\widetilde{\P} = \P$ we have
\begin{equation}
\mathrm{Var}\left(\widehat{D}^{\RE,(N)}_{\P^M}(\P^M | \Q^M)\right) = \frac{1}{N} {\Var}_{\P^M}\left( \log \frac{\mathrm d \P^M}{\mathrm d \Q^M}\right) =  \frac{1}{N}{\Var}_{\P^M}\left( \sum_{i=1}^d \log \frac{\mathrm d \P_i}{\mathrm d \Q_i}\right) =  \frac{M^2}{N}{\Var}_{\P}\left( \log \frac{\mathrm d \P}{\mathrm d \Q}\right),
\end{equation}
from which the robustness follows immediately. For $\widetilde{\P} \neq \P$, on the other hand,
\begin{equation}
\mathrm{Var}\left(\widehat{D}^{\RE,(N)}_{\widetilde{\P}^M}(\P^M | \Q^M)\right) = \frac{1}{N} {\Var}_{\widetilde{\P}^M}\left( \log \left(\frac{\mathrm d \P^M}{\mathrm d \Q^M}\right)\frac{\mathrm{d} \P^M}{\mathrm{d}\widetilde{\P}^M}\right),
\end{equation}
and the proof of the non-robustness proceeds as in \ref{it:cross entropy robust}.).\par\bigskip

\ref{it:var robust}.)
As in the proof of \ref{it:log var robust}.) we have 
\begin{equation}
\mathrm{Var}\left(\widehat{D}_{\widetilde{\mathbb{P}}^M}^{\Var,(N)}(\mathbb{P}^M\vert \mathbb{Q}^M)\right) = \frac{1}{N} \left( \mu_4 - \frac{N-3}{N-1}D_{\widetilde{\mathbb{P}}^M}^{\Var}(\mathbb{P}^M\vert \mathbb{Q}^M)^2 \right),
\end{equation}
where 
\begin{equation}
\mu_4 = \mathbb{E}_{\widetilde{\mathbb{P}}^M} \left[ \left(  \frac{\mathrm{d}\mathbb{Q}^M}{\mathrm{d}\mathbb{P}^M} - \mathbb{E}_{\widetilde{\mathbb{P}}^M}   \left[\frac{\mathrm{d}\mathbb{Q}^M}{\mathrm{d}\mathbb{P}^M} \right]\right)^4  \right],
\end{equation}
and 
\begin{align}
 D_{\widetilde{\mathbb{P}}^M}^{\Var}(\mathbb{P}^M\vert \mathbb{Q}^M) = \mathrm{Var}_{ \widetilde{\mathbb{P}}^M} \left( \frac{\mathrm d \mathbb{Q}^M }{\mathrm d\mathbb{P}^M } \right) = {\E}_{\widetilde{\P}}\left[\left(\frac{\mathrm d\Q}{\mathrm d\P}\right)^2\right]^M - {\E}_{\widetilde{\P}}\left[\frac{\mathrm d \Q}{\mathrm d \P}\right]^{2M}.
\end{align}
We can write the relative error as
\begin{equation}
\label{eq:rel error var}
r^{(N)} = \sqrt{\frac{1}{N} \left(\frac{\mu_4}{D_{\widetilde{\mathbb{P}}^M}^{\Var}(\mathbb{P}^M\vert \mathbb{Q}^M)^2} - \frac{N-3}{N-1}\right)},
\end{equation}
and estimate
\begin{subequations}
\label{eq:var inequalities}
\begin{align}
& \frac{\mu_4}{D_{\widetilde{\mathbb{P}}^M}^{\Var}(\mathbb{P}^M\vert \mathbb{Q}^M)^2} \ge \frac{\mathbb{E}_{\widetilde{\mathbb{P}}^{M}} \left[ \left(  \frac{\mathrm{d}\mathbb{Q}^M}{\mathrm{d}\mathbb{P}^M} - \mathbb{E}_{\widetilde{\mathbb{P}}^M}   \left[\frac{\mathrm{d}\mathbb{Q}^M}{\mathrm{d}\mathbb{P}^M} \right]\right)^4  \right]}{{\E}_{\widetilde{\P}}\left[\left(\frac{\mathrm d \Q}{\mathrm d \P}\right)^2\right]^{2M}} \ge \frac{\frac{1}{8} \mathbb{E}_{\widetilde{\mathbb{P}}^M} \left[ \left(\frac{\mathrm{d}\mathbb{Q}^M}{\mathrm{d}\mathbb{P}^M}  \right)^4\right] - \mathbb{E}_{\widetilde{\P}^M} \left[\frac{\mathrm{d}\mathbb{Q}^M}{\mathrm{d}\mathbb{P}^M} \right]^4}{{\E}_{\widetilde{\P}}\left[\left(\frac{\mathrm d \Q}{\mathrm d \P}\right)^2\right]^{2M}}
\\
& = \frac{\frac{1}{8} \mathbb{E}_{\widetilde{\mathbb{P}}} \left[ \left(\frac{\mathrm{d}\mathbb{Q}}{\mathrm{d}\mathbb{P}}  \right)^4\right]^M - \mathbb{E}_{\widetilde{\P}} \left[\frac{\mathrm{d}\mathbb{Q}}{\mathrm{d}\mathbb{P}} \right]^{4M}}{{\E}_{\widetilde{\P}}\left[\left(\frac{\mathrm d \Q}{\mathrm d \P}\right)^2\right]^{2M}} = \frac{1}{8} \left( \underbrace{\frac{\mathbb{E}_{\widetilde{\mathbb{P}}} \left[ \left(\frac{\mathrm{d}\mathbb{Q}}{\mathrm{d}\mathbb{P}}  \right)^4\right]}{{\E}_{\widetilde{\P}}\left[\left(\frac{\mathrm d \Q}{\mathrm d \P}\right)^2\right]^2}}_{=:C_1}\right)^M - \left( \underbrace{ \frac{\mathbb{E}_{\widetilde{\mathbb{P}}} \left[ \left(\frac{\mathrm{d}\mathbb{Q}}{\mathrm{d}\mathbb{P}}  \right)\right]^4}{{\E}_{\widetilde{\P}}\left[\left(\frac{\mathrm d \Q}{\mathrm d \P}\right)^2\right]^2}}_{=:C_2}\right)^M,
\end{align}
\end{subequations}
where the second bound is implied by the $c_r$-inequality \cite[Section 9.3]{loeve1963probability}. By Jensen's inequality and since $\frac{\mathrm{d}\Q}{\mathrm{d}\P}$ is not $\widetilde{\P}$-almost surely constant by assumption, it holds that $C_1 > 1$ and $C_2 < 1$. The claim therefore follows from combining \eqref{eq:rel error var} and \eqref{eq:var inequalities}.\par\bigskip

\ref{it:cross entropy robust}.)
Employing the notation introduced in \eqref{eq:product measures}, we see that
\begin{equation}
\label{eq:KL cross entropy}
D^{\CE}(\P^M \vert \Q^M) = \mathbb{E}_{\Q^M} \left[\log \left(\frac{\mathrm{d}\Q^M}{\mathrm{d}\P^M}\right) \right] = \sum_{i=1}^M {\E}_{\Q_i} \left[ \log \left(\frac{\mathrm{d} \Q_i}{\mathrm{d} \P_i}\right) \right] = M  D^{\CE}(\P \vert \Q).
\end{equation}
Furthermore,
\begin{subequations}
\label{eq:variance cross entropy}
\begin{align}
\Var\left(\widehat{D}^{\CE,(N)}_{\widetilde{\P}^M}(\mathbb{P}^M \vert \mathbb{Q}^M )\right) & = \frac{1}{N} \mathrm{Var}_{\widetilde{\mathbb{P}}^M} \left( \log \left( \frac{\mathrm{d} \mathbb{Q}^M}{\mathrm{d}\mathbb{P}^M} \right)\frac{\mathrm{d} \mathbb{Q}^M}{\mathbb{d}\mathrm{d}\widetilde{\mathbb{P}}^M}\right)
\\
& = \frac{1}{N} \left( \mathbb{E}_{\widetilde{\mathbb{P}}^M} \left[ \log^2 \left( \frac{\mathrm{d} \mathbb{Q}^M}{\mathbb{d}\mathrm{d}\mathbb{P}^M} \right) \left( \frac{\mathrm{d} \mathbb{Q}^M}{\mathbb{d}\mathrm{d}\widetilde{\mathbb{P}}^M} \right)^2 \right] - \mathbb{E}_{\widetilde{\mathbb{P}}^M} \left[ \log \left( \frac{\mathrm{d}\mathbb{Q}^M}{\mathbb{d}\mathrm{d}\mathbb{P}^M}\right)\frac{\mathrm{d}\mathbb{Q}^M}{\mathbb{d}\mathrm{d}\widetilde{\mathbb{P}}^M}\right]^2 \right)
\\
& = \frac{1}{N} \left( \mathbb{E}_{\mathbb{Q}^M} \left[ \log^2 \left( \frac{\mathrm{d}\mathbb{Q}^M}{\mathbb{d}\mathrm{d}\mathbb{P}^M} \right) \frac{\mathrm{d}\mathbb{Q}^M}{\mathbb{d}\mathrm{d}\widetilde{\mathbb{P}}^M}  \right] - M^2 \mathbb{E}_{\mathbb{Q}} \left[ \log \left( \frac{\mathrm{d} \mathbb{Q}}{\mathbb{d}\mathrm{d}\mathbb{P}}\right)\right]^2 \right).
\end{align}
\end{subequations}
Manipulating the first term, we obtain
\begin{subequations}
\begin{align}
& \mathbb{E}_{\mathbb{Q}^M} \left[ \log^2 \left( \frac{\mathrm{d}\mathbb{Q}^M}{\mathbb{d}\mathrm{d}\mathbb{P}^M} \right) \frac{\mathrm{d}\mathbb{Q}^M}{\mathbb{d}\mathrm{d}\widetilde{\mathbb{P}}^M}  \right]  = \mathbb{E}_{\mathbb{Q}^M} \left[ \left( \sum_{i=1}^M \log \left( \frac{\mathrm{d} \mathbb{Q}_i}{\mathrm{d}\mathbb{P}_i}\right) \right)^2 \frac{\mathrm{d} \mathbb{Q}^M}{\mathbb{d}\mathrm{d}\widetilde{\mathbb{P}}^M} \right]
\\
& = \sum_{i=1}^M \mathbb{E}_{\mathbb{Q}^M} \left[ \log^2 \left( \frac{\mathrm{d} \mathbb{Q}_i}{\mathrm{d}\mathbb{P}_i}\right) \frac{\mathrm{d} \mathbb{Q}^M}{\mathbb{d}\mathrm{d}\widetilde{\mathbb{P}}^M}  \right] + \sum_{\substack{i,j=1\\ i \neq j }}^M \mathbb{E}_{\mathbb{Q}^M} \left[ \log \left(\frac{\mathrm{d} \mathbb{Q}_i}{\mathrm{d}\mathbb{P}_i} \right) \log \left(\frac{\mathrm{d} \mathbb{Q}_j}{\mathrm{d}\mathbb{P}_j}\right) \frac{\mathrm{d} \mathbb{Q}^M}{\mathbb{d}\mathrm{d}\widetilde{\mathbb{P}}^M}\right]
\\
& = M \left(\mathbb{E}_{\mathbb{Q}} \left[\frac{\mathrm{d}\mathbb{Q}}{\mathrm{d}\widetilde{\mathbb{P}}}\right] \right)^{M-1} \mathbb{E}_{\mathbb{Q}} \left[  \log^2\left( \frac{\mathrm{d}\mathbb{Q}}{\mathrm{d}\mathbb{P}}\right) \frac{\mathrm{d}\mathbb{Q}}{\mathrm{d}\widetilde{\mathbb{P}}}\right] + \frac{M(M-1)}{2} \left( \mathbb{E}_{\mathbb{Q}} \left[\log \left( \frac{\mathrm{d}\mathbb{Q}}{\mathrm{d}\mathbb{P}}\right) \frac{\mathrm{d}\mathbb{Q}}{\mathrm{d}\widetilde{\mathbb{P}}}\right]\right)^2 \left( \mathbb{E}_{\mathbb{Q}}\left[\frac{\mathrm{d}\mathbb{Q}}{\mathrm{d}\widetilde{\mathbb{P}}}\right]\right)^{M-2}.
\end{align}
\end{subequations}
Notice that 
\begin{equation}
{\E}_{\Q} \left[ \frac{\mathrm{d \Q} }{\mathrm{d \widetilde{\P}} }\right] = {\E}_{\widetilde{\P}} \left[ \left(\frac{\mathrm{d \Q} }{\mathrm{d \widetilde{\P}} }\right)^2\right] = \chi^2(\Q \vert \widetilde{\P}) + 1.
\end{equation}
The claim now follows from combining \eqref{eq:KL cross entropy} and \eqref{eq:variance cross entropy} in definition \eqref{eq:relative error}.
\end{proof}
    
\subsection{Optimal control for Ornstein-Uhlenbeck dynamics with linear cost}
\label{OU_linear_solution}
The control problem considered in Section \ref{section_LLQC} can be solved analytically. Using \eqref{eqn: free energy}, we note that the value function solving the HJB-PDE \eqref{eq:HJB} fulfills $V(x, t) = -\log \psi(x, t)$, with
\begin{equation}
\psi(x, t) = \E\left[e^{-\gamma \cdot X_T} | X_t = x\right],
\end{equation}
where $(X_s)_{t \le s \le T}$ solves
\begin{equation}
\mathrm dX_s = AX_s \, \mathrm d s + B \,  \mathrm d W_s, \quad X_t = x.
\end{equation}
The distribution of $X_T$ is known explicitly, namely
\begin{equation}
(X_T|X_t = x) \sim \mathcal{N}\left( \mu_t, \Sigma_t \right)
\end{equation}
with
\begin{equation}
\mu_t = e^{A(T-t)}x, \qquad \Sigma_t = \int_0^{T-t} e^{As} B B^\top e^{A^\top s}\, \mathrm ds.
\end{equation}
We can now compute
\begin{equation}
\psi(x, t) = \exp\left(-\gamma \cdot \left(\mu_t - \frac{1}{2} \Sigma_t\gamma\right)\right),
\end{equation}
and the value function
\begin{equation}
V(x, t) = \gamma \cdot \left(\mu_t - \frac{1}{2} \Sigma_t \gamma\right),
\end{equation}
and therefore with \eqref{eq:connections u star} we obtain
\begin{equation}
u^*(x, t) = -B^\top \nabla V(x, t) = -B^\top e^{A^\top (T-t)}\gamma.
\end{equation}

\bibliographystyle{abbrv}
\bibliography{main}
    
\end{document}